\documentclass[preprint,3p,times,numbers]{elsarticle}

\usepackage{amssymb}
\usepackage{amsmath}
\usepackage{enumitem}
\setlist[itemize]{noitemsep,nolistsep}
\setlist[enumerate]{noitemsep,nolistsep}
\usepackage{amsthm}
\usepackage[normalem]{ulem}
\usepackage[utf8]{inputenc}
\usepackage[T1]{fontenc}
\usepackage{tikz-cd}
\usepackage{adjustbox}
\usepackage[nottoc]{tocbibind}

\usepackage{mathtools}
\usepackage{dsfont}
\usepackage{mathrsfs}
\usepackage{csquotes}

\usepackage{hyperref} 
\hypersetup{
  breaklinks=true,
  colorlinks=true,
  linkcolor=blue,
  urlcolor=orange!75!black,
  citecolor=green,
  hypertexnames=false,
  }
\usepackage[noabbrev,capitalize,nameinlink]{cleveref}
\usepackage{thmtools}

\newtheorem{theorem}{Theorem}[section]
\newtheorem{proposition}[theorem]{Proposition}
\newtheorem{corollary}[theorem]{Corollary}

\theoremstyle{definition}
\newtheorem{definition}[theorem]{Definition}
\newtheorem{remark}[theorem]{Remark}
\newtheorem{assumption}[theorem]{Assumption}
\newtheorem{example}[theorem]{Example}
\newtheorem{notation}[theorem]{Notation}
\newtheorem{cexample}[theorem]{Counterexample}

\newcommand*\dif{\mathop{}\!\mathrm{d}}
\newcommand{\cL}{\mathcal{L}^0}
\newcommand{\cLs}{\mathcal{L}^0_{\mathrm{s}}}

\newcommand{\overlinecL}{\overline{\mathcal{L}}{}^0}

\newcommand{\cLph}{\mathcal{L}^p_h}

\newcommand{\Lph}{L^p_h}

\DeclareMathOperator*{\esssup}{ess\,sup}
\DeclareMathOperator{\lebesgue}{\mathscr{L}^1}

\makeatletter 
\def\ps@pprintTitle{
 \let\@oddhead\@empty
 \let\@evenhead\@empty
 \def\@oddfoot{}
 \let\@evenfoot\@oddfoot}
\makeatother

\begin{document}

\begin{frontmatter}

\title{Nonlinear Lebesgue spaces: Curves and geometry}

\corref{cor}
\cortext[cor]{Corresponding author}

\affiliation[1]{organization={Université Paris Cité, CNRS, MAP5},
 postcode={F-75006},
 city={Paris},
 country={France}}

\author[1]{Guillaume Sérieys}
\ead{guillaume.serieys@u-pariscite.fr}

\begin{abstract}
This paper is the second in a series by the author and collaborators devoted to the study of geometric and analytic properties of \emph{nonlinear Lebesgue spaces}, that is,\mathversion{normal} $L^p$ spaces of mappings taking values in arbitrary metric spaces. The present article formalizes the pointwise description of their geometric properties---their length structure, bounds on their Alexandrov curvature as well as the definition of a velocity for absolutely continuous curves despite the lack of differential structure. To obtain this pointwise description, we first prove a nonlinear analogue of the Fubini--Lebesgue theorem, which yields an identification of $L^p$ curves in nonlinear Lebesgue spaces with mappings taking values in the space of $L^p$ curves. This identification of $L^p$ curves then enables a similar identification for absolutely continuous curves, from which the pointwise description of the geometric properties of nonlinear Lebesgue spaces follows.
\end{abstract}

\begin{keyword}
Nonlinear Lebesgue spaces \sep Absolutely continuous curves \sep Càdlàg curves of bounded variation \sep Length structure \sep Geodesics \sep Curvature bounds \sep Finsler manifolds
\end{keyword}

\end{frontmatter}

\tableofcontents

\section{Introduction}
\label{sec:intro}

Physical phenomena inherently involve noncontinuous signals with values in nonlinear spaces. For instance, in medical imaging the images cannot be assumed continuous due to the multiphase nature of anatomy and can be valued in non-Euclidean spaces, such as the space of symmetric positive definite matrices (as in diffusion tensor imaging, to measure the diffusion of water in tissues \cite{basser1994mr}), the probability simplex (to account for uncertainty in the identification of the nature of tissues \cite{ashburner2005unified} or to construct probabilistic atlases of organs \cite{mazziotta1995probabilistic}) or the space of probability measures on the sphere (as encountered in Q-ball imaging \cite{tuch2004q} or constrained spherical deconvolution \cite{tournier2004direct}). When working with this type of data, be it for reconstruction or statistical analysis purposes, the choice of a \emph{meaningful} metric on the space of values is important, and usually endows it with a nonlinear geometry. In addition, signals may themselves be supported on curved spaces, as in the framework of functional shapes \cite[Definition~2.1.1]{charon2013analysis} or in surface-based neuroimaging, where cortical data can be represented as signals defined on spherical parameterizations of cortical surfaces \cite{fischl1999cortical}. For all those reasons, studying merely measurable mappings defined on measurable spaces with values in arbitrary metric spaces is of importance from an applications perspective. See \cite{serieys2025nonlinear} for a more detailed discussion of applications.

From a more theoretical perspective, $L^p$ spaces of measurable mappings taking values in arbitrary metric spaces, which we call \emph{nonlinear Lebesgue spaces}, first emerged in the study of Sobolev spaces of measurable mappings taking values in arbitrary metric spaces and, in particular, of harmonic mappings, starting with the seminal work of N.~J.~Korevaar and R.~M.~Schoen \cite{korevaar1993sobolev} followed by the works of K.-T.~Sturm \cite{Sturm2001,sturm2002nonlinear} and J.~Jost \cite{jost1994equilibrium,jost1997generalized,jost1997nonpositive}. Nonlinear Lebesgue spaces typically model the space of transport maps in optimal transport theory \cite[Section~5.4]{ambrosio2005gradient}. They also appear as metric completions of certain $L^p$-type geometries. A prominent example is provided by the $L^2$ metric on the manifold of smooth Riemannian metrics, also known as the Ebin metric: its metric completion was characterized by B.~Clarke \cite{clarke2009thesis,clarke2013completion}. More recently, N.~Cavallucci gave a shorter proof of this result, exhibiting the completion explicitly as an $L^2$ space of mappings taking values in a fixed $\operatorname{CAT}(0)$ space \cite{cavallucci2023l2}. A similar completion result was obtained by N.~Cavallucci and Z.~Su for the space of full-rank vector-valued one-forms equipped with a generalized Ebin metric \cite{cavallucci2023metric}. Related recent developments include the studies by D.~Lenze of isometric rigidity of $L^2$ spaces of manifold-valued mappings \cite{lenze2025isometric} and of isometries of the Ebin metric \cite{lenze2025isometries}. Nonlinear Lebesgue spaces are also useful in the study of curves in arbitrary metric spaces such as Wasserstein spaces, as studied by S.~Lisini \cite[Section~2.3]{lisini2007characterization}, and recently appeared in a generalization, proposed by M.~Bauer, F.~Mémoli, T.~Needham and M.~Nishino \cite[Section~2.1.3]{bauer2025z}, of the Gromov--Wasserstein metric to measure spaces equipped with a kernel taking values in an arbitrary metric space instead of a standard metric. The analytical properties of these spaces derived in particular cases in the mentioned works---the characterizations of their completeness and separability, as well as the density of the spaces of simple, continuous, and smooth mappings---are gathered and generalized under a unified exposition in a recent work by the author and A.~Trouvé \cite{serieys2025nonlinear}. A recurrent feature of nonlinear Lebesgue spaces is the pointwise nature of their properties: many of their analytic and geometric properties are determined by corresponding properties of the target space, that is, the space in which the mappings take their values, and conversely. This phenomenon already appears in early results concerning completeness, geodesics and curvature \cite{korevaar1993sobolev,Sturm2001,sturm2002nonlinear,jost1994equilibrium,jost1997generalized,jost1997nonpositive,monod2006superrigidity}. In particular, K.-T.~Sturm \cite{Sturm2001,sturm2002nonlinear}, J.~Jost \cite{jost1994equilibrium,jost1997generalized,jost1997nonpositive} and N.~Monod \cite{monod2006superrigidity} derived curvature properties of nonlinear Lebesgue spaces when the target space is nonpositively curved in the sense of Alexandrov (see \cite[Chapter~4]{burago2022course} for a definition), notably through the construction of geodesics from pointwise geodesics in the target space. This suggests that the geometry of nonlinear Lebesgue spaces should admit a systematic pointwise description. Making this principle rigorous, however, requires a proper understanding of curves in these spaces. In this respect, one may expect a behaviour analogous to that of curves in Wasserstein spaces, for which S.~Lisini \cite{lisini2006absolutely,lisini2007characterization} exhibited a superposition principle for absolutely continuous curves (i.e., any such curve can be identified with a probability measure on absolutely continuous curves in the space supporting the measures). Let us conclude with a note on terminology. As highlighted in \cite[TVS V.86]{bourbaki2013topological}, $L^p$ spaces of real-valued mappings were first introduced by F.~Riesz in \cite{riesz1910untersuchungen}. Yet, they are commonly referred to as \enquote{Lebesgue} spaces \cite[Chapter III, Section 3]{dunford1988linear} or \enquote{Bochner} spaces when mappings are valued in a Banach space \cite[Section~1.2.b]{hytonen2016analysis}. For consistency with the literature, we use the terminology of \enquote{nonlinear Lebesgue spaces} to designate $L^p$ spaces of mappings taking values in arbitrary metric spaces, terminology that was first used by M.~Bačák in \cite[p.~18]{bacak2014convex}. 

In this article, we formalize the pointwise description of the geometric properties of nonlinear Lebesgue spaces resulting from a characterization of absolutely continuous (a.c.) curves, thus providing a complementary perspective to the analytic properties of nonlinear Lebesgue spaces studied in \cite{serieys2025nonlinear}. To do so, we first prove a nonlinear analogue of the Fubini--Lebesgue theorem (\cref{th:extension_embed_iterated_simple}) that allows for the characterization of $L^p$ curves in nonlinear Lebesgue spaces through the identification of a canonical isometry (\cref{th:isom_lebesgue}). Using \cref{th:isom_lebesgue}, we characterize a.c.~curves (\cref{th:charact_ac_curves}) as well as càdlàg curves of bounded variation (\cref{th:charact_bv_curves}), depending on the value of the parameter $p$ of nonlinear Lebesgue spaces. The characterization of a.c.~curves given by \cref{th:charact_ac_curves} notably allows for the identification of geodesics as mappings into the space of geodesics (\cref{th:charact_geodesics}), for the pointwise description of the length structure of nonlinear Lebesgue spaces (\cref{th:length_structure_lebesgue}), and for the definition of a notion of velocity for a.c.~curves in nonlinear Lebesgue spaces despite the lack of a differential structure (\cref{th:speed_ac_lebesgue}), which in turn yields the point-separation property of the geodesic distance induced by the $L^p$ metric on the manifold of smooth mappings and, under additional assumptions, the identification of its metric completion with the corresponding nonlinear Lebesgue space (\cref{th:lp_separates_smooth}). The particular characterization of geodesics given by \cref{th:charact_geodesics} also enables the pointwise description of the curvature, in the sense of Alexandrov, of nonlinear Lebesgue spaces (\cref{th:curvature_lebesgue}). To put it visually:

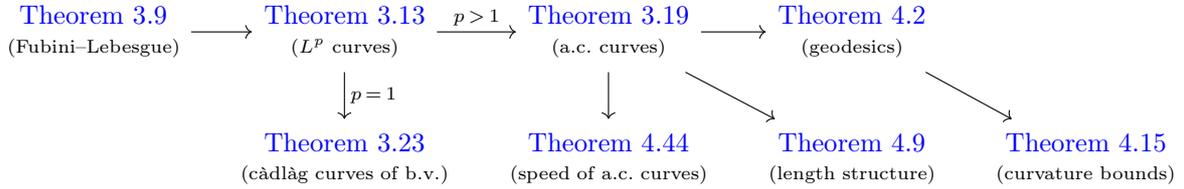
\newcommand\block[2]{\begin{tabular}{@{}c@{}}{#1}\\[-0.5mm]{\footnotesize (#2)}\end{tabular}}
\begin{center}
\adjustbox{scale=1,center}{%
\begin{tikzcd}[column sep=.425cm]
\block{\cref{th:extension_embed_iterated_simple}}{Fubini--Lebesgue}
\ar[r]
&
\block{\cref{th:isom_lebesgue}}{$L^p$ curves}
\ar[r, "p\,>\,1"]
\ar[d, "p\,=\,1"]
&
\block{\cref{th:charact_ac_curves}}{a.c.~curves}
\ar[r]
\ar[d]
\ar[dr]
&
\block{\cref{th:charact_geodesics}}{geodesics}
\ar[r]
&
\block{\cref{th:curvature_lebesgue}}{curvature bounds}
\\
&
\block{\cref{th:charact_bv_curves}}{càdlàg curves of b.v.}
&
\block{\cref{th:length_structure_lebesgue}}{length structure}
&
\block{\cref{th:speed_ac_lebesgue}}{velocity of a.c.~curves}
\ar[r]
&
\block{\cref{th:lp_separates_smooth}}{separation/completion $L^p$ metrics}
\end{tikzcd}
}
 \end{center}

\noindent The paper is structured as follows: 
\begin{itemize}
\item \cref{sec:preliminaries} introduces the basic setting of this article regarding minimal assumptions and introduces the core definitions and results needed to construct and study nonlinear Lebesgue spaces (\cref{sec:nonlinear_lebesgue_spaces}) as well as Sobolev curves and curves of bounded variation in metric spaces (\cref{sec:curves_metric_space}). \item \cref{sec:main_results} gathers the main results of this article, including the characterization of $L^p$ curves in nonlinear Lebesgue spaces (\cref{sec:embeddings_isomorphisms}) and its consequence for the characterization of a.c.~curves (\cref{sec:charact_ac}) and càdlàg curves of bounded variation (\cref{sec:charact_bv}) in these spaces. 
\item \cref{sec:applications} exposes applications of the main results established in \cref{sec:main_results}, including the consequences of \cref{sec:charact_ac} for the pointwise characterization of the geometry of nonlinear Lebesgue spaces (\cref{sec:geometry}) as well as the definition of a notion of velocity for a.c.~curves in these spaces and its consequences for the $L^p$ metric on the manifold of smooth mappings (\cref{sec:speed_ac}). 
\end{itemize}

\section{Preliminaries}
\label{sec:preliminaries}
\subsection{Nonlinear Lebesgue spaces}
\label{sec:nonlinear_lebesgue_spaces}

This section introduces notations, definitions and results that were presented in the paper \cite{serieys2025nonlinear} by the author and A.~Trouvé, only keeping the ones that will be needed in the present article. The reader already familiar with \cite{serieys2025nonlinear} may skip to \cref{sec:curves_metric_space}.

\subsubsection{Basic assumptions}
\label{sec:basic_setting}

Let us first introduce the basic assumptions of this article.

\begin{assumption}[Basic assumptions]
\label{assum:minimal}
    We consider mappings from a nonempty set $M$, called the \emph{base space}, to a nonempty set $N$, called the \emph{target space}. $M$ and $N$ are assumed to carry the following structures: 
    \begin{enumerate}[label=(\roman*)]
        \item \label{assum:base_space} $(M, \Sigma_M,\mu_M)$ is a measure space, that is, the set $M$ is paired with a $\sigma$-algebra $\Sigma_M$ on $M$ and a measure $\mu_M: \Sigma_M \to [0,\infty]$ on $(M,\Sigma_M)$.
    \item \label{assum:target_space} $(N,d_N)$ is a metric space with finite metric $d_N: N^2 \to [0,\infty)$.
    \end{enumerate}
\end{assumption}

These basic assumptions will not be repeated in the statements of subsequent results; yet, any additional assumption relative to \cref{assum:minimal} will be specified.

Notation and terminology are introduced progressively and collected for reference in \cref{sec:notations}.

\subsubsection{Measurable mappings}
\label{sec:measurable_mappings}

In this section, we introduce several classes of mappings that will serve as core ingredients in the construction of nonlinear Lebesgue spaces, namely, the set of measurable mappings and its variant.

\begin{definition}[Measurable mappings]
\label{def:measurable_map}
Define:
\begin{enumerate}[label=(\roman*)]
    \item $\cL(M,N)$ the set of \emph{measurable mappings}, that is, all mappings $f: M\to N$ such that for every set $B$ in $\mathcal{B}(N)$, the Borel $\sigma$-algebra of $N$, it satisfies $f^{-1}(B)\in \Sigma_M$.
    \item $\cLs(M,N)$ the set of \emph{separably valued measurable mappings}, that is, all mappings $f\in \cL(M,N)$ such that its range, denoted by $f(M)$, is separable. 
\end{enumerate}

\end{definition}

\begin{remark}[\enquote{Borel measurable} mappings]
When $M$ is a topological space with $\Sigma_M = \mathcal{B}(M)$, the Borel $\sigma$-algebra of $M$, a mapping belonging to $\cL(M,N)$ is usually called \enquote{Borel measurable}.
\end{remark}

A first well-known result on this class of mappings is that the pointwise limit of measurable and separably valued mappings, when it exists, is itself measurable and separably valued \cite[Proposition 2.6]{serieys2025nonlinear}. 

\begin{proposition}[$\cLs(M,N)$ is closed under pointwise limit]
\label{prop:measurability_pointwise_limit}
    Let $(f_n)_{n\in\mathbb{N}}$ be a sequence in $\cLs(M,N)$ and assume that $f(x)\coloneqq \lim_{n\to\infty} f_n(x)$ exists in $N$ for all $x\in M$. Then, $f\in \cLs(M,N)$.
\end{proposition}

In particular, \cref{prop:measurability_pointwise_limit} holds for sequences of \emph{simple mappings}, that is, measurable mappings with a finite range.

\begin{definition}[Simple mappings]
    \label{def:simple_map}
    Define $\mathcal{E}(M,N)$ the set of \emph{simple mappings}, that is, all mappings $f\in \cL(M,N)$ such that $\lvert f(M)\rvert$, the cardinality of the range of $f$, is finite. 
\end{definition}

Showing that a mapping belongs to $\cLs(M,N)$ thus usually comes down to verifying that it is the pointwise limit of a sequence of simple mappings.

Now, one can define an equivalence relation on the set of measurable mappings.
\begin{proposition}[Equivalence relation on $\cL(M,N)$]
\label{prop:equivalence}
    Define the relation $\sim_M$ on $\cL(M,N)$ by  $$f\sim_M f' \iff f(x)=f'(x)\quad\text{for $\mu_M$-a.e.~$x\in M$.}$$ Then, $\sim_M$ is an equivalence relation on $\cL(M,N)$.
\end{proposition}

Using this equivalence relation, one can define the sets of equivalence classes of measurable and simple mappings.

\begin{definition}[Equivalence classes of measurable mappings]
\label{def:equiv_class_measurable}
Let
    $$q_M: \overlinecL(M,N)\to \overlinecL(M,N)/\sim_M$$ 
    be the quotient mapping associated to $\sim_M$, given by 
    $$q_M(f)\coloneqq [f]_M, \quad [f]_M\coloneqq \{ f'\in \overlinecL(M,N) : f\sim_M f'\}.$$
    Then, define:
    \begin{enumerate}[label=(\roman*)]
        \item $L^0(M,N) \coloneqq q_M(\cLs(M,N))$ the set of equivalence classes admitting a separably valued measurable representative.
        \item $E(M,N) \coloneqq q_M(\mathcal{E}(M,N))$ the set of equivalence classes admitting a simple representative.
    \end{enumerate}
\end{definition}

From these first building blocks, we move on to the construction of nonlinear Lebesgue spaces.

\subsubsection{Lebesgue mappings}
\label{sec:lebesgue_mappings}

Before defining nonlinear Lebesgue spaces, let us first define the $\mathcal{L}^p$ semi-metrics (see \cite[Definition~1.1.4]{burago2022course} for a definition of the notion of semi-metric).

\begin{definition}[$\mathcal{L}^p$ semi-metrics]
\label{def:lp_metrics}
Let $p\in [1,\infty]$. Then, define the mapping $D_p: \cLs(M,N)^2\to [0,\infty]$ as 
$D_p(f,f') \coloneqq \lVert d_N(f,f')\rVert_{p,\mu_M}$, where $d_N(f,f')(x)\coloneqq d_N(f(x),f'(x))$. Precisely, for $p\in [1,\infty)$,
$$D_p(f, f') = \left(\int_M d_N(f(x), f'(x))^p\dif \mu_M(x)\right)^{1/p}$$
and, for $p=\infty$, 
$$D_\infty(f,f') = \mu_M\text{-}\esssup_{x\in M} d_N(f(x),f'(x))$$
with the $\mu_M$-essential supremum defined as
$$\mu_M\text{-}\esssup_{x\in M} d_N(f(x),f'(x))\coloneqq \inf\left\{C\in\mathbb{R}_+: d_N(f(x),f'(x)) \leq C \text{ for $\mu_M$-a.e. $x\in M$}\right\}.$$
\end{definition}

At this point, $D_p$ satisfies the symmetry and triangle inequality axioms of a metric on $\cLs(M,N)$ thanks to the symmetry of $d_N$ and the fact that both $d_N$ and $\lVert \cdot \rVert_{p,\mu_M}$ satisfy the triangle inequality for $p\in [1,\infty]$. However, $D_p$ is not necessarily finite and cannot distinguish two distinct measurable mappings that agree $\mu_M$-almost everywhere. In fact, $D_p$ separates equivalence classes in $L^0(M,N)$. 

\begin{proposition}[$\mathcal{L}^p$ semi-metrics separate equivalence classes]
\label{prop:lp_separates_equiv_class}
    Let $p\in [1,\infty]$ and $(f,f')\in \cLs(M,N)$. Then, $D_p(f,f')= 0$ if and only if $f\sim_M f'$.
\end{proposition}

\begin{remark}[$\mathcal{L}^p$ semi-metrics metrize $L^0(M,N)$]
    Note that $L^0(M,N)$ equipped with $D_p$, for any choice of $p\in [1,\infty]$, satisfies all the axioms of a metric space, but with a metric $D_p$ that might be infinite.
\end{remark}

First, to make $D_p$ finite, consider its restriction to \emph{$\mathcal{L}^p$ mappings}, which are defined as mappings at a finite $D_p$-distance from a \emph{base mapping} $h$ in $L^0(M,N)$.

\begin{definition}[$\mathcal{L}^p$ spaces] Let $h\in L^0(M,N)$ (recall that $h$ has a representative with separable range by \cref{def:equiv_class_measurable}) and $p\in [1,\infty]$. Then, define the set of measurable mappings at a finite $D_p$-distance from $h$ as
$$\cLph(M,N)\coloneqq\{ f\in \cLs(M,N): D_p(f, h) < \infty\}.$$ 
    
\end{definition}

\begin{remark}[Choice of the base mapping in the linear case]
    When $N$ is a normed vector space, the base mapping $h$ is usually taken as $h\equiv 0_N$, with $0_N$ the identity element of $N$. 
\end{remark}

Then, $(\cLph(M,N),D_p)$ satisfies all axioms of a metric space, except separation: it is a \emph{semi-metric} space (see \cite[Definition~1.1.4]{burago2022course} for a definition). We will often write $\cLph(M,N)$ instead of $(\cLph(M,N),D_p)$.

Now, to make it a metric space, $\mu_M$-a.e.~identical mappings should be identified through the equivalence relation defined in \cref{prop:equivalence}.

\begin{definition}[Nonlinear Lebesgue spaces]
\label{def:lebesgue_spaces}
Let $h\in L^0(M,N)$ and $p\in [1,\infty]$. Then, the nonlinear $p$-Lebesgue space $L_{h}^p(M,N)$ is defined as the quotient space 
$$L_{h}^p(M,N) \coloneqq q_M(\cLph(M,N)),$$
where we recall that $f\sim_M f'$ if and only if $D_p(f,f')=0$ (\cref{prop:lp_separates_equiv_class}).
Also, we will usually omit $p$ when referring to nonlinear $p$-Lebesgue spaces and call \enquote{Lebesgue mappings} the elements of nonlinear Lebesgue spaces. 
\end{definition}
\begin{remark}[Extension to $\mu_M$-measurable mappings]
\label{rem:extension_mu_measurable_mappings}
    Nonlinear Lebesgue spaces can equivalently be defined using $\mu_M$-measurable mappings, that is, mappings $f$ satisfying for all $B$ in $\mathcal{B}(N)$ that $f^{-1}(B)$ belongs to $\overline{\Sigma}_M$, the completion of the $\sigma$-algebra $\Sigma_M$ under $\mu_M$ \cite[Proposition 2.22]{serieys2025nonlinear}.
\end{remark}

The construction of nonlinear Lebesgue spaces developed above eventually yields that they are metric spaces \cite[Proposition 2.25]{serieys2025nonlinear}.

\begin{proposition}[Nonlinear Lebesgue spaces are metric spaces]
Let $p\in [1,\infty]$. When equipped with $D_p$, $\Lph(M,N)$ becomes a metric space with finite metric. 
Also, we will usually only write $\Lph(M,N)$ when referring to the metric space $(\Lph(M,N),D_p)$.
\end{proposition}

We highlight the fact that we will frequently not distinguish representatives from their associated equivalence class, so that we can write $f\in \Lph(M,N)$ while treating $f$ as a $\mathcal{L}^p$ mapping when it is clear from the context. Also, to ensure that our notation remains consistent with the literature on linear Lebesgue spaces, we denote $\mathcal{L}^p(M,N)\coloneqq \cLph(M,N)$ and $L^p(M,N)\coloneqq \Lph(M,N)$ when $N$ is a normed vector space and the base mapping is set to $h\equiv 0_N$ with $0_N$ the identity element of $N$.

An important matter to know about nonlinear Lebesgue spaces is whether they are reduced to the equivalence class of their base mapping, that is, $\Lph(M,N) =\left\{h\right\}$ for $h\in L^0(M,N)$. In such cases, nonlinear Lebesgue spaces are called \emph{trivial}. Nonlinear Lebesgue spaces are then called \emph{nontrivial} when they are strictly bigger, in the sense of inclusion, than the equivalence class of the base mapping. When $p\in [1,\infty)$, this is the case if and only if $\mu_M$ is not purely infinite, that is, the range of $\mu_M$ is not reduced to $\{0,\infty\}$, and $\lvert N\rvert > 1$. When $p=\infty$, this is the case if and only if $\mu_M$ is not trivial, that is, the range of $\mu_M$ is not reduced to $\{0\}$, and $\lvert N\rvert > 1$. These results are proved in \cite[Proposition 3.3]{serieys2025nonlinear} and collected in the following proposition. 

\begin{proposition}[Nontrivial nonlinear Lebesgue spaces]
\label{prop:trivial_lebesgue}
    Let $h\in L^0(M,N)$. Then, the following assertions hold:
    \begin{enumerate}[label=(\roman*)]
        \item for every $p\in [1,\infty)$, the space $\Lph(M,N)$ is nontrivial if and only if $\mu_M$ is not purely infinite and $\lvert N\rvert > 1$.
        \item the space $L^\infty_h(M,N)$ is nontrivial if and only if $\mu_M$ is nontrivial and $\lvert N\rvert > 1$.
    \end{enumerate}
\end{proposition}

Also, nonlinear Lebesgue spaces inherit completeness from their target space (see \cite[Proposition 4.1]{serieys2025nonlinear} or \cite[Proposition~3.3]{Sturm2001}). 

\begin{proposition}[Completeness of nonlinear Lebesgue spaces]
\label{prop:completeness_lebesgue}
Let $h\in L^0(M,N)$ and $p\in [1,\infty]$. Suppose that $N$ is complete. Then, $\Lph(M,N)$ is complete.
\end{proposition}

A subsidiary result of the proof of \cref{prop:completeness_lebesgue} is that the set of equivalence classes of measurable mappings also inherits completeness from the target space.

\begin{corollary}[Completeness of the space of measurable mappings]
\label{cor:completeness_measurable}
    Let $p\in [1,\infty]$. Suppose that $N$ is complete. Then, $(L^0(M,N),D_p)$ is complete.
\end{corollary}

In addition, when $\mu_M$ is not purely infinite, the target space is isometric to a closed subset of nonlinear Lebesgue spaces \cite[Corollary 3.8]{serieys2025nonlinear}.

\begin{proposition}[Closed isometric image of the target space into nonlinear Lebesgue spaces]
\label{prop:closed_embedding_target}
    Let $h\in L^0(M,N)$ and $p\in [1,\infty]$. Suppose that $\mu_M$ is not purely infinite. Then, there exists an isometric (up to a scaling factor) mapping $\iota : N\to \Lph(M,N)$ such that $\iota(N)$ is closed in $\Lph(M,N)$ and there exists a measurable set $A$ with $\mu_M(A)\in (0,\infty)$ for which $\iota(y)|_A\equiv y$ and $\iota(y)|_{M\setminus A} = h|_{M\setminus A}$ for all $y\in N$.
\end{proposition}
\begin{remark}[On a stronger conclusion when $\mu_M$ is finite and $h$ is bounded]
\label{rem:embedding_target_stronger_conclusion}
    When it is additionally assumed that $\mu_M$ is finite and that $h$ is bounded, we can actually choose $A=M$, so that $\iota(y)\equiv y$ for all $y\in N$ \cite[Proposition 3.7]{serieys2025nonlinear}.
\end{remark}

Finally, a class of mappings that will play an important role in propagating properties to nonlinear Lebesgue spaces through density arguments is the set of \emph{almost simple} mappings, that is, measurable mappings that are simple on a measurable set of finite $\mu_M$-measure and ($\mu_M$-a.e.) equal to the base mapping otherwise. 
\begin{definition}[Almost simple mappings]
\label{def:almost_simple}
Let $h\in L^0(M,N)$. Then, define:
    \begin{enumerate}[label=(\roman*)] 
        \item $\mathcal{E}_h(M,N)$ the set of \emph{almost simple mappings}, that is, all mappings $f\in \cLs(M,N)$ such that there exists $B\in \mathcal{F}_{\mu_M}$, that is, a measurable set of finite $\mu_M$-measure, for which $f$ satisfies $f|_B\in \mathcal{E}(B,N)$ and $f|_{M\setminus B}\sim_M h|_{M\setminus B}$.
        \item $E_h(M,N) \coloneqq q_M(\mathcal{E}_h(M,N))$ the set of equivalence classes admitting an almost simple representative.
    \end{enumerate}
\end{definition}

Indeed, almost simple mappings are dense in nonlinear Lebesgue spaces for $p\geq 1$ \cite[Proposition 3.10]{serieys2025nonlinear}.

\begin{proposition}[Density of almost simple mappings for $p\geq 1$]
\label{prop:density_almost_simple}
Let $h\in L^0(M,N)$ and $p\in [1,\infty)$. Then, $E_h(M,N)\cap \Lph(M,N)$ is a dense subspace of $\Lph(M,N)$.  
\end{proposition}

We now move to a reminder on curves in metric spaces.

\subsection{Curves in metric spaces}
\label{sec:curves_metric_space}

\subsubsection{Basic assumptions}

Before delving into the study of curves in metric spaces, we clarify the assumptions on the set representing time.
\begin{assumption}[Time interval]
\label{assum:time}
    In the rest of the article, the nonempty set $I$, representing time, is such that $I=[a,b]$ with $(a,b)\in \mathbb{R}^2$ satisfying $a< b$.
\end{assumption}

We also clarify the assumptions on the generic metric space used in this section.

\begin{assumption}[Generic metric space]
\label{assum:generic}
    In this section only, the generic nonempty set $X$ is assumed to carry the structure of a metric space with finite metric $d_X: X^2\to [0,\infty)$.
\end{assumption}

Throughout this article and this section, we shall not recall these assumptions in the statements of the results; yet, any additional assumption relative to \cref{assum:time} or \cref{assum:generic} will be specified.

\subsubsection{Sobolev curves and continuous representations}

This section follows \cite[Section 2.4]{lisini2007characterization}, but with slight relaxations on the assumptions of the results used in the present article. A first important class of curves in metric spaces is the set of \emph{Sobolev curves}.

\begin{definition}[Sobolev curves]
\label{def:sobolev_curves}
Let $p\in (1,\infty)$. Then, the set of \emph{Sobolev curves} is defined as
$$W^{1,p}(I,X)\coloneqq \left\{ c\in L^p(I,X): \sup_{0 < \tau < b-a} \int_a^{b-\tau} (\Delta_{X,\tau} c(t))^p\dif t < \infty \right\},$$
where $\Delta_{X,\tau} c(t)\coloneqq d_X(c(t),c(t+\tau))/\tau$. Also, define the finite metric $d_p: L^p(I,X)^2\to [0,\infty)$ as
$$d_p(c,c')\coloneqq \left(\int_I d_X(c(t),c'(t))^p \dif t\right)^{1/p}.$$
\end{definition}

\begin{remark}[On the definition of $L^p(I,X)$]
\label{rem:separable_range_sobolev}
    Note that unlike in \cite[Section 2.4]{lisini2007characterization}, the definition of $L^p(I,X)$ used in \cref{def:sobolev_curves} is that of \cref{def:lebesgue_spaces}, so that we can always pick a separably valued representative instead of assuming that $X$ is separable in most proofs. Also, the base mapping is a mapping constant equal to some $z_0\in X$ and thus any $y\in X$, since $I$ has finite $\lebesgue$-measure.
\end{remark}

\begin{remark}[$W^{1,p}(I,X)$ is a Borel subset of $L^p(I,X)$]
\label{rem:sobolev_borel}
    Let $p\in (1,\infty)$. As a direct consequence of \cref{def:sobolev_curves} and the fact that the functional 
    $c\mapsto \sup_{0<\tau <b-a}\int_a^{b-\tau} (\Delta_{X,\tau}c(t))^p\dif t$ is lower semicontinuous from $L^p(I,X)$ to $[0,\infty]$, $W^{1,p}(I,X)$ is a Borel subset of $L^p(I,X)$.
\end{remark}

We will see that Sobolev curves are related to another important class of curves in metric spaces, namely the set of \emph{absolutely continuous curves}, whose definition is recalled below and borrowed from \cite[Definition 1.1.1]{ambrosio2005gradient}.

\begin{definition}[Absolutely continuous curves]
\label{def:ac_curves}
A curve $c:I\to X$ belongs to $\mathcal{AC}^p(I,X)$, for $p\in [1,\infty]$, if and only if there exists $m\in L^p(I)$ such that for all $(s,t)\in I^2$ such that $s \leq t$ we have
\begin{align}
\label{eq:control_ac_def}
d_X(c(s),c(t))\leq \int_s^t m(r)\dif r.
\end{align}
\end{definition}

Absolutely continuous curves have the nice property of being almost everywhere differentiable in the metric sense \cite[Theorem~1.1.2]{ambrosio2005gradient}.

\begin{theorem}[Metric derivative]
\label{th:metric_derivative}
Let $p\in[1,\infty]$. Then, if $c\in \mathcal{AC}^p(I,X)$, the limit
$$\lvert c'\rvert_X(t) \coloneqq \lim\limits_{s\to t}\frac{d_X(c(s),c(t))}{\lvert s - t \rvert}$$
exists for $\lebesgue$-a.e. $t\in I$. In addition, the mapping $t \to \lvert c'\rvert_X(t)$ belongs to $L^p(I)$, is an admissible integrand for the right-hand side of \cref{eq:control_ac_def} and is minimal in the sense that for all $m\in L^p(I)$ verifying \cref{eq:control_ac_def} it satisfies 
$$\lvert c'\rvert_X(t) \leq m(t)\quad \text{for $\lebesgue$-a.e. $t\in I$.}$$
\end{theorem}

Now, when $p\in (1,\infty)$ and $X$ is complete, the set of absolutely continuous curves is isometric to the set of Sobolev curves for the metric $d_p$.

\begin{proposition}[Absolutely continuous representations of Sobolev curves]
\label{prop:ac_representation_sobolev}
    Let $p \in (1,\infty)$. Suppose that $X$ is complete. Then, the restriction of the quotient mapping 
    $$q_I: \left\{\begin{array}{ccl}\mathcal{L}^0_s(I,X)&\longrightarrow& L^0(I,X)\\
    c&\longmapsto& {[c]_I}
    \end{array}\right.$$ 
    to $\mathcal{AC}^p(I,X)$, denoted by $q_p$, is an isometry 
    between $(\mathcal{AC}^p(I,X),d_p)$ and $(W^{1,p}(I,X),d_p)$.
\end{proposition}

\begin{proof}
    Let $p\in (1,\infty)$. 
    
    \par\medskip \noindent \emph{Step 1: Definition of $q_p$.} Note that $q_p$ is well-defined since any curve in $\mathcal{AC}^p(I,X)$ has separable range as the continuous image of a separable set and since Borel measurability follows from continuity. Also, the fact that $q_p(\mathcal{AC}^p(I,X))\subset W^{1,p}(I,X)$ follows from the first part of the proof of \cite[Lemma 1]{lisini2007characterization}. 
    
    \par\medskip \noindent \emph{Step 2: Bijectivity of $q_p$.} The bijectivity of $q_p$ from $\mathcal{AC}^p(I,X)$ to $W^{1,p}(I,X)$ follows from the second part of the proof of \cite[Lemma 1]{lisini2007characterization} which shows that to each element $c\in W^{1,p}(I,X)$ we can associate a unique element $\tilde{c}\in \mathcal{AC}^p(I,X)$  such that $q_p(\tilde{c})= c$. The only variation in our case is that, instead of assuming that $X$ is separable, the proof can be carried out the same way by using the fact that elements of $W^{1,p}(I,X)$ have separably valued representatives (\cref{rem:separable_range_sobolev}).

    \par\medskip \noindent \emph{Step 3: Isometric property of $q_p$.} The isometric property of $q_p$ for $d_p$ follows from the invariance of $d_p$ to variations on $\lebesgue$-null sets.
\end{proof}

\cref{prop:ac_representation_sobolev} and \cref{def:sobolev_curves} yield an equivalent definition of the set of absolutely continuous curves, when $p\in (1,\infty)$ and $X$ is complete.

\begin{corollary}[Equivalent definition of absolutely continuous curves for $p> 1$]
\label{cor:equiv_def_ac}
    Let $p\in (1,\infty)$. Suppose that $X$ is complete. Then, the set of absolutely continuous curves is equivalently defined as 
    \begin{align*}
        \mathcal{AC}^p(I,X) = \left\{c\in \mathcal{C}(I,X): \sup_{0<\tau < b-a} \int_a^{b-\tau} (\Delta_{X,\tau} c(t))^p \dif t <\infty \right\}.
    \end{align*}
\end{corollary}
\begin{remark}[$\mathcal{AC}^p(I,X)$ is a Borel subset of $\mathcal{C}(I,X)$]
\label{rem:ac_borel}
    Let $p\in (1,\infty)$. As a direct consequence of \cref{cor:equiv_def_ac} and the fact that the functional 
    $c\mapsto \sup_{0<\tau <b-a}\int_a^{b-\tau} (\Delta_{X,\tau}c(t))^p\dif t$ is lower semicontinuous from $\mathcal{C}(I,X)$ to $[0,\infty]$, $\mathcal{AC}^p(I,X)$ is a Borel subset of $\mathcal{C}(I,X)$.
\end{remark}

Finally, note that the Borel structures induced by $d_p$ and $d_\infty$ on absolutely continuous curves coincide for all $p\in (1,\infty)$ as well as the corresponding notions of separable sets. 

\begin{proposition}[Coincidence of the Borel structures induced by $d_p$ and $d_\infty$ on $\mathcal{AC}^p(I,X)$]
\label{prop:comparison_dp_dinf}
Let $p\in (1,\infty)$. Then, the following assertions hold:
\begin{enumerate}[label=(\roman*)]
    \item the Borel $\sigma$-algebras induced by $d_p$ and $d_\infty$ on $\mathcal{AC}^p(I,X)$ coincide, that is, 
    $$\mathcal{B}\big(\mathcal{AC}^p(I,X),d_p\big) = \mathcal{B}\big(\mathcal{AC}^p(I,X),d_\infty\big).$$

    \item a subset $\mathcal{S}\subset\mathcal{AC}^p(I,X)$ is separable for
    $d_p$ if and only if it is separable for $d_\infty$.
\end{enumerate}
\end{proposition}

\begin{proof}
    See \ref{appendix:comparison_dp_dinf}. 
\end{proof}

\begin{remark}[Related results in the literature]
Assertion (i) of \cref{prop:comparison_dp_dinf} already appeared at the end of the proof of \cite[Lemma 1]{lisini2007characterization} under the additional assumptions that $X$ is complete and separable.
\end{remark}

As an immediate consequence of \cref{prop:comparison_dp_dinf}, the evaluation mapping is measurable on the set of absolutely continuous curves equipped with the metric $d_p$ as soon as $p\in (1,\infty)$.

\begin{corollary}[Measurability of the evaluation mapping for $d_p$]
\label{prop:measurability_evaluation_ac}
    Let $p\in (1,\infty)$. Then, for all $t\in I$, the restriction of the evaluation mapping 
    $$e_t: \left\{\begin{array}{ccl}
        \mathcal{C}(I,X)&\longrightarrow& X\\
        c&\longmapsto& c(t)
    \end{array}\right.$$ 
    to $\mathcal{AC}^p(I,X)$, denoted by $e_{t,p}$, is Borel measurable from $(\mathcal{AC}^p(I,X),d_p)$ to $X$.
\end{corollary}
\begin{proof}
    This follows from assertion (i) of \cref{prop:comparison_dp_dinf} as $e_t$ is continuous on $(\mathcal{C}(I,X),d_\infty)$ for every $t\in I$.
\end{proof}

We now continue with another important class of curves in metric spaces beyond continuous curves, namely the curves of bounded variation.

\subsubsection{Curves of bounded variation and càdlàg representations}

This section follows \cite[Section 2.2]{abedi2024absolutely}, but with slight relaxations on the assumptions of the results used in the present article. First, let us recall the definition of the \emph{variation} of curves.

\begin{definition}[Variation of curves]
\label{def:variation}
 The \emph{variation} of a curve $c: I\to X$ on any subset $J\subset I$ is defined as 
 \begin{align*}
     \operatorname{Var}_X(c; J) \coloneqq \sup\left\{\sum_{i=0}^n d_X(c(t_i),c(t_{i+1})): \{t_i\}_{0\leq i \leq n+1} \subset J\text{ and } t_0< \ldots < t_{n+1}\right\}
 \end{align*}
 and its \emph{essential variation} is defined as
 \begin{align*}
     \operatorname{ess\,Var}_X(c;J)\coloneqq \inf  \{\operatorname{Var}_X(\tilde{c};J): \text{$\tilde{c}=c$ $\lebesgue$-a.e.} \}.
 \end{align*}
    
\end{definition}

Using \cref{def:variation}, the set of \emph{curves of bounded variation} can be defined.

 \begin{definition}[Curves of bounded variation] 
 \label{def:bv_curves}
 The set of \emph{curves of bounded variation}, which we call BV-curves for short, is defined as
$$BV(I,X)\coloneqq \left\{ c\in L^1(I,X): \operatorname{ess\,Var}_X(c;I) < \infty \right\}.$$
Also, define the finite metric $d_1: L^1(I,X)^2\to [0,\infty)$ as
$$d_1(c,c')\coloneqq \int_I d_X(c(t),c'(t)) \dif t.$$
\end{definition}

\begin{remark}[On the definition of $L^1(I,X)$]
\label{rem:separable_range_bv}
    Note that unlike in \cite[Section 2.2]{abedi2024absolutely}, the definition of $L^1(I,X)$ used in \cref{def:bv_curves} is that of \cref{def:lebesgue_spaces}, so that we can always pick a separably valued representative instead of assuming that $X$ is separable in most proofs. Also, the base mapping is a mapping constant equal to some $z_0\in X$ and thus any $y\in X$, since $I$ has finite $\lebesgue$-measure.
\end{remark}

\begin{remark}[$BV(I,X)$ is a Borel subset of $L^1(I,X)$]
\label{rem:bv_borel}
    As a direct consequence of \cref{def:bv_curves} and the fact that the essential variation is lower semicontinuous from $L^1(I,X)$ to $[0,\infty]$ \cite[Lemma 2.7]{abedi2024absolutely}, $BV(I,X)$ is a Borel subset of $L^1(I,X)$.
\end{remark}

Using the definition of the essential variation of a curve, one can define the \emph{variation measure} associated to a BV-curve.

\begin{definition}[Variation measure]
\label{def:variation_measure}
     The \emph{variation measure} of a BV-curve $c\in BV(I,X)$ is defined as the Lebesgue--Stieltjes measure $\lvert Dc\rvert_X$ on $I$ induced by the non-decreasing function $V : I \to \mathbb{R}$ defined as $V(t)\coloneqq\operatorname{ess\,Var}_X(c;(a,t))$ (see \cite[Chapter 6, Section 3.3]{stein2009real} for the construction of $\lvert Dc\rvert_X$). In particular, $\lvert Dc\rvert_X$ satisfies for all $(s,t)\in I^2$ with $s< t$ that
     $$\lvert Dc\rvert_X((s,t)) = \lim_{u\to t^-} V(u) - \lim_{u\to s^+} V(u).$$
\end{definition}

Now, when $X$ is complete, we have an equivalent definition of the set of BV-curves.

\begin{proposition}[Equivalent definition of $BV(I,X)$] 
\label{prop:equiv_def_bv}
Suppose that $X$ is complete. Then, the set of curves of bounded variation is equivalently defined as
$$BV(I,X)= \left\{ c\in L^1(I,X): \sup_{0 < \tau < b-a} \int_a^{b-\tau} \Delta_{X,\tau} c(t)\dif t < \infty \right\}$$
and for all $c\in BV(I,X)$ we have for all $(s,t)\in I^2$ with $s<t $ that
$$\lvert Dc\rvert_X((s,t))= \operatorname{ess\,Var}_X(c;(s,t))=\sup_{0 < \tau < t-s} \int_s^{t-\tau} \Delta_{X,\tau} c(t)\dif t.$$
\end{proposition}

\begin{proof}
    The only variation with respect to the proof of \cite[Theorem 2.17]{abedi2024absolutely} is the fact that, instead of assuming that $X$ is separable, we can use the fact that elements of $BV(I,X)$ have separably valued representatives (\cref{rem:separable_range_bv}), so that, rather than taking a dense set in $X$ in the proof that (4)$\implies$(1), it is sufficient to take a dense set in the range of a separably valued representative.
\end{proof}

We will see that BV-curves are related to another important class of curves in metric spaces, namely the set of \emph{càdlàg curves of bounded variation}. Let us first recall the definition of \emph{càdlàg curves}.

\begin{definition}[Càdlàg curves]
 A curve $c: I\to X$ belongs to $\mathcal{D}(I,X)$ if it is right continuous at every $t\in [a,b)$, that is, $c(t) = \lim_{s\to t^+} c(s)$, and admits a left limit at every $t\in (a,b]$, that is, $\lim_{s \to t^-} c(s)$ exists in $X$. This set becomes a metric space when equipped with the \emph{Skorokhod metric} $d_{\operatorname{Sk}}: \mathcal{D}(I,X)^2 \to [0,\infty)$ defined as 
 $$d_{\operatorname{Sk}}(c,c')\coloneqq \inf_{\lambda}\max\{\lVert \lambda\rVert_B,d_\infty(c,c'\circ\lambda)\},$$
 where the infimum is taken over all increasing homeomorphisms $\lambda: I\to I$ and
 $$\lVert\lambda\rVert_B\coloneqq \sup_{(s,t)\in I^2, s <t}\left\lvert \log\left(\frac{\lambda(t) -\lambda(s)}{t - s}\right)\right\rvert.$$
\end{definition}

Then, one can define the set of \emph{càdlàg curves of bounded variation}.

\begin{definition}[Càdlàg curves of bounded variation]
\label{def:cadlag_bv}
A curve $c: I\to X$ belongs to $\mathcal{BV}(I,X)$ if it is càdlàg, continuous on the left at $t=b$ and has finite variation, that is, $c\in \mathcal{D}(I,X)$, $c(b)=\lim_{t\to b^-}c(t)$ and $\operatorname{Var}_X(c) < \infty$.
\end{definition}
\begin{remark}[On the condition of left-continuity at $t=b$]
    Note that unlike \cite[p.~16]{abedi2024absolutely} we define $\mathcal{BV}(I,X)$ by adding the condition of left-continuity at $t=b$. Indeed, for notational consistency, we argue that it makes more sense for $\mathcal{BV}(I,X)$ to denote the set that is in bijection with $BV(I,X)$ via the quotient mapping on $\cLs(I,X)$ (\cref{prop:cadlag_representation_bv}), rather than a strictly larger set. In addition, using the definition of $\mathcal{BV}(I,X)$ from \cref{def:cadlag_bv} instead of the one from \cite[p.~16]{abedi2024absolutely} does not affect the exposition of this section since the definition of $\mathcal{BV}(I,X)$ from \cref{def:cadlag_bv} is closed under limits for the Skorokhod metric, so that the results on the set of all càdlàg curves of finite variation mentioned in \cite{abedi2024absolutely} all hold for this definition.
\end{remark}

\begin{remark}[$\mathcal{BV}(I,X)$ is a Borel subset of $\mathcal{D}(I,X)$]
\label{rem:cadlag_borel}
    As a direct consequence of \cref{def:cadlag_bv} and the fact that the variation is lower semicontinuous from $\mathcal{D}(I,X)$ to $[0,\infty]$ \cite[Lemma 2.13]{abedi2024absolutely}, $\mathcal{BV}(I,X)$ is a Borel subset of $\mathcal{D}(I,X)$.
\end{remark}

When $X$ is complete, the set of BV-curves is isometric to the set of càdlàg curves of bounded variation for the metric $d_1$ \cite[Lemma 2.5]{abedi2024absolutely}.

\begin{proposition}[Càdlàg representations of BV-curves]
\label{prop:cadlag_representation_bv}
    Suppose that $X$ is complete. Then, the following assertions hold:
    \begin{enumerate}[label=(\roman*)]
        \item for all $c\in \mathcal{BV}(I,X)$ and all $(s,t)\in I^2$ such that $s\leq t$, we have
        $$\lvert D c\rvert_X((s,t)) = \operatorname{ess\,Var}_X(c;(s,t))= \operatorname{Var}_X(c;(s,t)).$$
        \item the restriction of the quotient mapping 
    $$q_I: \left\{\begin{array}{ccl}\mathcal{L}^0_s(I,X)&\longrightarrow& L^0(I,X)\\
    c&\longmapsto& [c]_I
    \end{array}\right.$$ 
    to $\mathcal{BV}(I,X)$, denoted by $q_1$, is an isometry between $(\mathcal{BV}(I,X),d_1)$ and $(BV(I,X),d_1)$.
    \end{enumerate}
\end{proposition}

Finally, note that the Borel structures induced by $d_1$ and $d_{\operatorname{Sk}}$ on càdlàg curves of bounded variation coincide as soon as $X$ is complete and separable. 

\begin{proposition}[Coincidence of the Borel structures induced by $d_1$ and $d_{\operatorname{Sk}}$ on $\mathcal{BV}(I,X)$]
\label{prop:comparison_d1_dsk}
Suppose that $X$ is complete and separable. Then, the Borel $\sigma$-algebras induced by $d_1$ and $d_{\operatorname{Sk}}$ on $\mathcal{BV}(I,X)$ coincide, that is, 
    $$\mathcal{B}\big(\mathcal{BV}(I,X),d_1\big) = \mathcal{B}\big(\mathcal{BV}(I,X),d_{\operatorname{Sk}}\big).$$
\end{proposition}

\begin{proof}
    The proof of \cref{prop:comparison_d1_dsk} follows the same main argument as the end of the proof of \cite[Lemma 1]{lisini2007characterization}. Recall that $BV(I,X)$ and $\mathcal{BV}(I,X)$ are Borel sets of $L^1(I,X)$ and $\mathcal{D}(I,X)$, respectively (see \cref{rem:bv_borel,rem:cadlag_borel}). As such, they are Lusin sets, hence Suslin sets as soon as $X$ is Polish (i.e., complete and separable), since, under such assumptions on $X$, $L^1(I,X)$ and $\mathcal{D}(I,X)$ become Polish (see \cite[Theorem 2, p.~95]{schwartz1973radon} for the proof that a Borel subset of a Lusin space is Lusin and \cite[Theorem 2.10]{abedi2024absolutely} for the proof that $\mathcal{D}(I,X)$ is Polish when $X$ is). Furthermore, $(\mathcal{BV}(I,X),d_1)$ being isometric to $(BV(I,X),d_1)$ by $q_1$ (\cref{prop:cadlag_representation_bv}), it is also Suslin. Hence, the conclusion follows from the fact that $d_1$ induces on $\mathcal{BV}(I,X)$ a Suslin topology that is weaker than the one induced by $d_{\operatorname{Sk}}$ \cite[Corollary 2, p.~101]{schwartz1973radon}.
\end{proof}

As an immediate consequence of \cref{prop:comparison_d1_dsk}, the evaluation mapping is measurable on the set of càdlàg curves of bounded variation equipped with the metric $d_1$ as soon as $X$ is complete and separable.

\begin{corollary}[Measurability of the evaluation mapping for $d_1$]
\label{prop:measurability_evaluation_bv}
 Suppose that $X$ is complete and separable. For all $t\in I$, the restriction of the evaluation mapping 
 $$e_t:\left\{\begin{array}{ccl}
     \mathcal{D}(I,X)&\longrightarrow& X\\
     c&\longmapsto& c(t)
 \end{array}\right.$$
 to $\mathcal{BV}(I,X)$, denoted by $e_{t,1}$, is Borel measurable from $(\mathcal{BV}(I,X),d_1)$ to $X$.
\end{corollary}

\begin{proof}
    The mapping $e_t$ being Borel measurable from $(\mathcal{D}(I,X),d_{\operatorname{Sk}})$ to $X$ for all $t\in I$ \cite[Proposition 2.15]{abedi2024absolutely}, its restriction to $\mathcal{BV}(I,X)$, denoted by $e_{t,1}$, is Borel measurable from $(\mathcal{BV}(I,X),d_{\operatorname{Sk}})$ to $X$ for all $t\in I$. Hence, the conclusion of \cref{prop:measurability_evaluation_bv} follows from \cref{prop:comparison_d1_dsk}. 
\end{proof}

To characterize the behaviours of all the classes of curves presented in \cref{sec:curves_metric_space} for the particular case of nonlinear Lebesgue spaces, we will need to characterize $L^p$ curves in nonlinear Lebesgue spaces since Sobolev and BV curves are defined using $L^p$ curves (\cref{def:sobolev_curves,def:bv_curves}). The next section thus presents such a characterization of $L^p$ curves and its consequence for the characterization of absolutely continuous curves and càdlàg curves of bounded variation in nonlinear Lebesgue spaces.

\section{Main results}
\label{sec:main_results}
\subsection{\texorpdfstring{Characterization of $L^p$ curves in nonlinear Lebesgue spaces}{Characterization of Lp curves in nonlinear Lebesgue spaces}}
\label{sec:embeddings_isomorphisms}

We first clarify some recurrent notation used in this section and thereafter.
\begin{notation}
\label{not:lp}
    Let $h\in L^0(M,N)$,  $p\in [1,\infty)$ and $\lebesgue\otimes \mu_M$ be the product measure obtained using Carathéodory's process \cite[Chapter II]{dibenedetto2002real}. From thereon:
    \begin{enumerate}[label=(\roman*)]
        \item $\bar{h}$ denotes the mapping from $I$ to $\Lph(M,N)$ defined as $\bar{h}\equiv h$.
        \item $\underline{h}$ denotes the mapping from $M$ to $L^p(I,N)$ defined as $\underline{h}\coloneqq \iota \circ h$, which is measurable from $M$ to  $L^p(I,N)$ as the composition of the measurable mapping $h$ and the isometric mapping $\iota$ defined in \cref{rem:embedding_target_stronger_conclusion}. The mapping $\underline{h}$ is also measurable from $M$ to $(E(I,N),d_p)$ since the range of $\iota$ is closed in $L^p(I,N)$ and included in $E(I,N)$.
\item $\hat{h}$ denotes the mapping from $I\times M$ to $N$ defined as $\hat{h}(t,x) = h(x)$ for $\lebesgue \otimes \mu_M$-a.e.~$(t,x)\in I\times  M$.
        \item $L^p(I,\Lph(M,N))$ denotes the nonlinear Lebesgue space with base mapping set as $\bar{h}$.
        \item $d_{p,p}$ denotes the finite metric on $L^p(I,\Lph(M,N))$ defined as 
$$d_{p,p}(c,c')\coloneqq \left(\int_I D_p(c(t),c'(t))^p\dif t\right)^{1/p}.$$
\item $D_{p,p}$ denotes the finite metric on $L^p_{\underline{h}}(M,L^p(I,N))$ defined as 
$$D_{p,p}(f,f')\coloneqq \left(\int_M d_p(f(x),f'(x))^p\dif\mu_M(x)\right)^{1/p}.$$
\item $\hat{D}_p$ denotes the finite metric on $L^p_{\hat{h}}(I\times M,N)$ defined as 
$$\hat{D}_p(\hat{c},\hat{c}')\coloneqq \left(\int_{I\times M} d_N(\hat{c}(t,x),\hat{c}'(t,x))^p\dif\lebesgue\otimes \mu_M(t,x)\right)^{1/p}.$$
\end{enumerate}
\end{notation}

Establishing a characterization of $L^p$ curves in nonlinear Lebesgue spaces essentially consists in the identification of a canonical isometry of spaces of simple mappings that extends to nonlinear Lebesgue spaces through a density argument. To identify such a canonical isometry, a core ingredient is the set of \emph{rectangular almost simple mappings} defined below. 

\begin{definition}[Rectangular almost simple mappings]
\label{def:rectangle_almost_simple}
    Let $h\in L^0(M,N)$. Then, define:
    \begin{enumerate}[label=(\roman*)]
        \item $\mathcal{E}_{\hat{h},\sigma}(I\times M,N)$ the set of \emph{rectangular almost simple mappings}, that is, all mappings $\hat{c}\in \cL(I\times M,N)$ such that there exist a finite collection $(y_J)_{j\in J}$ of elements of $N$ and a finite collection $(I_j\times M_j)_{j\in J}$ of disjoint measurable rectangles of finite $\lebesgue\otimes \mu_M$-measure such that $\hat{c}|_{I_j\times M_j}\equiv y_j$ for all $j\in J$ and $\hat{c}|_{(I\times M)\setminus R}\sim_{I\times M} \hat{h}|_{(I\times M)\setminus R}$ with $R\coloneqq \cup_{j\in J} I_j\times M_j$.
        \item $E_{\hat{h},\sigma}(I\times M,N)\coloneqq q_{I\times M}(\mathcal{E}_{\hat{h},\sigma}(I\times M,N))$ the set of equivalence classes admitting a rectangular almost simple representative.
    \end{enumerate}
    Note that, by definition, $E_{\hat{h},\sigma}(I\times M,N) \subset E_{\hat{h}}(I\times M,N)$.
\end{definition}

On the set of rectangular almost simple mappings, the section mappings on $I$ and $M$ define isometries. 

\begin{proposition}[Isometries induced by the section mappings on $I$ and $M$]
\label{prop:embed_iterated_simple}
    Let $h\in L^0(M,N)$ and $p\in [1,\infty)$. Then, the following assertions hold:
    \begin{enumerate}[label=(\roman*)]
    \item the section mapping on $I$ 
    $$\operatorname{sec}_{I,p}:\left\{\begin{array}{ccl}
       E_{\hat{h},\sigma}(I\times M,N)\cap L^p_{\hat{h}}(I\times M,N) &\longrightarrow& E(I,E_h(M,N)\cap L^p_h(M,N))\\
        {[\hat{c}]_{I\times M}} &\longmapsto & {[t\mapsto {[\hat{c}(t,\cdot)]_M}]_I}
    \end{array}\right.$$
    is an isometry between those sets equipped with the metrics $\hat{D}_p$ and $d_{p,p}$, respectively.
    \item the section mapping on $M$ 
    $$\operatorname{sec}_{M,p}:\left\{\begin{array}{ccl}
        E_{\hat{h},\sigma}(I\times M,N)\cap L^p_{\hat{h}}(I\times M,N) &\longrightarrow& E_{\underline{h}}(M,E(I,N))\cap L^p_{\underline{h}}(M,L^p(I,N))\\
        {[\hat{c}]_{I\times M}} &\longmapsto & {[x\mapsto {[\hat{c}(\cdot,x)]_I}]_M}
    \end{array}\right.$$
    is an isometry between those sets equipped with the metrics $\hat{D}_p$ and $D_{p,p}$, respectively.
    \end{enumerate}
\end{proposition}

\begin{remark}[On the choice of product measure $\lebesgue \otimes \mu_M$]
    The choice of product measure $\lebesgue \otimes \mu_M$ does not matter, since we can always reduce to integration over measurable rectangles on which all product measures coincide.
\end{remark}

\begin{proof}[Proof of \cref{prop:embed_iterated_simple}]
    Let $h\in L^0(M,N)$ and $p\in [1,\infty)$. 
    \begin{enumerate}[label=(\roman*),wide]
    \item \emph{Step 1: Definition of $\operatorname{sec}_{I,p}$.} Let $\hat{c}\in E_{\hat{h},\sigma}(I\times M,N)\cap L^p_{\hat{h}}(I\times M,N)$. Then, there is $\tilde{c}\in \mathcal{E}_{\hat{h},\sigma}(I\times M,N)\cap \mathcal{L}^p_{\hat{h}}(I\times M,N)$ such that $\hat{D}_p(\tilde{c},\hat{c})=0$. By definition of $\tilde{c}$, there exists $n\in \mathbb{N}$, a collection $(y_i)_{0\leq i\leq n}$ of elements of $N$ and a collection $(I_i\times M_i)_{0\leq i \leq n}$ of disjoint measurable rectangles of finite $\lebesgue\otimes \mu_M$-measure such that $\tilde{c}|_{I_i\times M_i}\equiv y_i$ for all $0\leq i\leq n$ and $\tilde{c}|_{(I\times M)\setminus R} = h|_{(I\times M)\setminus R}$ with $R\coloneqq\cup_{i=0}^n (I_i\times M_i)$. Then, $t\mapsto [\tilde{c}(t,\cdot)]_M$ is simple from $I$ to $(E_h(M,N),D_p)$ and, since $\tilde{c}\in\mathcal{L}^p_{\hat{h}}(I\times M,N)$ and $R$ is the finite union of measurable rectangles of finite $\lebesgue\otimes \mu_M$-measure, we have, by the Fubini--Tonelli theorem \cite[Proposition 5.2.1]{cohn2013measure}, that $D_p(\tilde{c}(t,\cdot),h)<\infty $ for $\lebesgue$-a.e.~$t\in I$, hence $[\tilde{c}(t,\cdot)]_M\in E_h(M,N)\cap \Lph(M,N)$ for $\lebesgue$-a.e.~$t\in I$. This shows that $\operatorname{sec}_{I,p}(\hat{c})=[t\mapsto [\tilde{c}(t,\cdot)]_M]_I$ belongs to $ E(I,E_h(M,N)\cap \Lph(M,N))$.

    \par\medskip \noindent \emph{Step 2: Isometric property of $\operatorname{sec}_{I,p}$.} Let $(\hat{c},\hat{c}')\in (E_{\hat{h},\sigma}(I\times M,N))\cap L^p_{\hat{h}}(I\times M,N))^2$. Since $\hat{c}$ and $\hat{c}'$ differ on the finite union of measurable rectangles of finite $\lebesgue\otimes \mu_M$-measure, we have, by the Fubini--Tonelli theorem \cite[Proposition 5.2.1]{cohn2013measure}, that 
    \begin{equation*}d_{p,p}(\sec_{I,p}(\hat{c}),\operatorname{\sec}_{I,p}(\hat{c}')) = \hat{D}_p(\hat{c},\hat{c}').\end{equation*}

    \par\medskip \noindent \emph{Step 3: Surjectivity of $\operatorname{sec}_{I,p}$.} Let $c\in E(I,E_h(M,N)\cap L^p_h(M,N))$. Then, there is $\tilde{c}\in \mathcal{E}(I,\mathcal{E}_h(M,N)\cap \cLph(M,N))$ such that $d_{p,p}(\tilde{c},c) = 0$. By definition of $\tilde{c}$, there exists a finite indexing set $I_c$ such that $\tilde{c}(I) = \{\tilde{s}_i\in \mathcal{E}_h(M,N)\cap \cLph(M,N): i\in I_c\}$. Also, define $I_i\coloneqq \tilde{c}^{-1}(\{\tilde{s}_i\})$ for all $i\in I_c$. In addition, there exist $B_i\in \mathcal{F}_{\mu_M}$ outside which $\tilde{s}_i$ agrees with $h$ and a finite indexing set $J_i$ such that $\tilde{s}_i(B_i)\coloneqq \{y_{i,j}\in N: j\in J_i\}$ and $\tilde{s}_i|_{M\setminus B_i}= h|_{M\setminus B_i}$. Then, define $B_{i,j}\coloneqq B_i\cap \tilde{s}_i^{-1}(\{y_{i,j}\})$ for all $j\in J_i$. Since the $I_i$'s are pairwise disjoint and, for every $i\in I_c$, the $(B_{i,j})_{j\in J_i}$ are pairwise disjoint, defining $\Gamma_c \coloneqq \{(i,j): i\in I_c, j\in J_i\}$, we observe that $(I_i\times B_{i,j})_{(i,j)\in \Gamma_c}$ is a finite collection of disjoint measurable rectangles of finite $\lebesgue\otimes \mu_M$-measure and we can define the mapping $\hat{c}: I\times M\to N$ as $\hat{c}|_{I_i\times B_{i,j}}\equiv y_{i,j}$ for all $(i,j)\in \Gamma_c$ and $\hat{c}(t,x) = h(x)$ for all $(t,x)\in(I\times M)\setminus P$ with $P\coloneqq\cup_{(i,j)\in \Gamma_c} I_i\times B_{i,j}$. By construction, $D_p(\hat{c}(t,\cdot),c(t))=0$ for $\lebesgue$-a.e.~$t\in I$ and $\hat{c}$ differs from $\hat{h}$ on the finite union of measurable rectangles of finite $\lebesgue\otimes \mu_M$-measure, hence, by the Fubini--Tonelli theorem \cite[Proposition 5.2.1]{cohn2013measure}, $\hat{D}_p(\hat{c},\hat{h}) = d_{p,p}(c,\bar{h}) < \infty$. Thus, up to the identification with its equivalence class, $\hat{c}$ belongs to $E_{\hat{h},\sigma}(I\times M,N)\cap L^p_{\hat{h}}(I\times M,N)$ and $c=\operatorname{sec}_{I,p}(\hat{c})$, that is, $c$ belongs to the range of $\operatorname{sec}_{I,p}$.
    \item The proof follows from the same arguments as the proof of (i), but using almost simple mappings.\qedhere
    \end{enumerate}
\end{proof}
\begin{remark}[On relaxing the assumptions on $I$]
\label{rem:relax_isom_simple}
    Note that the proof of \cref{prop:embed_iterated_simple} does not rely on any specific property of $I$ other than the finiteness of the measure on it to ensure that $E_{\underline{h}}(M,E(I,N))\cap L^p_{\underline{h}}(M,L^p(I,N))$ is nonempty. Hence, \cref{prop:embed_iterated_simple} holds for a measure space $(I,\Sigma_I,\mu_I)$ such that $\mu_I$ is finite. Alternatively, to ensure that $E_{\underline{h}}(M,E(I,N))\cap L^p_{\underline{h}}(M,L^p(I,N))$ is nonempty, we can replace the finiteness assumption on $\mu_I$ by the assumption that $h\in E(M,N)$. In that case, $E_{\underline{h}}(M,E(I,N))\cap L^p_{\underline{h}}(M,L^p(I,N))$ should be replaced by $E_{\underline{h}}(M,E(I,N))\cap L^p_{\underline{h}}(M,L^0(I,N))$ in assertion (ii) of \cref{prop:embed_iterated_simple}, with $L^p_{\underline{h}}(M,L^0(I,N))$ denoting the set of equivalence classes of separably valued measurable mappings from $M$ to $(L^0(I,N),d_p)$ at a finite $D_{p,p}$-distance from $\underline{h}$. If $h$ is not only simple, but constant, the assertion (ii) of \cref{prop:embed_iterated_simple} holds as is. The assertion (i) of \cref{prop:embed_iterated_simple} even holds when $(I,\Sigma_I,\mu_I)$ is just a measure space, with no further assumptions, by reasoning with almost simple mappings instead of simple mappings. 
\end{remark}

As a direct consequence of \cref{prop:embed_iterated_simple}, we can define an isometry between the target spaces of $\operatorname{sec}_{I,p}$ and $\operatorname{sec}_{M,p}$.

\begin{corollary}[Isometry between the target spaces of $\operatorname{sec}_{I,p}$ and $\operatorname{sec}_{M,p}$]
\label{cor:isom_iterated_simple}
    Let $h\in L^0(M,N)$ and $p\in [1,\infty)$. Then, the isometries  $\operatorname{sec}_{I,p}$ and $\operatorname{sec}_{M,p}$ defined in \cref{prop:embed_iterated_simple} induce the isometry
    $$i_p\coloneqq \operatorname{sec}_{M,p}\circ \operatorname{sec}_{I,p}^{-1}: E(I,E_h(M,N)\cap L^p_h(M,N))\to E_{\underline{h}}(M,E(I,N))\cap L^p_{\underline{h}}(M,L^p(I,N)).$$
\end{corollary}

Let us now identify the closures of the target and base spaces of $\operatorname{sec}_{I,p}$ and $\operatorname{sec}_{M,p}$. 
\begin{proposition}[Closures of the base and target spaces of $\operatorname{sec}_{I,p}$ and $\operatorname{sec}_{M,p}$]
\label{prop:density_iterated_almost_simple}
Let $h\in L^0(M,N)$ and $p\in [1,\infty)$. Then, the following assertions hold:
\begin{enumerate}[label=(\roman*)]
    \item the closure of the base space of $\operatorname{sec}_{I,p}$ and $\operatorname{sec}_{M,p}$ is 
    $$\overline{E_{\hat{h},\sigma}(I\times M,N)\cap L^p_{\hat{h}}(I\times M,N)}^{\,\hat{D}_p}=L^p_{\hat{h}}(I\times M,N).$$
    \item the closure of the target space of $\operatorname{sec}_{I,p}$ is 
    $$\overline{E(I, E_h(M,N)\cap L^p_h(M,N))}^{\,d_{p,p}}=L^p(I,L^p_h(M,N)).$$
    \item the closure of the target space of $\operatorname{sec}_{M,p}$ is 
    $$\overline{E_{\underline{h}}(M,E(I,N))\cap L^p_{\underline{h}}(M,L^p(I,N))}^{\,D_{p,p}}=L^p_{\underline{h}}(M,L^p(I,N)).$$
\end{enumerate}
\end{proposition}

\begin{proof}
    Let $h\in \cLs(M,N)$, $p\in [1,\infty)$ and $\varepsilon > 0$. 
    
    \begin{enumerate}[label=(\roman*),wide]
        \item Let $\hat{c}\in \mathcal{L}^p_{\hat{h}}(I\times M,N)$. Then, $\hat{c}$ differs from $\hat{h}$ on a measurable set $\hat{P}$ of $\sigma$-finite $\lebesgue\otimes \mu_M$-measure, that is, there exists a sequence $(P_n)_{n\in\mathbb{N}}$ of measurable sets, which can be assumed increasing, such that $\hat{P}=\cup_{n\in\mathbb{N}}P_n$.

    \par\medskip \noindent \emph{Step 1: Approximation by a Lebesgue mapping such that the set where it differs from $\hat{h}$ is contained in a measurable set where $\hat{h}$ is bounded.} Let $z_0\in N$ and $H_n\coloneqq\{x\in M:d_N(z_0,h(x))\leq n\}$ for all $n\in\mathbb{N}$, so that, $d_N$ being assumed finite, $M=\cup_{n\in\mathbb{N}} H_n$. Also, define the mapping $\hat{c}_n: I\times M\to N$ as 
    $$\hat{c}_n(t,x)\coloneqq\begin{cases}
        \hat{c}(t,x),&\text{if $x\in H_n$},\\
        \hat{h}(t,x),&\text{otherwise}.
    \end{cases}$$
    Then, since $(H_n)_{n\in\mathbb{N}}$ is an increasing sequence of sets in $M$, we have for all $(t,x)\in I\times M$ that $\hat{c}_n(t,x)\to \hat{c}(t,x)$ as $n\to \infty$ and $d_N(\hat{c}_n(t,x),\hat{c}(t,x))\leq d_N(\hat{h}(t,x),\hat{c}(t,x))$ for all $n\in \mathbb{N}$. Since $\hat{c}\in \mathcal{L}^p_{\hat{h}}(I\times M,N)$, we have, by Lebesgue's dominated convergence theorem \cite[Theorem~2.4.5]{cohn2013measure}, that $\hat{D}_p(\hat{c}_n,\hat{c})\to 0$ as $n\to \infty$. Hence, there exists $n_0\in\mathbb{N}$ such that $\hat{D}_p(\hat{c}_{n_0},\hat{c}) < \varepsilon /4$. In addition, $\hat{c}_{n_0}$ differs from $\hat{h}$ on a measurable set contained in the measurable rectangle $I\times H_{n_0}$.

    \par\medskip \noindent \emph{Step 2: Approximation by a Lebesgue mapping such that the set where it differs from $\hat{h}$ is contained in a measurable set of finite $\lebesgue\otimes \mu_M$-measure.} For all $n\in \mathbb{N}$, define the mapping $\hat{c}_n': I\times M\to N$ as 
    $$\hat{c}_n'(t,x)\coloneqq\begin{cases}
        \hat{c}_{n_0}(t,x),&\text{if $(t,x)\in P_n$},\\
        \hat{h}(t,x),&\text{otherwise}.
    \end{cases}$$
    Then, since $(P_n)_{n\in\mathbb{N}}$ is an increasing sequence of sets in $\hat{P}$, we have for all $(t,x)\in I\times M$ that $\hat{c}_n'(t,x)\to \hat{c}_{n_0}(t,x)$ as $n\to \infty$ and $d_N(\hat{c}_n'(t,x),\hat{c}_{n_0}(t,x))\leq d_N(\hat{h}(t,x),\hat{c}_{n_0}(t,x))$ for all $n\in \mathbb{N}$. Since $\hat{c}_{n_0}\in \mathcal{L}^p_{\hat{h}}(I\times M,N)$, we have, by Lebesgue's dominated convergence theorem \cite[Theorem~2.4.5]{cohn2013measure}, that $\hat{D}_p(\hat{c}_n',\hat{c}_{n_0})\to 0$ as $n\to \infty$. Hence, there exists $n_1\in\mathbb{N}$ such that $\hat{D}_p(\hat{c}_{n_1}',\hat{c}_{n_0}) < \varepsilon/4$. In addition, $\hat{c}_{n_1}'$ differs from $\hat{h}$ on a measurable set contained in $P_{n_1}'\coloneqq P_{n_1}\cap (I\times H_{n_0})$, which has finite $\lebesgue\otimes\mu_M$-measure.

    \par\medskip \noindent \emph{Step 3: Approximation by an almost simple mapping such that the set where it differs from $\hat{h}$ is contained in a measurable set of finite $\lebesgue\otimes \mu_M$-measure.} Using \cref{prop:density_almost_simple}, we can construct $\hat{s}\in \mathcal{E}_{\hat{h}}(I\times M,N)$ such that it differs from $\hat{h}$ on a measurable set $\hat{S}$ contained in $P_{n_1}'$ and $\hat{D}_p(\hat{c}_{n_1}',\hat{s}) < \varepsilon/4$. Hence, we can assume that there exists a finite indexing set $J$ such that $\hat{s}(\hat{S})=\{y_j\in N: j\in J\}$ and let $\hat{S}_j\coloneqq \hat{S}\cap \hat{s}^{-1}(\{y_j\})$ for all $j\in J$. Also, define 
    $$C \coloneqq \max \left\{\max_{i\neq j} d_N(y_i,y_j), \max_{j\in J}\{d_N(y_j,z_0) + n_0\}\right\}.$$
     \par\medskip \noindent \emph{Step 4: Approximation by a rectangular almost simple mapping.} For all $j\in J$, $\hat{S}_j$ has finite $\lebesgue\otimes \mu_M$-measure, hence, by construction of $\lebesgue\otimes \mu_M$, there exists a countable collection of measurable rectangles $(\hat{I}_{j,n}\times \hat{M}_{j,n})_{n\in\mathbb{N}}$ such that, defining $\hat{R}_j\coloneqq \cup_{n\in\mathbb{N}} (\hat{I}_{j,n}\times \hat{M}_{j,n})_{n\in\mathbb{N}}$, it satisfies  $\hat{S}_j\subset \hat{R}_j$ and $\lebesgue\otimes \mu_M(\hat{R}_j\setminus \hat{S}_j) < \varepsilon^p/(2\lvert J\rvert C^p4^p)$ \cite[Chapter II, Proposition 10.2]{dibenedetto2002real}. In addition, since $\hat{S}\subset I\times H_{n_0}$, it can be assumed that $\hat{M}_{j,n}\subset H_{n_0}$ for all $(j,n)\in J\times \mathbb{N}$ since the intersection of measurable rectangles remains a measurable rectangle. Then, define $\hat{R}_{j,n}\coloneqq \cup_{k=0}^n(\hat{I}_{j,k}\times \hat{M}_{j,k})$, so that $(\cup_{j\in J}\hat{R}_{j,n})_{n\in\mathbb{N}}$ is an increasing sequence of sets in $\cup_{j\in J}\hat{R}_j$. Hence, $J$ being finite, there exists $n_2\in \mathbb{N}$ such that $\sum_{j\in J}\lebesgue\otimes \mu_M(\hat{R}_j\setminus  \hat{R}_{j,n_2})< \varepsilon^p/(2 C^p4^p)$. Since $(\hat{R}_{j,n_2})_{j\in J}$ is a finite collection of finite unions of measurable rectangles and since the set of all measurable rectangles is a semi-algebra \cite[Chapter III, Section 12]{dibenedetto2002real}, the $\hat{R}_{j,n_2}$'s can be assumed disjoint, up to replacing them with the collection $(\hat{R}_{j,n_2}')_{j\in J}$ such that, assuming that $J=\{0,1,\ldots,m\}$ for some $m\in \mathbb{N}$, $\hat{R}_{0,n_2}'\coloneqq \hat{R}_{0,n_2}$ and $\hat{R}_{j,n_2}'\coloneqq \hat{R}_{j,n_2}\setminus (\cup_{k=0}^{j-1} \hat{R}_{k,n_2})$, which can be decomposed in the finite union of measurable rectangles, for all $0< j\leq m$. Now, define the mapping $\hat{s}_{n_2}: I\times M\to N$ as $\hat{s}_{n_2}|_{\hat{R}_{j,n_2}}\equiv y_j$ for all $j\in J$ and $\hat{s}_{n_2}|_{(I\times M)\setminus \hat{R}} = \hat{h}|_{(I\times M)\setminus \hat{R}}$ with $\hat{R}\coloneqq \cup_{j\in J}\hat{R}_{j,n_2}$. Since $\hat{R}$ is the finite union of measurable rectangles, it can be decomposed into the finite union of disjoint measurable rectangles. Hence, by construction, $\hat{s}_{n_2}$ belongs to $\mathcal{E}_{\hat{h},\sigma}(I\times M,N)$ and satisfies 
    \begin{align*}\hat{D}_p(\hat{s},\hat{s}_{n_2})^p&\leq C^p\sum_{j\in J} \lebesgue\otimes \mu_M(\hat{S}_j\Delta \hat{R}_{j,n_2})\\
    &\leq C^p\bigg(\sum_{j\in J} \lebesgue\otimes \mu_M(\hat{R}_j\setminus \hat{R}_{j,n_2}) + \sum_{j\in J} \lebesgue\otimes \mu_M(\hat{R}_j\setminus \hat{S}_j)\bigg) < (\varepsilon/4)^p.
    \end{align*}

    \noindent Finally, we have, by the triangle inequality, that 
    \begin{equation*}
        \hat{D}_p(\hat{c},\hat{s}_{n_2})\leq \hat{D}_p(\hat{c},\hat{c}_{n_0}) + \hat{D}_p(\hat{c}_{n_0},\hat{c}_{n_1}') + \hat{D}_p(\hat{c}_{n_1}',\hat{s}) + \hat{D}_p(\hat{s},\hat{s}_{n_2}) < \varepsilon.
    \end{equation*}
    
        \item Let $c\in L^p(I,L^p_h(M,N))$. 
        
        \par\medskip \noindent \emph{Step 1: Approximation by a simple mapping.} By \cref{prop:density_almost_simple}, there exists $s\in E(I,L^p_h(M,N))$ such that $d_{p,p}(c,s) < \varepsilon/2$. Up to its identification with an appropriate representative, we can assume that $s$ is a simple mapping, hence there exists a finite indexing set $I_s$ such that $s(I)\coloneqq \{ s_i \in L^p_h(M,N): i\in I_s\}$ and define $I_i \coloneqq s^{-1}(\{s_i\})$.
        
        \par\medskip \noindent \emph{Step 2: Approximation of the values of the simple mapping by almost simple mappings.} For each $i\in I_s$, still using \cref{prop:density_almost_simple}, there exists $\tilde{s}_i\in E_h(M,N)\cap \Lph(M,N)$ such that $D_p(s_i, \tilde{s_i}) < \varepsilon /(2^p\lvert I_s\rvert\,\lebesgue(I_i))^{1/p}$. Then, define the mapping $\tilde{s}: I\to E_h(M,N)\cap \Lph(M,N)$ as $\tilde{s}|_{I_i}\equiv \tilde{s}_i$ for all $i\in I_s$, so that we get 
    $$d_{p,p}(s,\tilde{s})^p = \sum_{i\in I_s} D_p(s_i,\tilde{s}_i)^p\lebesgue(I_i) < (\varepsilon/2)^p.$$
    
    \noindent Finally, we have, by the triangle inequality, that
    \begin{equation*}d_{p,p}(c, \tilde{s}) \leq d_{p,p}(c,s) + d_{p,p}(s,\tilde{s}) < \varepsilon.\end{equation*}
    \item The proof follows from the same arguments as the proof of (ii), but using almost simple mappings.\qedhere
    
    \end{enumerate}
\end{proof}

\begin{remark}[On relaxing the assumptions on $I$]
\label{rem:relax_density}
    Note that, similarly to what was highlighted in \cref{rem:relax_isom_simple} for \cref{prop:embed_iterated_simple}, the proof of \cref{prop:density_iterated_almost_simple} does not rely on any specific property of $I$ other than the finiteness of the measure on it to ensure that $E_{\underline{h}}(M,E(I,N))\cap L^p_{\underline{h}}(M,L^p(I,N))$ is nonempty. Hence, the result holds for a measure space $(I,\Sigma_I,\mu_I)$ such that $\mu_I$ is finite. The assumption of finiteness of $\mu_I$ can be replaced by the assumption that $h\in E(M,N)$. In that case, the assertion (iii) of \cref{prop:density_iterated_almost_simple} should be replaced by \cref{prop:alternative_density}. If $h$ is not only simple, but constant, the assertion (iii) of \cref{prop:density_iterated_almost_simple} holds as is. The assertions (i) and (ii) of \cref{prop:density_iterated_almost_simple} even hold when $(I,\Sigma_I,\mu_I)$ is just a measure space, with no further assumptions, by reasoning with almost simple mappings instead of simple mappings.
\end{remark}

Using \cref{prop:density_iterated_almost_simple}, the isometries defined in \cref{prop:embed_iterated_simple} can be uniquely extended to nonlinear Lebesgue spaces as soon as $N$ is complete. The latter statement actually yields a nonlinear analogue of the Fubini--Lebesgue theorem (see \cref{rem:related_extension} for a further discussion on that matter).

\begin{theorem}[Extension of $\operatorname{sec}_{I,p}$ and $\operatorname{sec}_{M,p}$]
    \label{th:extension_embed_iterated_simple}
    Let $h\in L^0(M,N)$ and $p\in [1,\infty)$. Suppose that $N$ is complete. Then, the following assertions hold:
    \begin{enumerate}[label=(\roman*)]
    \item the isometry $\operatorname{sec}_{I,p}$ defined in (i) of \cref{prop:embed_iterated_simple} extends uniquely to the isometry
    $$\overline{\operatorname{sec}}_{I,p}:\left\{\begin{array}{ccl}
       L^p_{\hat{h}}(I\times M,N) &\longrightarrow& L^p(I,L^p_h(M,N))\\
        {[\hat{c}]_{I\times M}} &\longmapsto & {[t\mapsto {[\hat{c}(t,\cdot)]_M}]_I}
    \end{array}\right..$$
    \item the isometry $\operatorname{sec}_{M,p}$ defined in (ii) of \cref{prop:embed_iterated_simple} extends uniquely to the isometry
    $$\overline{\operatorname{sec}}_{M,p}:\left\{\begin{array}{ccl}
        L^p_{\hat{h}}(I\times M,N) &\longrightarrow& L^p_{\underline{h}}(M,L^p(I,N))\\
        {[\hat{c}]_{I\times M}} &\longmapsto & {[x\mapsto {[\hat{c}(\cdot,x)]_I}]_M}
    \end{array}\right..$$
    \end{enumerate}
\end{theorem}

\begin{remark}[On relaxing the finiteness assumption on $d_N$]
\label{rem:metric_nonfinite}
    \cref{th:extension_embed_iterated_simple} extends to the case where $d_N$ might take infinite values (see \cref{assum:minimal} for the assumption on $d_N$) by assuming that $h$ is constant equal to some $z_0\in N$. In that case, we can define the open set, hence Borel set, $N_0\coloneqq\{y\in N: d_N(z_0,y)<\infty\}$ in $N$ on which $d_N$ is finite. Then, any mapping in $L^p_{\hat{h}}(I\times M,N)$ takes its values $\lebesgue\otimes \mu_M$-a.e.~in $N_0$ and $L^p_{\hat{h}}(I\times M,N)$ can be identified with $L^p_{\hat{h}}(I\times M,N_0)$. Hence,  when $h$ is constant, the extension to the case where $d_N$ might take infinite values reduces to \cref{th:extension_embed_iterated_simple} by replacing $N$ with $N_0$.
\end{remark}

\begin{proof}[Proof of \cref{th:extension_embed_iterated_simple}]
    Let $h\in L^0(M,N)$ and $p\in [1,\infty)$.

    \begin{enumerate}[label=(\roman*),wide]
    \item \emph{Step 1: Extension of $\operatorname{sec}_{I,p}$.} Recall that $\operatorname{sec}_{I,p}$ is an isometry from $(E_{\hat{h},\sigma}(I\times M,N)\cap L^p_{\hat{h}}(I\times M,N),\hat{D}_p)$ to $(E(I,E_h(M,N)\cap L^p_h(M,N)),d_{p,p})$. Now, $L^p_{\hat{h}}(I\times M,N)$ being the completion of $E_{\hat{h},\sigma}(I\times M,N)\cap L^p_{\hat{h}}(I\times M,N)$, by (i) of \cref{prop:density_iterated_almost_simple} and \cref{prop:completeness_lebesgue}, and $L^p(I,L^p_h(M,N))$ being the completion of $E(I,E_h(M,N)\cap L^p_h(M,N))$, by (ii) of \cref{prop:density_iterated_almost_simple} and \cref{prop:completeness_lebesgue}, we get that $\operatorname{sec}_{I,p}$ extends uniquely to an isometry $\overline{\operatorname{sec}}_{I,p}$ from $L^p_{\hat{h}}(I\times M,N)$ to $L^p(I,L^p_h(M,N))$ by setting $\overline{\operatorname{sec}}_{I,p}(\hat{c}) \coloneqq \operatorname{sec}_{I,p}(\hat{c})$ for all $\hat{c}\in E_{\hat{h},\sigma}(I\times M,N)\cap L^p_{\hat{h}}(I\times M,N)$ and $\overline{\operatorname{sec}}_{I,p}(\hat{c}) \coloneqq \lim_{n\to \infty} \operatorname{sec}_{I,p}(\hat{c}_n)$ for all $\hat{c}\in L^p_{\hat{h}}(I\times M,N)\setminus (E_{\hat{h},\sigma}(I\times M,N)\cap L^p_{\hat{h}}(I\times M,N))$ with $(\hat{c}_n)_{n\in\mathbb{N}}$ being any sequence in $E_{\hat{h},\sigma}(I\times M,N)\cap L^p_{\hat{h}}(I\times M,N)$ converging to $\hat{c}$, since, $\operatorname{sec}_{I,p}$ being an isometry, the choice of sequence does not matter.
    
    \par\medskip \noindent \emph{Step 2: Explicit expression of the extension of $\operatorname{sec}_{I,p}$.} Let $\hat{c}\in L^p_{\hat{h}}(I\times M,N)$. Then, there is $\hat{c}'\in \mathcal{L}^p_{\hat{h}}(I\times M, N)$ such that $\hat{D}_p(\hat{c},\hat{c}')=0$. By (i) of \cref{prop:density_iterated_almost_simple}, there is a sequence $(\hat{c}_n')_{n\in\mathbb{N}}$ in $\mathcal{E}_{\hat{h},\sigma}(I\times M,N)\cap \mathcal{L}^p_{\hat{h}}(I\times M,N)$ such that $\hat{D}_p(\hat{c}_n',\hat{c}')\to 0$ as $n\to \infty$ and define $\hat{c}_n\coloneqq[\hat{c}_n']_{I\times M}$ as well as $c_n\coloneqq \operatorname{sec}_{I,p}(\hat{c}_n)$. By definition of the $\hat{c}_n'$'s (\cref{def:rectangle_almost_simple}), the union of the sets on which the $\hat{c}_n'$'s differ from $\hat{h}$ is contained in the countable union of a collection $(I_m\times M_m)_{m\in \mathbb{N}}$ of measurable rectangles of finite $\lebesgue\otimes \mu_M$-measure. Thus, the set $\hat{M}\coloneqq\cup_{m\in\mathbb{N}} M_m$ has $\sigma$-finite $\mu_M$-measure. In addition, since $\hat{D}_p(\hat{c}_n',\hat{c}')\to 0$ as $n\to \infty$, we get that
    $$\int_{I\times (M\setminus \hat{M})} d_N(\hat{h}(t,x),\hat{c}'(t,x))^p\dif\lebesgue\otimes \mu_M(t,x) = 0,$$
    so that, up to the modification on a $\lebesgue\otimes \mu_M$-null set, we can assume that $\hat{c}'(t,x) = \hat{h}(t,x)$ for all $(t,x)\in I\times (M\setminus \hat{M})$. Now, define $c: t\mapsto \hat{c}'(t,\cdot)$ for which we know that $c(t)\in \cLs(M,N)$ for all $t\in I$. Then, since $I$ has finite $\lebesgue$-measure and $\hat{M}$ has $\sigma$-finite $\mu_M$-measure, we have, by the Fubini--Tonelli theorem \cite[Proposition 5.2.1]{cohn2013measure}, that $d_{p,p}(c_n,c)=\hat{D}_p(\hat{c}_n',\hat{c}')\to 0$ as $n\to \infty$, so that, up to the extraction of a subsequence, we can assume for $\lebesgue$-a.e.~$t\in I$ that $D_p(c_n(t),c(t))\to 0$ as $n\to \infty$. Finally, using the triangle inequality and the definition of $\operatorname{sec}_{I,p}$, we get that for $\lebesgue$-a.e.~$t\in I$ 
    $$D_p(\overline{\operatorname{sec}}_{I,p}(\hat{c})(t),\hat{c}'(t,\cdot))\leq D_p(\overline{\operatorname{sec}}_{I,p}(\hat{c})(t),\operatorname{sec}_{I,p}(\hat{c}_n)(t)) + D_p(c_n(t),c(t))$$
    holds for all $n\in\mathbb{N}$. Thus, using the definition of $\overline{\operatorname{sec}}_{I,p}$ and up to the further extraction of subsequence to obtain for $\lebesgue$-a.e.~$t\in I$ that $D_p(\overline{\operatorname{sec}}_{I,p}(\hat{c})(t),\operatorname{sec}_{I,p}(\hat{c}_n)(t))\to 0$ as $n\to \infty$, we get that $$D_p(\overline{\operatorname{sec}}_{I,p}(\hat{c})(t),[\hat{c}'(t,\cdot)]_M)= 0\quad \text{for $\lebesgue$-a.e.~$t\in I$}.$$
    \item The proof follows from the same arguments as the proof of (i), but using (iii) of \cref{prop:density_iterated_almost_simple} instead of (ii) of \cref{prop:density_iterated_almost_simple}.\qedhere
    \end{enumerate}
\end{proof}

\begin{remark}[On relaxing the assumptions on $I$]
\label{rem:relax_extension}
Since the proof of \cref{th:extension_embed_iterated_simple} only relies on the additional use of \cref{prop:density_iterated_almost_simple} relative to \cref{prop:embed_iterated_simple} and we can define, if necessary, $\hat{I}\coloneqq \cup_{m\in\mathbb{N}} I_m$ to use the Fubini--Tonelli theorem in step 2 of the proof of assertion (i) of \cref{th:extension_embed_iterated_simple}, the relaxations of \cref{prop:embed_iterated_simple,prop:density_iterated_almost_simple} given by \cref{rem:relax_isom_simple,rem:relax_density} ensure that \cref{th:extension_embed_iterated_simple} holds when $(I,\Sigma_I,\mu_I)$ is just a measure space and either $\mu_I$ is finite or $h$ is simple while replacing $L^p_{\underline{h}}(M,L^p(I,N))$ with $L^p_{\underline{h}}(M,L^0(I,N))$ (see \cref{rem:relax_isom_simple} for a definition of this space) in assertion (ii) of \cref{th:extension_embed_iterated_simple} for that last case and recalling that $(L^0(I,N),d_p)$ is complete when $N$ is (\cref{cor:completeness_measurable}). If $h$ is not only simple, but constant, the assertion (ii) of \cref{th:extension_embed_iterated_simple} holds as is. The assertion (i) of \cref{th:extension_embed_iterated_simple} even holds when $(I,\Sigma_I,\mu_I)$ is just a measure space, with no further assumptions.
\end{remark}

\begin{remark}[Related results in the literature]
\label{rem:related_extension}
    \cref{th:extension_embed_iterated_simple} is a well-known result for linear Lebesgue spaces in the case where $N$ is a Banach space, $h$ is constant, and both $\mu_I$ (see \cref{rem:relax_extension} for the definition of this measure) and $\mu_M$ are assumed $\sigma$-finite \cite[Proposition 1.2.24]{hytonen2016analysis}. When $(I,\Sigma_I,\mu_I)$ is just a measure space, with no further assumptions, $d_N$ is non-finite and $h$ is constant (\cref{rem:metric_nonfinite,rem:relax_extension}), \cref{th:extension_embed_iterated_simple} is actually a nonlinear version of the Fubini--Lebesgue theorem that requires neither that $\mu_I$ and $\mu_M$ are both $\sigma$-finite \cite[Theorem 5.2.2]{cohn2013measure}, nor that they are both complete \cite[Chapter II, Theorem 14.1]{dibenedetto2002real}. 
\end{remark}

As a direct consequence of \cref{th:extension_embed_iterated_simple}, the isometry defined in \cref{cor:isom_iterated_simple} extends uniquely to an isometry yielding a characterization of $L^p$ curves in nonlinear Lebesgue spaces.

\begin{theorem}[Characterization of $L^p$ curves in nonlinear Lebesgue spaces]
\label{th:isom_lebesgue}
    Let $h\in L^0(M,N)$ and $p\in [1,\infty)$. Suppose that $N$ is complete. Then, the isometry $i_p$ defined in \cref{cor:isom_iterated_simple} extends uniquely to the isometry
    $$\bar{i}_p\coloneqq \overline{\operatorname{sec}}_{M,p}\circ \overline{\operatorname{sec}}_{I,p}^{-1}: L^p(I,\Lph(M,N)) \to L^p_{\underline{h}}(M,L^p(I,N)).$$
\end{theorem}

\begin{remark}[On relaxing the finiteness assumption on $d_N$]
Thanks to \cref{rem:metric_nonfinite}, the conclusion of \cref{th:isom_lebesgue} holds when $d_N$ is not assumed finite (see \cref{assum:minimal} for the assumption on $d_N$), but  requires $h$ to be constant in that case.
\end{remark}

\begin{remark}[On relaxing the assumptions on $I$]
    Thanks to \cref{rem:relax_extension}, the conclusion of \cref{th:isom_lebesgue} holds when $(I,\Sigma_I,\mu_I)$ is just a measure space such that either $\mu_I$ is finite or $h$ is simple while replacing $L^p_{\underline{h}}(M,L^p(I,N))$ with $L^p_{\underline{h}}(M,L^0(I,N))$ (see \cref{rem:relax_isom_simple} for a definition of this space) in that last case. If $h$ is not only simple, but constant, the conclusion of \cref{th:isom_lebesgue} holds as is.
\end{remark}

\begin{remark}[Related results in the literature]
    This result is well-known in the context of linear Lebesgue spaces where $N$ is a Banach space and $h$ is constant \cite[Corollary 1.2.23]{hytonen2016analysis}.
\end{remark}

Let us end this section with a discussion of an edge case to better understand the scope of application of \cref{th:extension_embed_iterated_simple,th:isom_lebesgue}.

\begin{example}[Edge case for \cref{th:extension_embed_iterated_simple,th:isom_lebesgue}]
    One sometimes encounters the claim that the Fubini--Lebesgue theorem, like the Fubini--Tonelli theorem, requires the measures to be $\sigma$-finite. In this context, the indicator function of the diagonal with the Lebesgue and counting measures is frequently mentioned as an edge case \cite[Remark 1.2.8]{hytonen2016analysis}. In fact, this example contradicts neither \cref{th:extension_embed_iterated_simple} nor \cref{th:isom_lebesgue}, but remains an interesting edge case to better understand the scope of application of \cref{th:extension_embed_iterated_simple,th:isom_lebesgue}. Indeed, suppose that $I=M=[0,1]$ with its standard topology and its Borel $\sigma$-algebra, that $\mu_I= \lebesgue|_I$, that $\mu_M=\#$, that $N=\mathbb{R}$ with its standard metric and that $h\equiv 0$. Also, let $D\coloneqq \{(x,x): x\in [0,1]\}$ be the diagonal, which is a Borel set, and let $f\coloneqq \mathds{1}_D$. Then, we have  
    $$\int_{I}\left(\int_{M} \mathds{1}_D(t,x) \dif\mu_M(x)\right)\dif\mu_I(t) = \int_{I}\left(\int_{M} \mathds{1}_{\{t\}}(x) \dif\mu_M(x)\right)\dif\mu_I(t) = \int_{[0,1]} \#(\{t\})\dif \lebesgue_I(t) = 1$$
    and 
    $$\int_{M}\left(\int_{I} \mathds{1}_D(t,x) \dif\mu_I(t)\right)\dif\mu_M(x) = \int_{M}\left(\int_{I} \mathds{1}_{\{x\}}(t) \dif\mu_I(t)\right)\dif\mu_M(x) =\int_{[0,1]} \lebesgue|_I(\{x\})\dif \#(x) = 0.$$
    Hence, we get   $$\int_{I}\left(\int_{M} \mathds{1}_D(t,x) \dif\mu_M(x)\right)\dif\mu_I(t)\neq \int_{M}\left(\int_{I} \mathds{1}_D(t,x) \dif\mu_I(t)\right)\dif\mu_M(x).$$    However, the above conclusion contradicts neither \cref{th:extension_embed_iterated_simple} nor \cref{th:isom_lebesgue}. Indeed, the diagonal $D$ has infinite $\mu_I\otimes \mu_M$-measure, as, otherwise, $D$ would be contained in a countable union of measurable rectangles $(A_n\times B_n)_{n\in\mathbb{N}}$ of finite $\mu_I\otimes \mu_M$-measure \cite[Chapter II, Proposition 10.2]{dibenedetto2002real}. For every $n\in\mathbb{N}$, the set $A_n \cap B_n$ has null Lebesgue measure: if $B_n$ is finite, $A_n \cap B_n$ is finite, while if $B_n$ is infinite, the finiteness of $\mu_I\otimes \mu_M(A_n\times B_n) = \lebesgue(A_n)\#(B_n)$ forces $\lebesgue(A_n)=0$. Since the rectangles cover $D$, we would have $[0,1]\subset \cup_{n\in\mathbb{N}} (A_n\cap B_n)$, which is impossible. Thus, the equivalence class of the indicator function of $D$ does not belong to $L^1(I\times M,\mathbb{R})$ and neither $\overline{\operatorname{sec}}_{I,p}$ nor $\overline{\operatorname{sec}}_{M,p}$ is defined on $f$ (\cref{th:extension_embed_iterated_simple}). Furthermore, $f_I:t\mapsto f(t,\cdot)=\mathds{1}_{\{t\}}$ does not have a separable range in $L^1(M,\mathbb{R})$ since the family $\{\mathds{1}_{\{t\}}: t\in [0,1]\}$ is uncountable and $\int_M \lvert \mathds{1}_{\{t\}}(x) - \mathds{1}_{\{s\}}(x)\rvert \dif\mu_M(x) = 2> 0$ as soon as $s\neq t$, so that $\mathds{1}_{\{t\}}$ and $\mathds{1}_{\{s\}}$ cannot belong to the same equivalence class. Also, $f_I$ does not admit an equivalent mapping with separable range in $L^1(M,\mathbb{R})$ since the counting measure imposes a pointwise equality on $M$ for the values of $f_I$, so it cannot be approximated by simple mappings and $\overline{\operatorname{sec}}_{I,p}^{-1}$ is actually not defined on the equivalence class of $f_I$ (\cref{th:extension_embed_iterated_simple}). On the contrary, $f_M:x\mapsto f(\cdot,x)=\mathds{1}_{\{x\}}$ does satisfy the conditions of \cref{th:extension_embed_iterated_simple,th:isom_lebesgue}. Indeed, $D_{p,p}(f_M,\underline{h})=0$, so that $f_M$ is assimilated to the equivalence class of $\underline{h}$ whose image by $\overline{\operatorname{sec}}_{M,p}^{-1}$ is the equivalence class of the mapping constant equal to $0$ on $I\times M$ and by $\bar{i}_p^{-1}$ is the equivalence class of the mapping constant equal to $h$ on $I$. The mappings $f$ and $f_I$ are thus in fact not related to $f_M$ through the section mappings on $I$ and $M$ defined in \cref{th:extension_embed_iterated_simple}.
\end{example}

We now continue with the characterization of curves in nonlinear Lebesgue spaces using the isometry defined in \cref{th:isom_lebesgue}.

\subsection{\texorpdfstring{Characterization of absolutely continuous curves for $p> 1$}{Characterization of absolutely continuous curves for p> 1}}
\label{sec:charact_ac}

We first clarify some recurrent notation used in this section and thereafter.
\begin{notation}
\label{not:ac}
    Let $h\in L^0(M,N)$ and $p\in (1,\infty)$. From thereon:
    \begin{enumerate}[label=(\roman*)]
        \item $L^p_{\underline{h}}(M,\mathcal{AC}^p(I,N))$ denotes the set of equivalence classes of separably valued measurable mappings from $M$ to $(\mathcal{AC}^p(I,N),d_p)$ at a finite $D_{p,p}$-distance from $\underline{h}$ (see \cref{not:lp} for the definition of $\underline{h}$).
        \item $\lvert c'\rvert_p\coloneqq \lvert c'\rvert_{\Lph(M,N)}$ denotes the metric derivative of the curve $c\in \mathcal{AC}^p(I,\Lph(M,N))$ (\cref{th:metric_derivative}).
        \item $Q_p$ denotes the restriction to $\mathcal{AC}^p(I,L^p_h(M,N))$ of the quotient mapping 
$$Q_I: \left\{\begin{array}{ccl}
    \mathcal{L}^0_s(I,L^p_h(M,N))&\longrightarrow& L^0(I,L^p_h(M,N))\\
    c&\longmapsto& {[c]_I}
\end{array}.\right.$$
\end{enumerate}
\end{notation}

Then, the isometry given by \cref{th:isom_lebesgue} yields a characterization of absolutely continuous curves in nonlinear Lebesgue spaces for $p > 1$.

\begin{theorem}[Characterization of absolutely continuous curves for $p> 1$]
\label{th:charact_ac_curves}
    Let $h\in L^0(M,N)$ and $p\in (1,\infty)$. Suppose that $N$ is complete. Then, the isometry $\bar{i}_p$ defined in \cref{th:isom_lebesgue} 
    induces the isometric mapping
    \begin{align*}
        \tilde{i}_p\coloneqq (q_p^{-1})_* \circ \bar{i}_p\circ Q_p: (\mathcal{AC}^p(I,\Lph(M,N)),d_{p,p})\to (L^p_{\underline{h}}(M,\mathcal{AC}^p(I,N)),D_{p,p})
    \end{align*}
    whose inverse mapping on its range is $\tilde{i}_p^{-1}\coloneqq Q_p^{-1}\circ \bar{i}_p^{-1}\circ (q_p)_*$ and satisfies 
    \begin{align}
    \label{eq:inverse_map_evaluation_map}
    \tilde{i}_p^{-1}(f)(t) = (e_{t,p})_*(f)\quad\text{for all $t\in I$  and $f\in \tilde{i}_p( \mathcal{AC}^p(I,L^p_h(M,N)))$.}
    \end{align}
    In addition, $\tilde{i}_p$ satisfies for all $c\in \mathcal{AC}^p(I,\Lph(M,N))$ that
    \begin{align}
    \label{eq:derivative_equals_pw_derivative}
        \lvert c'\rvert_p^p(t) = \int_M \lvert \tilde{i}_p(c)(x)'\rvert_N^p(t)\dif \mu_M(x)\quad\text{for $\lebesgue$-a.e.~$t\in I$}
    \end{align}
    and its range is the set of all $f\in L^p_{\underline{h}}(M,\mathcal{AC}^p(I,N))$ such that
    $$\int_M \int_I \lvert f(x)'\rvert_N^p(t)\dif t\dif\mu_M(x) < \infty.$$
    
\end{theorem}

\begin{remark}[On the definition of pushforwards for equivalence classes]
\label{rem:pushforward_equiv}
    The pushforward of an equivalence class by a mapping is meant as the equivalence class of the composition of the given mapping with a compatible representative of the input equivalence class. 
\end{remark}

\begin{remark}[On relaxing the finiteness assumption on $d_N$]
\label{rem:relax_metric_charact_ac}
The conclusion of \cref{th:charact_ac_curves} holds when $d_N$ is not assumed finite (see \cref{assum:minimal} for the assumption on $d_N$) by assuming that $h$ is constant equal to some $z_0\in N$. In that case, we can define the open set, hence Borel set, $N_0\coloneqq\{y\in N: d_N(z_0,y)<\infty\}$ in $N$ on which $d_N$ is finite. Then, any mapping in $\Lph(M,N)$ takes its values $\mu_M$-a.e.~in $N_0$ and $\Lph(M,N)$ can be identified with $\Lph(M,N_0)$. Hence,  when $h$ is constant, the extension to the case where $d_N$ might take infinite values reduces to \cref{th:charact_ac_curves} by replacing $N$ with $N_0$.
\end{remark}

\begin{proof}[Proof of \cref{th:charact_ac_curves}]
 The proof follows the same approach used in \cite[Theorems 4 and 5]{lisini2007characterization} for Wasserstein spaces, which consists in using piecewise constant approximations of absolutely continuous curves to prove inequalities and propagate the results to absolutely continuous curves by passing to the limit.
 
 \noindent Let $h\in \cLs(M,N)$, $p\in (1,\infty)$. 

 \par\medskip \noindent \emph{Step 1: Reduction to the case where $N$ is separable.} Let us first reduce the proof to the case where $N$ is separable. Indeed, when $c\in \mathcal{AC}^p(I,\Lph(M,N))$, the set $N_0\coloneqq \overline{h(M)\cup(\cup_{t\in \mathbb{Q}\cap I}c(t)(M))}$ is separable as the closure of the countable union of separable sets. Then, since $c$ is continuous, we have for all $t\in I$ that $c(t)= \lim_{s\to t,s\in \mathbb{Q}\cap I} c(s)$ in $\Lph(M,N)$, hence, up to the extraction of a subsequence, $c(t)(x) = \lim_{s\to t, s\in \mathbb{Q}\cap I} c(s)(x)$ in $N$ for $\mu_M$-a.e.~$x\in M$, that is, $c(t)(x)\in N_0$ for $\mu_M$-a.e.~$x\in M$. Thus, we can replace $N$ with $N_0$ in the rest of the proof, so that $\cL(M,N) = \cLs(M,N)$.

 \par\medskip \noindent \emph{Step 2: Definition of $\tilde{i}_p$.} Let $c\in \mathcal{AC}^p(I,L^p_h(M,N))$. Then, $Q_p(c)$ belongs to $W^{1,p}(I,\Lph(M,N))$, by \cref{prop:ac_representation_sobolev}, hence $(\bar{i}_p \circ Q_p)(c)$ belongs to $L^p_{\underline{h}}(M,L^p(I,N))$, by \cref{th:isom_lebesgue}. Define for all $n\in \mathbb{N}^*$ the set $I_n\coloneqq \{0,1,\ldots, 2^n\}$ and the sequence $(t_{n,i})_{i\in I_n}$ such that $t_{n,i}\coloneqq a+ \lebesgue(I)i2^{-n}$. Also, define the simple mapping $c_n: I \to \Lph(M,N)$ as $c_n|_{[t_{n,i},t_{n,i+1})} \equiv c(t_{n,i})$ for all $i\in I_{n-1}$ and $c_n(b)\coloneqq c(b)$. Since the set of dyadic rationals (scaled by $\lebesgue(I)$ and translated by $a$) is dense in $I$ and $c$ is continuous on $[a,b)$ and $\lebesgue(\{b\})=0$, we have for $\lebesgue$-a.e.~$t\in I$ that $c_n(t)\to c(t)$ in $\Lph(M,N)$ as $n\to \infty$. In addition, if $t\in [t_{n,i},t_{n,i+1})$,
 \begin{align*}
     D_p(c_n(t),c(t)) = D_p(c(t_{n,i}),c(t)) \leq \int_{t_{n,i}}^{t}\lvert c'\vert_p(u)\dif u \leq \int_I \lvert c'\rvert_p(u)\dif u. 
 \end{align*}
 Hence, we get, using Hölder's inequality, that 
 $$D_p(c_n(t),c(t))^p\leq \lebesgue(I)^{p-1}\int_I \lvert c'\rvert_p^p(u)\dif u\quad\text{for all $(t,n)\in I\times \mathbb{N}^*$}.$$
 Since $I$ has finite $\lebesgue$-measure, the mapping constant equal to $\lebesgue(I)^{p-1}\int_I \lvert c'\rvert_p^p(u)\dif u <\infty$ on $I$ is integrable, so that, by Lebesgue's dominated convergence theorem \cite[Theorem~2.4.5]{cohn2013measure}, we have $d_{p,p}(c_n,c) \to 0$ as $n\to \infty$. In addition, for all $n\in\mathbb{N}^*$ and all $i\in I_n$ there is $c_{n,i}\in \cLph(M,N)$ such that $D_p(c_{n,i},c(t_{n,i}))=0$, so we can define $c_n': I\to \cLph(M,N)$ as $c_n'|_{[t_{n,i},t_{n,i+1})} \equiv c_{n,i}$ for all $i\in I_{n-1}$ and $c_n'(b)\coloneqq c_{n,2^n}$, which satisfies $D_p(c_n(t),c_n'(t))=0$ for all $t\in I$. Thus, we can define $\hat{c}_n':I\times M\to N$ as $\hat{c}_n'(t,x)\coloneqq c_n'(t)(x)$, which belongs to $\mathcal{L}^p_{\hat{h}}(I\times M,N)$ (simple on $I$ and by definition of the $c_{n,i}$'s) and satisfies $D_p(\hat{c}_n'(t,\cdot),c_n(t))=0$ for all $t\in I$. Hence, $\overline{\operatorname{sec}}_{I,p}$ being an isometry (\cref{th:extension_embed_iterated_simple}), we get that $\hat{D}_p(\hat{c}_n',\overline{\operatorname{sec}}_{I,p}^{-1}(c_n))=0$. Now, defining the mapping $f_n': M\to \mathcal{L}^p(I,N)$ as $f_n'(x)\coloneqq \hat{c}_n'(\cdot,x)$ and recalling that $\overline{\operatorname{sec}}_{M,p}$ is an isometry (\cref{th:extension_embed_iterated_simple}), we get that $D_{p,p}(f_n',\bar{i}_p(c_n))=0$, hence that $d_p(f_n'(x),\bar{i}_p(c_n)(x))=0$ for $\mu_M$-a.e.~$x\in M$. Then, define for all $n\in \mathbb{N}^*$ the functional $\psi_n: L^p(I,N)\to [0,\infty]$ as 
    $$\psi_n(\eta)\coloneqq\sup_{\lebesgue(I) 2^{-n} < \tau < b-a} \int_a^{b-\tau} (\Delta_{N,\tau} \eta(t))^p\dif t, $$
   so that $\psi_n(\eta) \leq \psi_{n+1}(\eta)$ for all $(n,\eta) \in \mathbb{N}^*\times L^p(I,N)$ and let $\psi_\infty(\eta) \coloneqq \lim_{n\to \infty} \psi_n(\eta)$, which is lower semicontinuous.
    Now, observe that for all $(m,n)\in (\mathbb{N}^*)^2$ such that $m\leq n$ we have for $\mu_M$-a.e.~$x\in M$ that
    \begin{align*}
        \psi_m(\bar{i}_p(c_n)(x)) &= \sup_{\lebesgue(I) 2^{-m} < \tau < b-a} \int_a^{b-\tau} (\Delta_{N,\tau} \bar{i}_p(c_n)(x)(t))^p\dif t \\
        &\leq \sup_{\lebesgue(I) 2^{-n} < \tau < b-a} \int_a^{b-\tau} (\Delta_{N,\tau} \bar{i}_p(c_n)(x)(t))^p\dif t\\
        &= \sup_{\lebesgue(I) 2^{-n} < \tau < b-a} \int_a^{b-\tau} (\Delta_{N,\tau} f_n'(x)(t))^p\dif t.
    \end{align*}
    Then, take the integer $k\geq 1$ such that $k\lebesgue(I)2^{-n} \leq \tau \leq (k+1)\lebesgue(I) 2^{-n}$ and define $u_{n,i}(t) \coloneqq t+ \lebesgue(I)i2^{-n}$, so that, using the triangle inequality, we get 
    \begin{align*}d_N(f_n'(x)(t),f_n'(x)(t+\tau)) &\leq \sum_{i=0}^k d_N(f_n'(x)(u_{n,i}(t)), f_n'(x)(u_{n,i+1}(t)))\\
    &= \sum_{i=0}^k d_N(c_n'(u_{n,i}(t))(x), c_n'(u_{n,i+1}(t))(x)).
    \end{align*}
    Hence, Hölder's discrete inequality yields that 
    $$d_N(f_n'(x)(t),f_n'(x)(t+\tau))^p \leq (k+1)^{p-1}\sum_{i=0}^k d_N(c_n'(u_{n,i}(t))(x), c_n'(u_{n,i+1}(t))(x))^p.$$
    Thus, we get 
    \begin{align*}
        \int_a^{b-\tau} (\Delta_{N,\tau}f_n'(x)(t))^p\dif t &\leq \tau^{-p}\int_a^{b-k\lebesgue(I)2^{-n}} d_N(f_n'(x)(t),f_n'(x)(t+\tau))^p\dif t\\
        &\leq \tau^{-p}\int_a^{a+(2^n -k)\lebesgue(I)2^{-n}} (k+1)^{p-1}\sum_{i=0}^k d_N(c_n'(u_{n,i}(t))(x), c_n'(u_{n,i+1}(t))(x))^p \dif t  \\
        &= \tau^{-p}(k+1)^{p-1}\sum_{i=0}^k \lebesgue(I)2^{-n}\sum_{j=0}^{2^n - k - 1} d_N(c_{n,i+j}(x), c_{n,i+j+1}(x))^p.
    \end{align*}
    Since $d_N(c_{n,i+j}(x), c_{n,i+j+1}(x))^p$ is counted at most $k+1$ times, we get 
    \begin{align*}
        \int_a^{b-\tau} (\Delta_{N,\tau}f_n'(x)(t))^p\dif t &\leq (\tfrac{k+1}{\tau})^p\lebesgue(I)2^{-n}\sum_{i=0}^{2^n-1} d_N(c_{n,i}(x), c_{n,i+1}(x))^p.
    \end{align*}
    The latter inequality together with the fact that $\lebesgue(I)2^{-n}\leq (\lebesgue(I)2^{-n})^{1-p}\frac{\tau^p}{k^p}$ yields that 
    \begin{align*}
    \psi_m(\bar{i}_p(c_n)(x)) &\leq\left(\frac{k+1}{k}\right)^p (\lebesgue(I)2^{-n})^{1-p}\sum_{i=0}^{2^n-1} d_N(c_{n,i}(x), c_{n,i+1}(x))^p\\
        &\leq2^p(\lebesgue(I)2^{-n})^{1-p}\sum_{i=0}^{2^n-1} d_N(c_{n,i}(x), c_{n,i+1}(x))^p\\
        &= 2^p(\lebesgue(I)2^{-n})^{1-p}\sum_{i=0}^{2^n-1} d_N(c_{n,i}(x), c_{n,i+1}(x))^p.
    \end{align*}
    Then, integrating over $M$, we get for all $(m,n)\in (\mathbb{N}^*)^2$ satisfying $m\leq n$ that
    \begin{align*}
        \int_M \psi_m(\bar{i}_p(c_n)(x)) \dif\mu_M(x)& \leq 2^p(\lebesgue(I)2^{-n})^{1-p}\int_M\sum_{i=0}^{2^n-1} d_N(c_{n,i}(x), c_{n,i+1}(x))^p\dif\mu_M(x)\\
        &= 2^p(\lebesgue(I)2^{-n})^{1-p} \sum_{i=0}^{2^n-1} D_p(c(t_{n,i}), c(t_{n,i+1}))^p\\
        &\leq 2^p(\lebesgue(I)2^{-n})^{1-p} \sum_{i=0}^{2^n-1} \left(\int_{t_{n,i}}^{t_{n,i+1}} \lvert c'\vert_p(t)\dif t\right)^p\\
        &\leq 2^p \sum_{i=0}^{2^n-1} \int_{t_{n,i}}^{t_{n,i+1}} \lvert c'\vert_p^p(t)\dif t\\
        &=2^p  \int_I \lvert c'\rvert_p^p(t)\dif t,
    \end{align*}
    where Hölder's inequality was used in the second to last inequality. Now, $\bar{i}_p$ being isometric and using the definition of $Q_p$ (\cref{not:ac}), we also have $D_{p,p}(\bar{i}_p(c_n),(\bar{i}_p\circ Q_p)(c))\to 0$ as $n\to \infty$. Then, up to the extraction of a subsequence, we can assume for $\mu_M$-a.e.~$x\in M$  that $\bar{i}_p(c_n)(x)\to (\bar{i}_p \circ Q_p)(c)(x)$ in $L^p(I,N)$ as $n\to\infty $, so that, by lower semicontinuity of $\psi_m$ for all $m\in \mathbb{N}^*$ and using Fatou's lemma \cite[Theorem 2.4.4]{cohn2013measure}, we get for all $m\in \mathbb{N}^*$ that 
    \begin{align*}
        \int_M \psi_m((\bar{i}_p \circ Q_p)(c)(x)) \dif\mu_M(x) &\leq\int_M \underset{n\to \infty}{\lim\inf}\,\psi_m (\bar{i}_p(c_n)(x)) \dif\mu_M(x)\\
        &\leq\underset{n\to \infty}{\lim\inf}\int_M \psi_m (\bar{i}_p(c_n)(x)) \dif\mu_M(x)\\
        &\leq 2^p \int_I \lvert c'\rvert_p^p(t)\dif t < \infty.
    \end{align*}
    In addition, $\psi_m(\gamma)$ increases to $\psi_\infty(\gamma)$ as $m\to \infty$ for every $\gamma\in L^p(I,N)$. Thus, using the monotone convergence theorem \cite[Theorem 2.4.1]{cohn2013measure}, we get 
    $$\int_M \psi_\infty((\bar{i}_p\circ Q_p)(c)(x))\dif\mu_M(x) \leq 2^p\int_I \lvert c'\rvert_p^p(t)\dif t.$$
    The latter inequality yields that for $\mu_M$-a.e.~$x\in M$ 
    $$\psi_\infty((\bar{i}_p \circ Q_p)(c)(x)) = \sup_{0 < \tau < b-a} \int_a^{b-\tau} (\Delta_{N,\tau} (\bar{i}_p \circ Q_p)(c)(x)(t))^p\dif t <\infty,$$
    which implies that $(\bar{i}_p \circ Q_p)(c)(x)$ belongs to $W^{1,p}(I,N)$ (\cref{def:sobolev_curves}). Hence, $W^{1,p}(I,N)$ being a Borel subset of $L^p(I,N)$ (\cref{rem:sobolev_borel}), $(\bar{i}_p\circ Q_p)(c)$ belongs to $L^p_{\underline{h}}(M,W^{1,p}(I,N))$. Therefore, the isometric mapping 
 $$\tilde{i}_p\coloneqq (q_p^{-1})_* \circ \bar{i}_p \circ Q_p: \mathcal{AC}^p(I,\Lph(M,N)) \to L^p_{\underline{h}}(M,\mathcal{AC}^p(I,N))$$
 is well-defined and its inverse mapping on its range is $\tilde{i}_p^{-1}\coloneqq Q_p^{-1}\circ \bar{i}_p^{-1}\circ (q_p)_*$.

    \par\medskip \noindent \emph{Step 3: Proof of \cref{eq:inverse_map_evaluation_map}.} Let $f$ be in the range of $\tilde{i}_p$. Then, there is a unique $c\in \mathcal{AC}^p(I,L^p_h(M,N))$ such that $f= \tilde{i}_p(c)$ and define 
    $$\hat{f}\coloneqq (\overline{\operatorname{sec}}_{I,p}^{-1}\circ \bar{i}_p^{-1}\circ (q_p)_*)(f) = (\overline{\operatorname{sec}}_{M,p}^{-1}\circ  (q_p)_*)(f).$$
    Also, note that $\hat{f}= (\overline{\operatorname{sec}}_{I,p}^{-1}\circ Q_p)(c)$ since $\tilde{i}_p= (q_p^{-1})_* \circ \bar{i}_p \circ Q_p$. Now, there is $f'\in \mathcal{L}_{\underline{h}}^p(M,\mathcal{AC}^p(I,N))$ such that $D_{p,p}(f',f)= 0$, which implies that $f'(x)=f(x)$ in $\mathcal{C}(I,N)$ for $\mu_M$-a.e.~$x\in M$ (\cref{prop:ac_representation_sobolev}). Thus, we can define the mapping $\tilde{f}: I\times M\to N$ as $\tilde{f}(t,x)\coloneqq f'(x)(t)=(e_{t,p}\circ f')(x)$. To show that $\tilde{f}$ is measurable, consider for all $n\in\mathbb{N}^*$ the mapping $\tilde{f}_n: I\times M\to N$ defined as $\tilde{f}_n(t,x)\coloneqq (e_{t_{n,i},p}\circ f')(x)$ if $t\in [t_{n,i},t_{n,i+1})$ with $i\in I_{n-1}$ and $\tilde{f}_n(b,x)\coloneqq(e_{b,p}\circ f')(x)$, so that $\tilde{f}_n$ belongs to $\mathcal{L}^0(I\times M,N)$ since $e_{t,p}\circ f'\in \mathcal{L}^0(M,N)$ for all $t\in I$ (\cref{prop:measurability_evaluation_ac}). Then, the set of dyadic rationals (scaled by $\lebesgue(I)$ and translated by $a$) being dense in $I$ and $f'(x)$ being continuous on $[a,b)$ for all $x\in M$, we have $\tilde{f}(t,x) = \lim_{n\to \infty} \tilde{f}_n(t,x)$ for all $(t,x)\in I\times M$. Therefore, $\tilde{f}$ is measurable as the pointwise limit of measurable mappings (\cref{prop:measurability_pointwise_limit}). In addition, $D_p(\tilde{f}(t,\cdot),(e_{t,p})_*(f))=0$ for all $t\in I$ since $\tilde{f}(\cdot,x)= f(x)$ in $\mathcal{C}(I,N)$ for $\mu_M$-a.e.~$x\in M$. Then, using the definitions of $\hat{f}$, $\overline{\operatorname{sec}}_{M,p}$ (\cref{th:extension_embed_iterated_simple}) and $q_p$ (\cref{prop:ac_representation_sobolev}), we have $d_p(\tilde{f}(\cdot,x),\hat{f}(\cdot,x)) =0$ for $\mu_M$-a.e.~$x\in M$. Hence, $\overline{\operatorname{sec}}_{M,p}$ being an isometry, we get $\hat{D}_p(\tilde{f},\hat{f}) = 0$. Thus, using the definitions of $\hat{f}$, $\overline{\operatorname{sec}}_{I,p}$ (\cref{th:extension_embed_iterated_simple}) and $Q_p$ (\cref{not:ac}), we get $D_p(\tilde{f}(t,\cdot),c(t)) =0$ for $\lebesgue$-a.e.~$t\in I$, which yields $\tilde{i}_p^{-1}(f)(t) = (e_{t,p})_*(f)$ in $\Lph(M,N)$ for $\lebesgue$-a.e.~$t\in I$. It remains to prove the equality for all $t\in I$. To do so, define for all $\tau\neq 0$ the mappings 
    $$g_\tau:\left\{\begin{array}{ccl}
    I\times \mathcal{AC}^p(I,N)&\longrightarrow& [0,\infty)\\
    (t,u)&\longmapsto& d_N(u(t),\bar{u}(t+\tau))/\lvert \tau\rvert
    \end{array}\right.,$$
    where $\bar{u}$ is the extension of $u$ such that $\bar{u}(t)= u(a)$ if $t< a$ and $\bar{u}(t)= u(b)$ if $t > b$. Then, $g_\tau$ is continuous, hence Borel measurable, for all $\tau \neq 0$. Thus, the set 
    \begin{align*}
    A&\coloneqq \{(t,x)\in I\times M: \lvert \tilde{i}_p(c)(x)'\rvert_N(t)\text{ does not exist}\}\\
    &= \{(t,x)\in I\times M: \underset{\tau,\tau'\to 0}{\lim\,\sup}\, \lvert g_\tau(t,\tilde{i}_p(c)(x)) - g_{\tau'}(t,\tilde{i}_p(c)(x))\rvert > 0\}
    \end{align*}
    is measurable since $g_\tau$ is Borel measurable and $\tilde{i}_p(c)$ is measurable from $M$ to $(\mathcal{AC}^p(I,N),d_\infty)$ (\cref{prop:comparison_dp_dinf}), so that $(t,x)\mapsto g_\tau(t,\tilde{i}_p(c)(x))$ is measurable from $I\times M$ to $[0,\infty)$ for every $\tau \neq 0$. In addition, the mapping $(t,x)\mapsto \lvert \tilde{i}_p(c)(x)'\rvert_N(t)$, set to $0$ on $A$, is measurable from $I\times M$ to $[0,\infty)$, as the pointwise limit of measurable mappings (\cref{prop:measurability_pointwise_limit}) since $\lvert \tilde{i}_p(c)(x)'\rvert_N(t) = \lim_{\tau\to 0}g_\tau(t,\tilde{i}_p(c)(x))$ for $(t,x)\in (I\times M)\setminus A$. Furthermore, $\lebesgue(A^x) =0$ since $\tilde{i}_p(c)(x)\in \mathcal{AC}^p(I,N)$ for $\mu_M$-a.e.~$x\in M$. Then, recalling that $d_p$ separates elements of $\mathcal{AC}^p(I,N)$ (\cref{prop:ac_representation_sobolev}) and defining the measurable set  
    $$M_c\coloneqq \{x\in M: d_p(\tilde{i}_p(c)(x),\underline{h}(x)) >0\} = \{x\in M: d_\infty(\tilde{i}_p(c)(x),\underline{h}(x)) >0\}$$
    of $\sigma$-finite $\mu_M$-measure, we have $\lvert \tilde{i}_p(c)(x)'\rvert_N(t) = 0$ for all $(x,t)\in (M\setminus M_c)\times I$ since $\underline{h}(x)$ is constant over $I$. Thus, $A$ is a subset of the measurable rectangle $I\times M_c$ and, $M_c$ having $\sigma$-finite $\mu_M$-measure and $I$ having finite $\lebesgue$-measure, we get, by the Fubini--Tonelli theorem \cite[Proposition 5.2.1]{cohn2013measure}, that $\lebesgue\otimes \mu_M(A)=0$ and $\mu_M(A_t)=0$ for $\lebesgue$-a.e.~$t\in I$. Hence, for all $(s,t)\in I^2$ such that $s< t$, we have
    \begin{align*}
        D_p((e_{s,p})_*(\tilde{i}_p(c)), (e_{t,p})_*(\tilde{i}_p(c)))^p &= \int_M d_N(\tilde{i}_p(c)(x)(s), \tilde{i}_p(c)(x)(t))^p \dif\mu_M(x)\\
        &\leq \int_M \left( \int_s^t\lvert \tilde{i}_p(c)(x)'\rvert_N(t)\dif t\right)^p\dif\mu_M(x)\\
        &\leq \lvert t - s\rvert^{p-1}\int_M \int_s^t\lvert \tilde{i}_p(c)(x)'\rvert_N^p(t)\dif t\dif\mu_M(x)\\
        &= \lvert t - s\rvert^{p-1}\int_{M_c} \int_s^t\lvert \tilde{i}_p(c)(x)'\rvert_N^p(t)\dif t\dif\mu_M(x)\\
        &= \lvert t - s\rvert^{p-1}\int_s^t\int_{M_c} \lvert \tilde{i}_p(c)(x)'\rvert_N^p(t)\dif\mu_M(x)\dif t\\
        &= \lvert t - s\rvert^{p-1}\int_s^t\int_M \lvert \tilde{i}_p(c)(x)'\rvert_N^p(t)\dif\mu_M(x)\dif t,
    \end{align*}
    where the second to last equality follows from the Fubini--Tonelli theorem \cite[Proposition 5.2.1]{cohn2013measure}. The latter inequality yields that $t\mapsto (e_{t,p})_*(\tilde{i}_p(c))$ belongs to $\mathcal{AC}^p(I,\Lph(M,N))$, hence, $d_{p,p}$ separating the elements of $\mathcal{AC}^p(I,\Lph(M,N))$ (\cref{prop:ac_representation_sobolev}), we get $\tilde{i}_p^{-1}(f)(t) = (e_{t,p})_*(f)$ in $\Lph(M,N)$ for all $t\in I$, which ends the proof of \cref{eq:inverse_map_evaluation_map}.
    
    \par\medskip \noindent \emph{Step 4: Proof of \cref{eq:derivative_equals_pw_derivative}.} Let $c\in \mathcal{AC}^p(I,\Lph(M,N))$. Using Lebesgue's differentiation theorem \cite[Theorem 11.1]{dibenedetto2002real} in the previous inequality, we obtain that 
    \begin{equation}
    \label{eq:md_lt_pmd}
    \lvert c'\rvert_p^p(t) \leq \int_M \lvert \tilde{i}_p(c)(x)'\rvert_N^p(t)\dif \mu_M(x)\quad\text{for $\lebesgue$-a.e.~$t\in I$}.\end{equation}
    In addition, using Fatou's lemma \cite[Theorem 2.4.4]{cohn2013measure} and recalling that $c(t) =\tilde{i}_p^{-1}(\tilde{i}_p(c))(t) = (e_{t,p})_*(\tilde{i}_p(c))$ in $\Lph(M,N)$ for all $t\in I$, we have for all $(s,t)\in I^2 $ such that $s\leq t$
 \begin{align*}
      \int_s^t\int_M \lvert \tilde{i}_p(c)(x)'\rvert_N^p(u) \dif\mu_M(x) \dif u &= \int_s^t\int_M \underset{\tau\to 0}{\lim\inf} (\Delta_{N,\tau} \tilde{i}_p(c)(x)(u))^p \dif\mu_M(x)\dif u\\
     &\leq  \int_s^t \underset{\tau\to 0}{\lim\inf}\int_M  (\Delta_{N,\tau} \tilde{i}_p(c)(x)(u))^p \dif\mu_M(x)\dif u\\
     &=  \int_s^t \underset{\tau\to 0}{\lim\inf}\int_M  (\Delta_{N,\tau} c(u)(x))^p \dif\mu_M(x)\dif u\\
     &=  \int_s^t \lvert c'\rvert_p^p(u)\dif u.
 \end{align*}
 Therefore, using Lebesgue's differentiation theorem \cite[Theorem 11.1]{dibenedetto2002real}, we get that 
 $$\int_M \lvert \tilde{i}_p(c)(x)'\vert_p^p(t)\dif\mu_M(x)\leq \lvert c'\rvert_p^p(t)\quad \text{for $\lebesgue$-a.e.~$t\in I$}.$$
 Together with \cref{eq:md_lt_pmd}, this yields the equality
    $$\lvert c'\rvert_p^p(t) = \int_M \lvert \tilde{i}_p(c)(x)'\rvert_N^p(t)\dif \mu_M(x)\quad\text{for $\lebesgue$-a.e.~$t\in I$ },$$
    so that \cref{eq:derivative_equals_pw_derivative} is proved.

    \par\medskip \noindent \emph{Step 5: Characterization of the range of $\tilde{i}_p$}. By the previous step and the Fubini--Tonelli theorem \cite[Proposition 5.2.1]{cohn2013measure}, we have for all $c\in \mathcal{AC}^p(I,\Lph(M,N))$ that 
    $$\int_M \int_I \lvert \tilde{i}_p(c)(x)'\rvert_N^p(u)\dif u\dif\mu_M(x) \leq \int_I \lvert c'\rvert_p^p(u)\dif u <\infty.$$
    Hence, the range of $\tilde{i}_p$ is contained in the set of all $f\in L^p_{\underline{h}}(M,\mathcal{AC}^p(I,N))$ such that
    \begin{align*}
    \int_M \int_I \lvert f(x)'\rvert_N^p(t)\dif t \dif\mu_M(x)< \infty.
    \end{align*}
 Let us now prove the reverse inclusion. Thus, let $f\in L^p_{\underline{h}}(M,\mathcal{AC}^p(I,N))$ such that
    \begin{align*}
    \int_M \int_I \lvert f(x)'\rvert_N^p(t)\dif t \dif\mu_M(x)< \infty.
    \end{align*}
    We already know that $\bar{i}_p^{-1}\circ (q_p)_*(f) \in L^p(I,\Lph(M,N))$, by \cref{prop:ac_representation_sobolev,th:isom_lebesgue}, and, by \cref{th:extension_embed_iterated_simple}, that
    $$\hat{f}\coloneqq (\overline{\operatorname{sec}}_{I,p}^{-1}\circ \bar{i}_p^{-1}\circ (q_p)_*)(f) = (\overline{\operatorname{sec}}_{M,p}^{-1}\circ  (q_p)_*)(f)$$
    belongs to $L^p_{\hat{h}}(I\times M, N)$. Also, by the definitions of $\overline{\operatorname{sec}}_{M,p}$, $\overline{\operatorname{sec}}_{I,p}$ and $q_p$  (\cref{th:extension_embed_iterated_simple,prop:ac_representation_sobolev}), we have $\hat{f}(\cdot,x)=f(x)$ in $L^p(I,N)$ for $\mu_M$-a.e.~$x\in M$ and $\hat{f}(t,\cdot)=(\bar{i}_p^{-1}\circ (q_p)_*)(f)(t)$ in $\Lph(M,N)$ for $\lebesgue$-a.e.~$t\in I$. However, $f$ differs from $\underline{h}$ on a measurable set $M_f$ of $\sigma$-finite $\mu_M$-measure, so that $\hat{f}$ differs from $\hat{h}$ on a measurable set contained in the measurable rectangle $I\times M_f$. Hence, we get, using the Fubini--Tonelli theorem \cite[Proposition 5.2.1]{cohn2013measure} and Hölder's inequality, that 
    \begin{align*}\sup\limits_{0< \tau < b-a}\int_a^{b-\tau}(\Delta_{p,\tau} (\bar{i}_p^{-1}\circ (q_p)_*)(f)(u))^p\dif u &= \sup\limits_{0< \tau < b-a}\int_a^{b-\tau} \int_{M_f} (\Delta_{N,\tau} \hat{f}(u,x))^p\dif\mu_M(x)\dif u\\
    &= \sup\limits_{0< \tau < b-a} \int_{M_f} \int_a^{b-\tau} (\Delta_{N,\tau} \hat{f}(u,x))^p\dif u\dif\mu_M(x)\\
    &= \sup\limits_{0< \tau < b-a} \int_M \int_a^{b-\tau} (\Delta_{N,\tau} f(x)(u))^p\dif u\dif\mu_M(x)\\
    &\leq \int_M \sup\limits_{0< \tau < b-a}\int_a^{b-\tau}(\Delta_{N,\tau} f(x)(u))^p\dif u \dif\mu_M(x)\\
    &\leq \int_M \sup\limits_{0< \tau < b-a}\int_a^{b-\tau}\tau^{-p}\left(\int_u^{u+\tau}\lvert f(x)'\rvert_N(\theta)\dif\theta\right)^p\dif u \dif\mu_M(x)\\
    &\leq \int_M \sup\limits_{0< \tau < b-a}\int_a^{b-\tau}\tau^{-1}\int_u^{u+\tau}\lvert f(x)'\rvert_N^p(\theta)\dif\theta \dif u \dif\mu_M(x)\\
    &= \int_M \sup\limits_{0< \tau < b-a}\int_I\lvert f(x)'\rvert_N^p(\theta)\tau^{-1}\left(\int_a^{b-\tau}\mathds{1}_{\{u\leq \theta \leq u+\tau\}}(\theta) \dif u\right) \dif\theta \dif\mu_M(x)\\
    &= \int_M \sup\limits_{0< \tau < b-a}\int_I\lvert f(x)'\rvert_N^p(\theta)\tau^{-1}\left(\int_a^{b-\tau}\mathds{1}_{\{\theta - \tau \leq u \leq \theta\}}(u) \dif u\right) \dif\theta \dif\mu_M(x),
    \end{align*}
    where the Fubini--Tonelli theorem \cite[Proposition 5.2.1]{cohn2013measure} was used to interchange the order of integration over $u$ and $\theta$. Now, we have the following 
    \begin{align*}
    \int_a^{b-\tau}\mathds{1}_{\{\theta - \tau \leq u \leq \theta\}}(u) \dif u = \begin{cases}
        \theta - a,& \text{if $a \leq \theta < a +\tau $},\\
        \tau,& \text{if $a+\tau \leq \theta \leq b -\tau $},\\
        b -\theta,& \text{if $b-\tau < \theta \leq b$}.
    \end{cases}
    \end{align*}
    In any case, we get that $\tau^{-1}\int_a^{b-\tau}\mathds{1}_{\{\theta - \tau \leq u \leq \theta\}}(u) \dif u \leq 1$ for all $\tau\in (0,b-a)$ and $\theta \in I$, so that we eventually get the following inequality
    \begin{align*}
        \sup_{0 < \tau < b-a} \int_a^{b-\tau} (\Delta_{p,\tau} (\bar{i}_p^{-1}\circ (q_p)_*)(f)(u))^p\dif u &\leq\int_M \int_I\lvert f(x)'\rvert_N^p(\theta)\dif \theta \dif\mu_M(x)< \infty.
    \end{align*}
    Thus, $(\bar{i}_p^{-1}\circ (q_p)_*)(f)$ belongs to $ W^{1,p}(I,\Lph(M,N))$ (\cref{def:sobolev_curves}) and we can define $c \coloneqq (Q_p^{-1}\circ \bar{i}_p^{-1}\circ (q_p)_*)(f)$, which belongs to $\mathcal{AC}^p(I,\Lph(M,N))$ and satisfies $f=\tilde{i}_p(c)$. Hence, $f$ belongs to the range of $\tilde{i}_p$. We thus proved that the range of $\tilde{i}_p$ is the set of all $f\in L^p_{\underline{h}}(M,\mathcal{AC}^p(I,N))$ such that
    \begin{equation*}
    \int_M \int_I \lvert f(x)'\rvert_N^p(t)\dif t \dif\mu_M(x)< \infty.\qedhere
    \end{equation*}
\end{proof}

The mapping $\tilde{i}_p$, $p\in (1,\infty)$, defined in \cref{th:charact_ac_curves} need not be surjective, as illustrated by the following counterexample.

\begin{cexample}[$\tilde{i}_p$ need not be surjective]
\label{ex:ip_not_surjective}
    Let $p\in (1,\infty)$, $I=[0,1]$, $M=(0,1)$ with $\Sigma_M=\mathcal{B}((0,1))$ and $\mu_M = \lebesgue|_{(0,1)}$, $N=\mathbb{R}$ with its standard metric, and let $h\equiv 0$. Define $A_n \coloneqq [2^{-n},2^{-n+1})$ for every $n\geq 1$, so that $(0,1)=\cup_{n\geq 1} A_n$ and $\lebesgue(A_n) = 2^{-n}$. Then, for every $n\geq 1$ and $x\in A_n$, set $f(x)(t)\coloneqq \sin(2\pi 2^n t)$ for all $t\in I$, so that $f(x)$ is smooth on the compact $[0,1]$, hence belongs to $\mathcal{AC}^p([0,1],N)$. Since $f$ is countably valued and each $A_n$ is measurable, $f: (0,1)\to \mathcal{AC}^p([0,1],\mathbb{R})$ is separably valued and measurable. In addition, defining $C_p\coloneqq \int_0^1 \lvert \sin(2\pi t)\rvert^p \dif t\in (0,\infty)$, we have 
    $$\int_0^1\int_0^1 \lvert f(x)(t)\rvert^p\dif t \dif x = C_p\sum_{n\geq 1}2^{-n} = C_p < \infty,$$
    so that $f\in L^p((0,1),\mathcal{AC}^p([0,1],\mathbb{R}))$. However, for every $x\in A_n$, we have
    $$(\partial_t f(x))(t) =2\pi 2^n \cos(2\pi 2^n t)$$
    and $\lvert (\partial_t f(x))(t)\rvert = \lvert f(x)'\rvert_{\mathbb{R}}(t)$ for $\lebesgue$-a.e.~$t\in I$ \cite[Remark 1.1.3]{ambrosio2005gradient}.
    Hence, defining $C_p'\coloneqq (2\pi)^p\int_0^1\lvert \cos(2\pi t)\rvert^p\dif t\in (0,\infty)$, we get 
    $$\int_0^1\int_0^1\lvert f(x)'\rvert_{\mathbb{R}}^p(t)\dif t\dif x = C_p' \sum_{n\geq 1}2^{n(p-1)} = \infty,$$
    since $p -1 > 0$. Therefore, $f$ does not belong to the range of $\tilde{i}_p$, that is, there is no curve $c\in \mathcal{AC}^p([0,1],L^p((0,1),\mathbb{R}))$ such that $\tilde{i}_p(c)= f$.
\end{cexample}

In the next section, we show that, when $p=1$, the appropriate class of curves is the set of càdlàg curves of bounded variation and we will see why \cref{th:charact_ac_curves} can fail when $p=1$ (\cref{ex:counterex_charct_ac_curve}).

\subsection{\texorpdfstring{Characterization of càdlàg curves of bounded variation for $p=1$}{Characterization of càdlàg curves of bounded variation for p=1}}
\label{sec:charact_bv}

We first clarify some recurrent notation used in this section and thereafter.
\begin{notation}
\label{not:bv}
    Let $h\in L^0(M,N)$. From thereon, it is assumed that:
    \begin{enumerate}[label=(\roman*)]
        \item $L^1_{\underline{h}}(M,\mathcal{BV}(I,N))$ denotes the set of equivalence classes of separably valued measurable mappings from $M$ to $(\mathcal{BV}(I,N),d_1)$ at a finite $D_{1,1}$-distance from $\underline{h}$ (see \cref{not:lp} for the definition of $\underline{h}$).
        \item $\operatorname{Var}_1\coloneqq \operatorname{Var}_{L^1_h(M,N)}$ denotes the variation of curves in $L^1_h(M,N)$ (\cref{def:variation}).
        \item $\lvert Dc\rvert_1\coloneqq \lvert Dc\rvert_{L^1_h(M,N)}$ denotes the variation measure of the curve $c\in \mathcal{BV}(I, L^1_h(M,N))$ (\cref{def:variation_measure}).
\item $Q_1$ denotes the restriction to $\mathcal{BV}(I,L^1_h(M,N))$ of the quotient mapping $Q_I$ (see \cref{not:ac} for the definition of $Q_I$).
\end{enumerate}
\end{notation}

Then, the isometry given by \cref{th:isom_lebesgue} yields a characterization of càdlàg curves of bounded variation in nonlinear Lebesgue spaces for $p=1$.

\begin{theorem}[Characterization of càdlàg curves of bounded variation for $p=1$]
\label{th:charact_bv_curves}
    Let $h\in L^0(M,N)$. Suppose that $N$ is complete. Then, the isometry $\bar{i}_1$ defined in \cref{th:isom_lebesgue} induces the isometric mapping \begin{align*}
        \tilde{i}_1\coloneqq (q_1^{-1})_* \circ \bar{i}_1\circ Q_1: (\mathcal{BV}(I,L^1_h(M,N)),d_{1,1})\to (L^1_{\underline{h}}(M,\mathcal{BV}(I,N)),D_{1,1})
    \end{align*}
     whose inverse mapping on its range is $\tilde{i}_1^{-1}\coloneqq Q_1^{-1}\circ \bar{i}_1^{-1}\circ (q_1)_*$ and satisfies 
     \begin{align}
     \label{eq:inverse_map_eval_map_bv}
     \tilde{i}_1^{-1}(f)(t) = (e_{t,1})_*(f)\quad\text{for all $t\in I$ and $f\in \tilde{i}_1( \mathcal{BV}(I,L^1_h(M,N)))$.}
     \end{align}
     In addition, $\tilde{i}_1$ satisfies for all $c\in \mathcal{BV}(I,L^1_h(M,N))$ that 
    \begin{align}
    \label{eq:meausure_equals_pointwise_measure}
        \lvert Dc\rvert_1 = \int_M \lvert D\tilde{i}_1(c)(x)\rvert_N\dif \mu_M(x) 
    \end{align}
    and its range is the set of all $f\in L^1_{\underline{h}}(M,\mathcal{BV}(I,N))$ such that 
    $$\int_M \lvert Df(x)\rvert_N(I)\dif \mu_M(x) < \infty.$$
\end{theorem}
\begin{remark}[On the definition of pushforwards for equivalence classes]
    The pushforward of an equivalence class by a mapping is meant as the equivalence class of the composition of the given mapping with a compatible representative of the input equivalence class. 
\end{remark}

\begin{remark}[On relaxing the finiteness assumption on $d_N$]
Using the same arguments as \cref{rem:relax_metric_charact_ac}, the conclusion of \cref{th:charact_bv_curves} holds when $d_N$ is not assumed finite (see \cref{assum:minimal} for the assumption on $d_N$), but requires $h$ to be constant in that case.
\end{remark}

\begin{proof}[Proof of \cref{th:charact_bv_curves}]
    The proof follows the same approach used in \cite[Theorems 3.1 and 3.3]{abedi2024absolutely} for Wasserstein spaces, which itself follows the approach of \cite[Theorems 4 and 5]{lisini2007characterization}. It consists in using piecewise constant approximations of càdlàg curves of bounded variation to prove inequalities and propagate the results to càdlàg curves of bounded variation  by passing to the limit.

    \noindent Let $h\in\mathcal{L}^0(M,N)$ and let $c \in \mathcal{BV}(I,L^1_h(M,N))$. 

     \par\medskip \noindent \emph{Step 1: Reduction to the case where $N$ is separable.} Let us first reduce the proof to the case where $N$ is separable. Indeed, when $c\in \mathcal{BV}(I,L^1_h(M,N))$, the set $N_0\coloneqq \overline{h(M)\cup(\cup_{t\in \mathbb{Q}\cap I}c(t)(M))}$ is separable as the closure of the countable union of separable sets. Then, since $c$ is right-continuous on $[a,b)$ and left-continuous at $t=b$, we have for all $t\in I$ that $c(t)= \lim_{s\to t^+,s\in \mathbb{Q}\cap I} c(s)$ and $c(b)= \lim_{s\to b^-,s\in \mathbb{Q}\cap I} c(s)$ in $L^1_h(M,N)$, hence, up to the extraction of a subsequence $c(t)(x) = \lim_{s\to t^+, s\in \mathbb{Q}\cap I} c(s)(x)$ and $c(b)(x) = \lim_{s\to b^-, s\in \mathbb{Q}\cap I} c(s)(x)$ in $N$ for $\mu_M$-a.e.~$x\in M$, that is, $c(t)(x)\in N_0$ and $c(b)(x)\in N_0$ for $\mu_M$-a.e.~$x\in M$. Thus, we can replace $N$ with $N_0$ in the rest of the proof, so that $\cL(M,N) = \cLs(M,N)$.

    \par\medskip \noindent \emph{Step 2: Definition of $\tilde{i}_1$.} Let $c\in \mathcal{BV}(I,L^1_h(M,N))$. Then, $Q_1(c)$ belongs to $BV(I,L^1_h(M,N))$, by \cref{prop:cadlag_representation_bv}, hence $(\bar{i}_1\circ Q_1)(c)$ belongs to $L^1_{\underline{h}}(M,L^1(I,N))$, by \cref{th:isom_lebesgue}. Define for all $n\in \mathbb{N}^*$ the set $I_n\coloneqq \{0,1,\ldots, 2^n\}$ and the sequence $(t_{n,i})_{i\in I_n}$ such that $t_{n,i}\coloneqq a+\lebesgue(I)i2^{-n}$. Also, define the simple mapping $c_n: I \to L^1_h(M,N)$ as $c_n|_{[t_{n,i},t_{n,i+1})} \equiv c(t_{n,i})$ for all $i\in I_{n-1}$ and $c_n(b)\coloneqq c(b)$. Since the set of dyadic rationals (scaled by $\lebesgue(I)$ and translated by $a$) is dense in $I$ and $c$ is right-continuous on $[a,b)$ and $\lebesgue(\{b\})=0$, we have for $\lebesgue$-a.e.~$t\in I$ that $c_n(t)\to c(t)$ in $L^1_h(M,N)$ as $n\to \infty$. In addition, if $t\in [t_{n,i},t_{n,i+1})$,
 \begin{align*}
     D_1(c_n(t),c(t)) = D_1(c(t_{n,i}),c(t)) \leq \operatorname{Var}_1(c;I). 
 \end{align*}
 Hence, we get that 
 $D_1(c_n(t),c(t))\leq \operatorname{Var}_1(c;I)$ for all $(t,n)\in I\times \mathbb{N}^*$. Since $I$ has finite $\lebesgue$-measure, the mapping constant equal to $\operatorname{Var}_1(c;I)<\infty$ on $I$ is integrable, so that, by Lebesgue's dominated convergence theorem \cite[Theorem~2.4.5]{cohn2013measure}, we have $d_{1,1}(c_n,c) \to 0$ as $n\to \infty$. In addition, for all $n\in\mathbb{N}^*$ and all $i\in I_n$ there is $c_{n,i}\in \mathcal{L}^1_h(M,N)$ such that $D_1(c_{n,i},c(t_{n,i}))=0$, so we can define $c_n': I\to \mathcal{L}^1_h(M,N)$ as $c_n'|_{[t_{n,i},t_{n,i+1})} \equiv c_{n,i}$ for all $i\in I_{n-1}$ and $c_n'(b)\coloneqq c_{n,2^n}$, which satisfies $D_1(c_n(t),c_n'(t))=0$ for all $t\in I$. Thus, we can define $\hat{c}_n':I\times M\to N$ as $\hat{c}_n'(t,x)\coloneqq c_n'(t)(x)$, which belongs to $\mathcal{L}^1_{\hat{h}}(I\times M,N)$ (simple on $I$ and by definition of the $c_{n,i}$'s) and satisfies $D_1(\hat{c}_n'(t,\cdot),c_n(t))=0$ for all $t\in I$. Hence, $\overline{\operatorname{sec}}_{I,1}$ being an isometry (\cref{th:extension_embed_iterated_simple}), we get that $\hat{D}_1(\hat{c}_n',\overline{\operatorname{sec}}_{I,1}^{-1}(c_n))=0$. Now, defining the mapping $f_n': M\to \mathcal{L}^1(I,N)$ as $f_n'(x)\coloneqq \hat{c}_n'(\cdot,x)$ and recalling that $\overline{\operatorname{sec}}_{M,1}$ is an isometry (\cref{th:extension_embed_iterated_simple}), we get that $D_{1,1}(f_n',\bar{i}_1(c_n))=0$, hence that $d_1(f_n'(x),\bar{i}_1(c_n)(x))=0$ for $\mu_M$-a.e.~$x\in M$. Let $(s,t)\in I^2$ be such that $s < t$ and define the mapping
    $$\psi_{s,t}: \gamma\mapsto \sup_{0<\tau<t-s} \int_s^{t-\tau}\Delta_{N,\tau}\gamma(u)\dif u.$$
 Then, letting $\varepsilon \in (0, t - s)$ and taking $n$ large enough such that $\lebesgue(I)2^{-n} < \varepsilon$, we have 
    \begin{align*}
    \int_M \psi_{s+\varepsilon,t}(\bar{i}_1(c_n)(x))\dif \mu_M(x) &= \int_M \psi_{s+\varepsilon,t}(f_n'(x))\dif \mu_M(x)\\
    &= \int_M \sup\limits_{0<\tau<\lebesgue(I)2^{-n}}\int_{s+\varepsilon}^{t-\tau}\Delta_{N,\tau} f_n'(x)(u)\dif u\dif \mu_M(x) \\
    &= \int_M \sum_{\substack{1\leq i\leq 2^n\\
s+\varepsilon<t_{n,i}<t}} d_N(f_n'(x)(t_{n,i-1}), f_n'(x)(t_{n,i})) \dif \mu_M(x)\\
    &=  \sum_{\substack{1\leq i\leq 2^n\\
s+\varepsilon<t_{n,i}<t}} \int_M d_N(c_n'(t_{n,i-1})(x), c_n'(t_{n,i})(x)) \dif \mu_M(x)\\
    &=  \sum_{\substack{1\leq i\leq 2^n\\
s+\varepsilon<t_{n,i}<t}} D_1(c_n(t_{n,i-1}), c_n(t_{n,i}))\\
    &=  \sum_{\substack{1\leq i\leq 2^n\\
s+\varepsilon<t_{n,i}<t}} D_1(c(t_{n,i-1}), c(t_{n,i}))\\
    &\leq \operatorname{Var}_1(c;(s,t))= \lvert Dc\rvert_1((s,t)) <\infty,
    \end{align*}
    where the second equality follows from \cite[Lemma 2.18]{abedi2024absolutely} and the last equality follows from assertion (i) of \cref{prop:cadlag_representation_bv}. The third equality follows from the fact, if $s+\varepsilon < t_{n,i} < t$, then $t_{n,i-1} = t_{n,i} - \lebesgue(I)2^{-n} > s$.
    The mapping $\bar{i}_1$ being isometric, we also have $D_{1,1}(\bar{i}_1(c_n),(\bar{i}_1\circ Q_1)(c))\to 0$ as $n\to \infty$. Hence, up to the extraction of a subsequence, we can assume for $\mu_M$-a.e.~$x\in M$ that $\bar{i}_1(c_n)(x)\to (\bar{i}_1\circ Q_1)(c)(x)$ in $L^1(I,N)$ as $n\to \infty$. Therefore, recalling that $\psi_{s+\varepsilon,t}$ is lower semicontinuous on $L^1(I,N)$ and using Fatou's lemma \cite[Theorem 2.4.4]{cohn2013measure}, the latter inequality yields that 
    \begin{align}
    \int_M \psi_{s+\varepsilon,t}((\bar{i}_1\circ Q_1)(c)(x))\dif \mu_M(x) &\leq \int_M \underset{n\to \infty}{\lim\inf} \psi_{s+\varepsilon,t}(\bar{i}_1(c_n)(x))\dif \mu_M(x) \notag\\
    &\leq \underset{n\to \infty}{\lim\inf} \int_M \psi_{s+\varepsilon,t}(\bar{i}_1(c_n)(x))\dif \mu_M(x) \notag\\
    &\leq \lvert D c\rvert_1((s,t)).\label{eq:bound_bv}
    \end{align}
    Taking $s=a$ and $t=b$ in \cref{eq:bound_bv}, and letting $\varepsilon\to 0$, we use the fact that $\psi_{a+\varepsilon,b}(\gamma)$ increases to $\psi_{a,b}(\gamma)$ as $\varepsilon \to 0$ for every $\gamma\in L^1(I,N)$. Hence, using the monotone convergence theorem \cite[Theorem 2.4.1]{cohn2013measure}, we get 
    $$\int_M \psi_{a,b}((\bar{i}_1\circ Q_1)(c)(x))\dif\mu_M(x) \leq \lvert Dc\rvert_1(I) <\infty.$$
    This yields that $\psi_{a,b}((\bar{i}_1\circ Q_1)(c)(x)) < \infty$ for $\mu_M$-a.e.~$x\in M$, which implies that $(\bar{i}_1\circ Q_1)(c)(x)$ belongs to $BV(I,N)$ (\cref{prop:equiv_def_bv}). Hence, $BV(I,X)$ being a Borel subset of $L^1(I,X)$ (\cref{rem:bv_borel}), the mapping $(\bar{i}_1\circ Q_1)(c)$ belongs to $L^1_{\underline{h}}(M,BV(I,N))$. Therefore, the isometric mapping
    $$\tilde{i}_1\coloneqq (q_1^{-1})_*\circ \bar{i}_1\circ Q_1: \mathcal{BV}(I,L^1_h(M,N))\to L^1_{\underline{h}}(M,\mathcal{BV}(I,N))$$
    is well-defined and its inverse mapping on its range is $\tilde{i}_1^{-1}\coloneqq Q_1^{-1}\circ \bar{i}_1^{-1}\circ (q_1)_*$. In addition, since $\psi_{s+\varepsilon,t} ((\bar i_1\circ Q_1)(c)(x)) = \lvert D \tilde{i}_1(c)(x)\rvert_N((s+\varepsilon,t))$ (\cref{prop:equiv_def_bv}) for $\mu_M$-a.e.~$x\in M$, \cref{eq:bound_bv} yields that 
    $$\int_M\lvert D\tilde{i}_1(c)(x)\rvert_N((s+\varepsilon,t))\dif\mu_M(x) \leq \lvert Dc\rvert_1((s,t)).$$
    Letting $\varepsilon\to 0$, the intervals $(s+\varepsilon,t)$ increase to $(s,t)$. Thus, by continuity from below of the variation measures \cite[Proposition 1.2.5 (a)]{cohn2013measure} and the monotone convergence theorem \cite[Theorem 2.4.1]{cohn2013measure}, we get
    $$\int_M \lvert D\tilde{i}_1(c)(x)\rvert_N((s,t)) \dif\mu_M(x) \leq \lvert Dc\rvert_1((s,t)).$$
    Since the latter inequality holds for every open interval $(s,t)\subset I$, we finally obtain 
    \begin{equation}
    \label{eq:measure_greater_thant_pw_measure}
    \int_M \lvert D\tilde{i}_1(c)(x)\rvert_N\dif\mu_M(x)\leq \lvert Dc\vert_1.
    \end{equation}

    \par\medskip \noindent \emph{Step 3: Characterization of the range of $\tilde{i}_1$.} By \cref{eq:measure_greater_thant_pw_measure}, we know that the range of $\tilde{i}_1$ is contained in the set of all $f\in L^1_{\underline{h}}(M,\mathcal{BV}(I,N))$ such that
    \begin{align*}
    \int_M \lvert Df(x)\rvert_N(I)\dif\mu_M(x)< \infty.
    \end{align*}
 Let us now prove the reverse inclusion. Thus, let $f\in L^1_{\underline{h}}(M,\mathcal{BV}(I,N))$ such that
    \begin{align*}
    \int_M \lvert Df(x)\rvert_N(I)\dif\mu_M(x) < \infty.
    \end{align*}
    We already know that $\bar{i}_1^{-1}\circ (q_1)_*(f) \in L^1(I,L^1_h(M,N))$ and, by \cref{th:extension_embed_iterated_simple}, we know that
    $$\hat{f}\coloneqq (\overline{\operatorname{sec}}_{I,1}^{-1}\circ \bar{i}_1^{-1}\circ (q_1)_*)(f) = (\overline{\operatorname{sec}}_{M,1}^{-1}\circ (q_1)_*)(f)$$
    belongs to $L^1_{\hat{h}}(I\times M, N)$. Also, by the definitions of $\overline{\operatorname{sec}}_{M,1}$, $\overline{\operatorname{sec}}_{I,1}$ and $q_1$ (\cref{th:extension_embed_iterated_simple,prop:cadlag_representation_bv}), it satisfies $\hat{f}(\cdot,x)=f(x)$ in $L^1(I,N)$ for $\mu_M$-a.e.~$x\in M$ and $\hat{f}(t,\cdot)=(\bar{i}_1^{-1}\circ (q_1)_*)(f)(t)$ in $L^1_h(M,N)$ for $\lebesgue$-a.e.~$t\in I$. However, $f$ differs from $\underline{h}$ on a measurable set $M_f$ of $\sigma$-finite $\mu_M$-measure, so that $\hat{f}$ differs from $\hat{h}$ on a measurable set contained in the measurable rectangle $I\times M_f$. Hence, we get, using the Fubini--Tonelli theorem \cite[Proposition 5.2.1]{cohn2013measure}, that for all $(s,t)\in I^2$ such that $s< t$
    \begin{align*}\sup\limits_{0< \tau < t-s}\int_s^{t-\tau}\Delta_{1,\tau} (\bar{i}_1^{-1}\circ (q_1)_*)(f)(u)\dif u &= \sup\limits_{0< \tau < t-s}\int_s^{t-\tau} \int_{M_f} \Delta_{N,\tau} \hat{f}(u,x)\dif\mu_M(x)\dif u\\
    &= \sup\limits_{0< \tau < t-s} \int_{M_f} \int_s^{t-\tau} \Delta_{N,\tau} \hat{f}(u,x)\dif u\dif\mu_M(x)\\
    &\leq \int_M \sup\limits_{0< \tau < t-s}\int_s^{t-\tau}\Delta_{N,\tau} f(x)(u)\dif u \dif\mu_M(x)\\
    &= \int_M \lvert Df(x)\rvert_N((s,t))\dif\mu_M(x) < \infty.
    \end{align*}
    This yields that $(\bar{i}_1^{-1}\circ (q_1)_*)(f)\in BV(I,L^1_h(M,N))$ (\cref{prop:equiv_def_bv}) and we can define $c \coloneqq (Q_1^{-1}\circ \bar{i}_1^{-1}\circ (q_1)_*)(f)$, which belongs to $\mathcal{BV}(I,L^1_h(M,N))$ and satisfies $f= \tilde{i}_1(c)$. Hence, $f$ belongs to the range of $\tilde{i}_1$. We thus proved that the range of $\tilde{i}_1$ is the set of all $f\in L^1_{\underline{h}}(M,\mathcal{BV}(I,N))$ such that 
    $$\int_M \lvert Df(x)\rvert_N(I)\dif \mu_M(x) < \infty.$$
    In addition, we get, from the previous inequality, that for all $c\in \mathcal{BV}(I,L^1_h(M,N))$
    \begin{equation}
    \label{eq:pw_measure_greater_than_measure}
        \lvert Dc\rvert_1\leq\int_M \lvert D\tilde{i}_1(c)(x)\rvert_N\dif\mu_M(x).
    \end{equation}
    
    \par\medskip \noindent \emph{Step 4: Proof of \cref{eq:meausure_equals_pointwise_measure}.} As a consequence of \cref{eq:measure_greater_thant_pw_measure,eq:pw_measure_greater_than_measure}, we get that for all $c\in \mathcal{BV}(I,L^1_h(M,N))$
    \begin{equation*}\lvert Dc\rvert_1 = \int_M \lvert D\tilde{i}_1(c)(x)\rvert_N\dif\mu_M(x).\end{equation*}

    \par\medskip \noindent \emph{Step 5: Proof of \cref{eq:inverse_map_eval_map_bv}.} Let $f$ be in the range of $\tilde{i}_1$. Then, let $c\in \mathcal{BV}(I,L^1_h(M,N))$ be the unique curve such that $f= \tilde{i}_1(c)$ and define
    $$\hat{f}\coloneqq (\overline{\operatorname{sec}}_{I,1}^{-1}\circ \bar{i}_1^{-1}\circ (q_1)_*)(f) = (\overline{\operatorname{sec}}_{M,1}^{-1}\circ (q_1)_*)(f).$$
    Also, note that $\hat{f}= \overline{\operatorname{sec}}_{I,1}^{-1}\circ Q_1(c)$ since $\tilde{i}_1= (q_1^{-1})_*\circ \bar{i}_1\circ Q_1$. Now, there is $f'\in \mathcal{L}_{\underline{h}}^1(M,\mathcal{BV}(I,N))$ such that $D_{1,1}(f',f)= 0$, which implies that $f'(x)=f(x)$ in $\mathcal{D}(I,N)$ for $\mu_M$-a.e.~$x\in M$ (\cref{prop:cadlag_representation_bv}). Thus, we can define the mapping $\tilde{f}: I\times M\to N$ as $\tilde{f}(t,x)\coloneqq f'(x)(t)=(e_{t,1}\circ f')(x)$. To show that $\tilde{f}$ is measurable, consider for all $n\in\mathbb{N}^*$ the mapping $\tilde{f}_n: I\times M\to N$ defined as $\tilde{f}_n(t,x)\coloneqq (e_{t_{n,i},1}\circ f')(x)$ if $t\in [t_{n,i},t_{n,i+1})$ with $i\in I_{n-1}$ and $\tilde{f}_n(b,x)\coloneqq(e_{b,1}\circ f')(x)$, so that $\tilde{f}_n$ belongs to $\mathcal{L}^0(I\times M,N)$ since $e_{t,1}\circ f'\in \mathcal{L}^0(M,N)$ for all $t\in I$ (\cref{prop:measurability_evaluation_bv}). Then, the set of dyadic rationals (scaled by $\lebesgue(I)$ and translated by $a$) being dense in $I$ and $f'(x)$ being right-continuous on $[a,b)$ for all $x\in M$, we have $\tilde{f}(t,x) = \lim_{n\to \infty} \tilde{f}_n(t,x)$ for all $(t,x)\in I\times M$. Therefore, $\tilde{f}$ is measurable as the pointwise limit of measurable mappings (\cref{prop:measurability_pointwise_limit}). In addition, $D_1(\tilde{f}(t,\cdot),(e_{t,1})_*(f))=0$ for all $t\in I$ since $\tilde{f}(\cdot,x)= f(x)$ in $\mathcal{D}(I,N)$ for $\mu_M$-a.e.~$x\in M$. Then, using the definitions of $\hat{f}$, $\overline{\operatorname{sec}}_{M,1}$ (\cref{th:extension_embed_iterated_simple}) and $q_1$ (\cref{prop:cadlag_representation_bv}), we have $d_1(\tilde{f}(\cdot,x),\hat{f}(\cdot,x)) =0$ for $\mu_M$-a.e.~$x\in M$. Hence, $\overline{\operatorname{sec}}_{M,1}$ being an isometry, we get $\hat{D}_1(\tilde{f},\hat{f}) = 0$. Thus, using the definitions of $\hat{f}$, $\overline{\operatorname{sec}}_{I,1}$ (\cref{th:extension_embed_iterated_simple}) and $Q_1$ (\cref{not:bv}), we get $D_1(\tilde{f}(t,\cdot),c(t)) =0$ for $\lebesgue$-a.e.~$t\in I$, which yields $\tilde{i}_1^{-1}(f)(t) = (e_{t,1})_*(f)$ in $L^1_h(M,N)$ for $\lebesgue$-a.e.~$t\in I$. Together with the equality of \cref{eq:meausure_equals_pointwise_measure}, this implies that the curve $t\mapsto (e_{t,1})_*(f)$ belongs to $\mathcal{BV}(I,L^1_h(M,N))$ and since $d_{1,1}$ separates $\mathcal{BV}(I,L^1_h(M,N))$ (\cref{prop:cadlag_representation_bv}), $\tilde{i}_1^{-1}(f)(t) = (e_{t,1})_*(f)$ in $L^1_h(M,N)$ for all $t\in I$, which ends the proof of \cref{eq:inverse_map_eval_map_bv}.
\end{proof}

Similarly to what was observed for the mapping introduced in \cref{th:charact_ac_curves}, the mapping $\tilde{i}_1$ defined in \cref{th:charact_bv_curves} need not be surjective, as illustrated by the following counterexample, which uses essentially the same construction as \cref{ex:ip_not_surjective}.

\begin{cexample}[$\tilde{i}_1$ need not be surjective]
    Consider the same spaces, partition $(A_n)_{n\geq 1}$ of $(0,1)$ and mapping $f: (0,1)\to \mathcal{C}^\infty([0,1],\mathbb{R})$ as in \cref{ex:ip_not_surjective}. In particular, recall that $f$ is defined as $f(x)(t)\coloneqq\sin(2\pi 2^n t)$ for every $t\in [0,1]$, $n\geq 1$ and $x\in A_n$. Then, we have, defining $C_1\coloneqq \int_0^1 \lvert \sin(2\pi t)\rvert \dif t\in (0,\infty)$, that
    $$\int_0^1\int_0^1 \lvert f(x)(t)\rvert\dif t \dif x = C_1\sum_{n\geq 1}2^{-n} = C_1 < \infty,$$
    so that $f\in L^1((0,1),\mathcal{BV}([0,1],\mathbb{R}))$. However, for every $x\in A_n$, we have
    $$(\partial_t f(x))(t) =2\pi 2^n \cos(2\pi 2^n t)$$
    and $\lvert (\partial_t f(x))(t)\rvert = \lvert f(x)'\rvert_{\mathbb{R}}(t)$ for $\lebesgue$-a.e.~$t\in I$ \cite[Remark 1.1.3]{ambrosio2005gradient}. Since $f(x)\in \mathcal{AC}([0,1],\mathbb{R})$, we have that $\lvert Df(x)\rvert_{\mathbb{R}} = \lvert f(x)'\rvert_\mathbb{R}\lebesgue$ \cite[Section 2.2]{abedi2024absolutely}. Hence, defining $C_1'\coloneqq 2\pi\int_0^1\lvert \cos(2\pi t)\rvert\dif t\in (0,\infty)$, we have for every $x\in A_n$ that $\lvert Df(x)\rvert_{\mathbb{R}}([0,1]) = C_1'2^n$ and we get 
    $$\int_0^1\lvert Df(x)\rvert_{\mathbb{R}}([0,1])\dif x = \sum_{n\geq 1}\int_{A_n }C_1' 2^n \dif x = C_1'\sum_{n\geq 1} 1= \infty.$$
    Therefore, $f$ does not belong to the range of $\tilde{i}_1$, that is, there is no curve $c\in \mathcal{BV}([0,1],L^1((0,1),\mathbb{R}))$ such that $\tilde{i}_1(c)= f$.
\end{cexample}

Finally, to better understand why \cref{th:charact_ac_curves} does not extend to $p=1$, we provide a simple counterexample.

\begin{cexample}[Counterexample to \cref{th:charact_ac_curves} when $p=1$]
\label{ex:counterex_charct_ac_curve}
Let $I=M=[0,1]$ with $\Sigma_I=\Sigma_M= \mathcal{B}([0,1])$ and $\mu_I=\mu_M=\lebesgue|_{[0,1]}$, let $N=\mathbb{R}$ with its standard metric, and let $h\equiv 0$. Then, define $c:[0,1]\to L^1([0,1],\mathbb{R})$ as $c(t)\coloneqq \mathds{1}_{(0,t)}$. For every $(s,t)\in [0,1]^2$, the curve $c$ satisfies 
$D_1(c(s),c(t)) = \int_{[0,1]}\lvert \mathds{1}_{(0,s)}(x)- \mathds{1}_{(0,t)}(x)\rvert \dif x = \lvert s - t\rvert$. Hence, $c$ belongs to $ \mathcal{AC}^1([0,1],L^1([0,1],\mathbb{R}))$. However, it holds for $\lebesgue$-a.e.~$x\in [0,1]$ that $\tilde{i}_1(c)(x)=\mathds{1}_{[x,1]}$, which is clearly not a continuous curve. 
\end{cexample}

In the next section, we show that \cref{th:charact_ac_curves} allows for the pointwise characterization of the geometry of nonlinear Lebesgue spaces for $p> 1$ and that, in these spaces, a notion of velocity can be defined for absolutely continuous curves despite the lack of differential structure outside the case where the target space $N$ is linear.

\section{Applications}
\label{sec:applications}
\subsection{\texorpdfstring{Geometry of nonlinear Lebesgue spaces for $p>1$}{Geometry of nonlinear Lebesgue spaces for p>1}}
\label{sec:geometry}

\subsubsection{Characterization of geodesics}
\label{sec:charact_geodesics}

Let us first recall the definition of \emph{constant-speed geodesics}.

\begin{definition}[Geodesics]
\label{def:constant_geodesics}
    Let $X$ be a metric space with metric $d_X: X^2\to [0,\infty]$. A curve $c: I\to X$ is called a \emph{constant-speed geodesic} if and only if it satisfies $d_X(c(a),c(b))<\infty$ and for all $(s,t)\in I^2$  
    $$d_X(c(s),c(t)) = \left\lvert\frac{t - s }{ b - a }\right\rvert d_X(c(a),c(b)).$$
    The set of all such curves is denoted by $\mathcal{G}(I,X)$ and is a subset of $\mathcal{AC}^p(I,X)$ for all $p\in [1,\infty]$.
\end{definition}

Then, \cref{th:charact_ac_curves} yields a characterization of constant-speed geodesics in nonlinear Lebesgue spaces for $p> 1$.

\begin{theorem}[Characterization of geodesics for $p>1$]
\label{th:charact_geodesics}
    Let $h\in L^0(M,N)$ and $p\in (1,\infty)$. Suppose that $N$ is complete. Then, the isometric mapping $\tilde{i}_p$ defined in \cref{th:charact_ac_curves} satisfies 
    \begin{align*}
    \tilde{i}_p(\mathcal{G}(I,\Lph(M,N))) = L^p_{\underline{h}}(M,\mathcal{G}(I,N)).
    \end{align*}
\end{theorem}

\begin{remark}[On relaxing the finiteness assumption on $d_N$]
As highlighted in \cref{rem:relax_metric_charact_ac}, \cref{th:charact_ac_curves} holds when $d_N$ is not assumed finite, but $h$ is assumed constant. In that case, \cref{th:charact_geodesics} also holds.
\end{remark}

\begin{proof}[Proof of \cref{th:charact_geodesics}]
    Let $h\in \cLs(M,N)$ and $p\in (1,\infty)$. 
    
    \par\medskip\noindent \emph{Step 1: Direct inclusion.} Let $c\in \mathcal{G}(I,\Lph(M,N))$. By \cref{th:charact_ac_curves}, $\tilde{i}_p(c)$ belongs to $L^p_{\underline{h}}(M,\mathcal{AC}^p(I,N))$ since $\mathcal{G}(I,\Lph(M,N))$ is a subset of $ \mathcal{AC}^p(I,\Lph(M,N))$. Then, defining the measurable set  
    $$M_c\coloneqq \{x\in M: d_p(\tilde{i}_p(c)(x),\underline{h}(x)) >0\} = \{x\in M: d_\infty(\tilde{i}_p(c)(x),\underline{h}(x)) >0\}$$
    of $\sigma$-finite $\mu_M$-measure, we have that $\lvert \tilde{i}_p(c)(x)'\rvert_N(t) = 0$ for all $(x,t)\in (M\setminus M_c)\times I$ since $\underline{h}(x)$ is constant over $I$. Therefore, using the Fubini--Tonelli theorem \cite[Proposition 5.2.1]{cohn2013measure}, we get that for all $(s,t)\in I^2$ such that $s< t$
    \begin{align*}
        \int_M d_N(\tilde{i}_p(c)(x)(s),\tilde{i}_p(c)(x)(t))^p\dif\mu_M(x) &= D_p(c(s),c(t))^p\\
        &= \lvert t -s\rvert^{p-1}\int_s^t\lvert c'\rvert_p^p(u)\dif u\\
        &= \lvert t -s\rvert^{p-1}\int_s^t \int_M \lvert \tilde{i}_p(c)(x)'\rvert_N^p(u)\dif\mu_M(x)\dif u\\
        &= \lvert t -s\rvert^{p-1}\int_s^t \int_{M_c} \lvert \tilde{i}_p(c)(x)'\rvert_N^p(u)\dif\mu_M(x)\dif u\\
        &= \lvert t -s\rvert^{p-1}\int_{M_c} \int_s^t\lvert \tilde{i}_p(c)(x)'\rvert_N^p(u)\dif u\dif\mu_M(x)\\
        &= \lvert t -s\rvert^{p-1}\int_M \int_s^t\lvert \tilde{i}_p(c)(x)'\rvert_N^p(u)\dif u\dif\mu_M(x).
    \end{align*}
    Therefore, together with the triangle inequality and Hölder's inequality, this yields that for $\mu_M$-a.e.~$x\in M$
    $$d_N(\tilde{i}_p(c)(x)(a),\tilde{i}_p(c)(x)(b)) = \left(\int_a^b\lvert \tilde{i}_p(c)(x)'\rvert_N(t)\dif t\right)^p = \lvert b -a\rvert^{p-1}\int_a^b\lvert \tilde{i}_p(c)(x)'\rvert_N^p(t)\dif t < \infty.$$
    Thus, for $\lebesgue$-a.e.~$t\in I$
    $$\lvert \tilde{i}_p(c)(x)'\rvert(t) = \frac{d_N(\tilde{i}_p(c)(x)(a),\tilde{i}_p(c)(x)(b))}{\lvert b - a\rvert}.$$
    Hence, for $\mu_M$-a.e.~$x\in M$ we have that for all $(s,t)\in (\mathbb{Q}\cap I)^2$ such that $s< t$
    $$d_N(\tilde{i}_p(c)(x)(s),\tilde{i}_p(c)(x)(t)) =\left\lvert\frac{ t -s}{ b - a}\right\rvert d_N(\tilde{i}_p(c)(x)(a),\tilde{i}_p(c)(x)(b)).$$
    By uniform continuity of $\tilde{i}_p(c)(x)$ and density of $\mathbb{Q}\cap I$ in $I$, this holds for all $(s,t)\in I^2$ such that $s< t$, hence $\tilde{i}_p(c)(x)\in \mathcal{G}(I,N)$ for $\mu_M$-a.e.~$x\in M$ and $\tilde{i}_p(c)\in L^p_{\underline{h}}(M,\mathcal{G}(I,N))$.

    \par\medskip\noindent \emph{Step 2: Reverse inclusion.} Let $f\in L^p_{\underline{h}}(M,\mathcal{G}(I,N))$. Then, \cref{th:charact_ac_curves} gives us a condition for $f$ to belong to the range of $\tilde{i}_p$, namely it must satisfy 
    $$\int_M\int_I \lvert f(x)'\rvert_N^p(t)\dif t\dif\mu_M(x) < \infty.$$
    Choose a representative of $f$, still denoted by $f$, taking values in $\mathcal{G}(I,N)$. Thus, $f(x)\in \mathcal{G}(I,N)$ for every $x\in M$ and we have, in particular, that 
    $$\lvert f(x)'\rvert_N(t)= \frac{d_N(f(x)(a),f(x)(b))}{\lvert b - a\rvert}\quad\text{for $\lebesgue$-a.e.~$t\in I$},$$
    so that we have
    $$\int_I \lvert f(x)'\rvert_N(t)^p\dif t = \frac{d_N(f(x)(a),f(x)(b))^p}{\lvert b - a\rvert^{p-1}}.$$
    Hence, showing that 
    $$\int_M\int_I \lvert f(x)'\rvert_N^p(t)\dif t\dif\mu_M(x) < \infty$$
    amounts to showing that 
    $$\int_M\frac{d_N(f(x)(a),f(x)(b))^p}{\lvert b - a\rvert^{p-1}}\dif\mu_M(x) < \infty.$$
    By the triangle inequality, we have that 
    $$d_N(f(x)(a),f(x)(b)) \leq d_N(f(x)(a),h(x)) + d_N(h(x),f(x)(b)),$$
    so that one endpoint is at least at a distance $d_N(f(x)(a),f(x)(b))/2$ from $h(x)$. Suppose, for instance, that $d_N(f(x)(a),h(x))\geq d_N(f(x)(a),f(x)(b))/2$. Then, since $f(x)$ has constant speed, 
    $$d_N(f(x)(a),f(x)(t)) = \frac{t - a}{b- a}d_N(f(x)(a),f(x)(b)).$$
    Therefore, for $t\in [a, a + (b-a)/4]$, the reverse triangle inequality gives that 
    \begin{align*}
    d_N(f(x)(t),h(x))&\geq d_N(f(x)(a),h(x)) - d_N(f(x)(a),f(x)(t))\\
    &\geq\frac{d_N(f(x)(a),f(x)(b))}{2} - \frac{d_N(f(x)(a),f(x)(b))}{4}\\
    &\geq \frac{d_N(f(x)(a),f(x)(b))}{4}.
    \end{align*}
    The case $d_N(h(x),f(x)(b))\geq d_N(f(x)(a),f(x)(b))/2$ is treated identically on $[b - (b-a)/4,b]$. Thus, 
    $$\int_Id_N(f(x)(t),h(x))^p\dif t \geq \frac{\lvert b - a\rvert}{4}\left(\frac{d_N(f(x)(a),f(x)(b))}{4}\right)^p,$$
    so that 
    $$\int_M\int_I\lvert f(x)'\rvert_N^p\dif t \dif\mu_M(x) \leq \frac{4^{p+1}}{\lvert b - a\rvert^p}\int_M\int_I d_N(f(x)(t),h(x))^p\dif t\dif\mu_M(x) <\infty,$$
    since $f\in L^p_{\underline{h}}(M,\mathcal{G}(I,N))$. Hence, $f$ belongs to the range of $\tilde{i}_p$.
    Now, let $c\coloneqq \tilde{i}_p^{-1}(f)\in \mathcal{AC}^p(I,\Lph(M,N))$. Recalling that $c(t)= (e_{t,p})_*f$ for all $t\in I$, we get that for all $(s,t)\in I^2$
    \begin{align*}
        D_p(c(s),c(t))^p &= \int_M d_N(f(x)(s),f(x)(t))^p\dif\mu_M(x)\\
        &= \left\lvert\frac{s - t}{b- a}\right\rvert^p \int_M d_N(f(x)(a),f(x)(b))^p\dif\mu_M(x)\\
        &= \left\lvert\frac{s - t}{b- a}\right\rvert^p D_p(c(a),c(b))^p.
    \end{align*}
    Hence, $c\in \mathcal{G}(I,\Lph(M,N))$.
\end{proof}

\begin{remark}[Related results in the literature]
    A related pointwise characterization of geodesics was established by N.~Monod in \cite[Proposition 44]{monod2006superrigidity} in the case $p=2$, $h$ constant and $\mu_M$ a probability measure. In N.~Monod's formulation, geodesics in $L^2_h(M,N)$ are represented by measurable families of geodesics in $N$ with measurable pointwise speeds. \cref{th:charact_geodesics} extends this characterization to $p\in (1,\infty)$ in the present framework and incorporates it into the canonical isometric identification of \cref{th:charact_ac_curves}.
\end{remark}

In the next two sections, we show that \cref{th:charact_geodesics} allows for the pointwise characterization of the geometry of nonlinear Lebesgue spaces for $p>1$.

\subsubsection{Characterization of the length structure}
\label{sec:charact_length}

Before delving into the length structure of nonlinear Lebesgue spaces, let us point out that a reminder on length spaces is provided in \ref{sec:metric_geometry}. Now, for $p\in [1,\infty]$, denote by $L_p$ the length induced by the metric $D_p$ on $\Lph(M,N)$ (\cref{def:length_curve}). Thus, the intrinsic metric induced by $D_p$ (\cref{def:intrinsic_metric}) is defined, for all $(f,f')\in \Lph(M,N)^2$, as
$$\widehat{D}_p(f,f') \coloneqq \inf \left\{L_p(c): c\in \mathcal{AC}(I,\Lph(M,N)), c(a)=f, c(b)=f'\right\}.$$

A first notable result on the length structure (see \cref{def:length_spaces} for a definition of length spaces)  of nonlinear Lebesgue spaces is that it is induced by the length structure of the target space. 

\begin{proposition}[Length structure of nonlinear Lebesgue spaces]
\label{prop:length_lebesgue}
    Let $h\in L^0(M,N)$ and $p\in [1,\infty]$. Suppose that $\Lph(M,N)$ is trivial (\cref{prop:trivial_lebesgue}) or that $N$ is a complete and separable length (resp.~geodesic) space. Then, $L^p_h(M,N)$ is a length (resp.~geodesic) space.
\end{proposition}

\begin{proof}
    The proof follows the same approach used in \cite[Proposition~1]{lisini2007characterization} for Wasserstein spaces, which essentially consists in constructing a curve whose length is controlled by the distance between its endpoints from the measurable selection of pointwise curves satisfying such a constraint.

    \noindent Let $p\in [1,\infty]$, $h\in \cLs(M,N)$, $(f,f')\in \cLph(M,N)^2$ and $\kappa > 1$. In addition, let $h^2(x)\coloneqq(h(x),h(x))$ and define the mapping $\pi: M\to N^2$ as $\pi(x)\coloneqq(f(x),f'(x))$ which belongs to $\mathcal{L}^p_{h^2}(M,N^2)$ since $(f,f')\in \cLph(M,N)^2$. When $\Lph(M,N)$ is trivial, that is, $\Lph(M,N)=\left\{[h]_M\right\}$, it is also a geodesic space as it contains a single element. Therefore, assume that $\Lph(M,N)$ is nontrivial and that $N$ is a complete and separable length space.  

    \par\medskip \noindent \emph{Step 1: Reduction to the case where $\mu_M$ is $\sigma$-finite.} The mappings $f$ and $f'$ both differ from $h$ on a measurable set of $\sigma$-finite $\mu_M$-measure \cite[Proposition 3.4]{serieys2025nonlinear}. Hence, $f$ and $f'$ differ on a set of $\sigma$-finite $\mu_M$-measure, so that we can assume that $\mu_M$ is $\sigma$-finite in the rest of the proof. Now, since $\mu_M$ is assumed $\sigma$-finite, there exists a finite measure which has the same null sets, for instance, $\nu_M: A\mapsto  \sum_{n\in\mathbb{N}} 2^{-n}\frac{\mu_M(A\cap B_n)}{\mu_M(B_n)+1}$ where $(B_n)_{n\in\mathbb{N}}$ is a sequence of measurable sets of finite $\mu_M$-measure such that $M=\cup_{n\in\mathbb{N}} B_n$.
    
    \par\medskip \noindent \emph{Step 2: Construction of a Borel measurable selection of curves connecting pairs of points in $N$.} Then, define the multi-valued application $\Gamma_\kappa: N^2\to 2^{C(I,N)}$ as follows
    $$\Gamma_\kappa(y,y') \coloneqq \left\{\gamma \in \mathcal{C}(I,N):\gamma(a)=y,\gamma(b)=y',\mathrm{Lip}(\gamma)\lvert b - a \rvert \leq \kappa\,d_N(y,y')\right\}.$$
    Since $N$ is a length space (\cref{def:length_spaces}), $\Gamma_\kappa(y,y')$ is non-empty for each $(y,y')\in N^2$. Now, using the lower semicontinuity of the Lipschitz constant, we can show that the graph of $\Gamma_\kappa$ defined as 
    $$G(\Gamma_\kappa) \coloneqq \left\{(y,y',\gamma)\in N^2\times \mathcal{C}(I,N): \gamma\in \Gamma_\kappa(y,y')\right\}$$
    is closed. Indeed, taking a sequence $((y_n,y_n',\gamma_n))_{n\in\mathbb{N}}$ in $G(\Gamma_\kappa)$ that converges to $(y,y',\gamma)\in N^2\times \mathcal{C}(I,N)$, we have that $d_N(y_n,y)\to 0$, $d_N(y_n',y')\to 0$ and $d_\infty(\gamma_n,\gamma)\to 0$ as $n\to\infty$. Hence, by uniqueness of the limit, $\gamma(a) = y$ and $\gamma(b) = y'$ and, by lower semicontinuity of the Lipschitz constant \cite[Proposition~2.3.4~(iv)]{burago2022course}, we can pass to the limit in $\mathrm{Lip}(\gamma_n)\lvert b  -a\vert \leq \kappa\,d_N(y_n,y_n')$, so that $\mathrm{Lip}(\gamma)\lvert b - a\vert\leq \kappa\,d_N(y, y')$, which yields that $\gamma \in \Gamma_\kappa(y,y')$ and thus that $(y,y',\gamma)\in G(\Gamma_\kappa)$. Since $(\mathcal{C}(I,N),d_\infty)$ is complete and separable and $G(\Gamma_\kappa)$ is closed, we have, by Aumann's measurable selection theorem (\cref{th:aumman_selection_theorem}), that there exists a $\pi_*\nu_M$-measurable mapping (see \cref{rem:extension_mu_measurable_mappings} for a definition) $S_\kappa: N\times N \to \mathcal{C}(I,N)$ such that $S_\kappa(y,y')\in \Gamma_\kappa(y,y')$ for all $(y,y')\in N^2$. Since, $\mathcal{C}(I,N)$ is separable and $\pi_*\mu_M$ has the same null sets as $\pi_*\nu_M$, there exists a Borel measurable mapping $\tilde{S}_\kappa$ such that $\tilde{S}_\kappa(y,y') = S_\kappa(y,y')$ for $\pi_*\mu_M$-a.e. $(y,y')\in N^2$ \cite[Proposition~2.12]{serieys2025nonlinear}. 
    
    \par\medskip \noindent \emph{Step 3: Construction of a curve in $L^p_h(M,N)$ from the Borel measurable selection.} Define $c(t)\coloneqq e_t\circ \tilde{S}_\kappa\circ \pi$ for all $t\in I$. Since we have $(\pi, \tilde{S}_\kappa, e_t)\in \mathcal{L}^0(M,N^2)\times \mathcal{L}^0(N^2,\mathcal{C}(I,N))\times \mathcal{L}^0(\mathcal{C}(I,N),N)$ with the Borel $\sigma$-algebras on $N^2$ and $\mathcal{C}(I,N)$, we get, by composition, that $c(t)\in \mathcal{L}^0(M,N)$ for all $t\in I$. Now, recalling that $\tilde{S}_\kappa (f(x),f'(x))\in \Gamma_\kappa(f(x),f'(x))$ for $\mu_M$-a.e.~$x\in M$, we have, for all $(s,t)\in I^2$, that
    \begin{align*}
    D_p(c(s),c(t))^p &= \int_M d_N(\tilde{S}_\kappa(f(x),f'(x))(s), \tilde{S}_\kappa(f(x),f'(x))(t))^p\dif\mu_M(x)\\
    &\leq \int_M \mathrm{Lip}(\tilde{S}_\kappa(f(x),f'(x)))^p\lvert s - t\rvert^p\dif\mu_M(x)\\
    &\leq \int_M \kappa^p\, d_N(f(x),f'(x))^p\left\lvert\frac{s - t}{b - a}\right\rvert^p\dif\mu_M(x)\\
    &= \kappa^p\, D_p(f,f')^p\left\lvert\frac{s - t}{b - a}\right\rvert^p.
    \end{align*}
    Since $c(a) = f\in \cLph(M,N)$, this yields that $c(t)\in \cLph(M,N)$ for all $t\in I$ and that, up to the identification with an equivalence class for each $t\in I$, $c$ determines a curve in $\mathcal{AC}(I,\Lph(M,N))$ such that $L_p(c) \leq \kappa\,D_p(f,f')$. Since we can construct such a curve for an arbitrary $\kappa > 1$, this yields that $\widehat{D}_p(f,f') = D_p(f,f')$, that is, $\Lph(M,N)$ is a length space. When $N$ is a geodesic space, the proof works with $\kappa = 1$. 
\end{proof}

\begin{remark}[On the separability assumption on $N$]
\label{rem:separability_length}
    Note that the proof of \cref{prop:length_lebesgue} only requires the completeness and separability assumptions on $N$ to ensure that the pointwise selection of curves is measurable. When $N$ is uniquely geodesic and $d_N$ is jointly convex, for instance, when $N$ is a global NPC space (\cref{def:alexandrov_curv}), $N$ need not be separable for the conclusion of \cref{prop:length_lebesgue} to hold since there is no measurable selection of geodesics to construct as any pair of points in $N$ is joined by a unique geodesic and geodesics in $N$ are continuous with respect to their endpoints thanks to the joint convexity of $d_N$. The question of whether the assumptions of completeness and separability on $N$ can be relaxed remains open.
\end{remark}

\begin{remark}[Related results in the literature]
    We emphasize that while the conclusion of \cref{prop:length_lebesgue} in the geodesic case already appeared in a work (under the assumption that $h$ is constant) by J.~Jost \cite[Lemma 4.1.2]{jost1997nonpositive}, the proof seemed incomplete as it selected pointwise geodesics without checking the measurability with respect to endpoints,  which is precisely what is done in the proof of \cref{prop:length_lebesgue}, hence missing the key assumptions of completeness and separability on $N$. As highlighted in \cref{rem:separability_length}, the question of whether the completeness and separability assumptions on $N$ can be relaxed remains open.
\end{remark}

The reverse implication holds when restricting to $p\in (1,\infty)$ and a measure $\mu_M$ that is not purely infinite, but does not require $N$ to be separable.

\begin{proposition}[Length structure of the target space]
\label{prop:length_target}
    Let $h\in L^0(M,N)$ and $p\in (1,\infty)$. Suppose that $\mu_M$ is not purely infinite and that $N$ is complete. Then, if $\Lph(M,N)$ is a length (resp.~geodesic) space, so is $N$.
\end{proposition}

\begin{remark}[Why it fails when $\mu_M$ is purely infinite]
    When $\lvert N\rvert = 1$, $\Lph(M,N)$ is trivial (\cref{prop:trivial_lebesgue}), hence a geodesic space, but so is $N$ (since it contains a single element). Thus, the conclusion of \cref{prop:length_target} holds in that case. However, when $\mu_M$ is purely infinite and $\lvert N\rvert > 1$, we know that $\Lph(M,N)$ remains trivial  (\cref{prop:trivial_lebesgue}), hence a geodesic space, regardless of the nature of the target space $N$.
\end{remark}

\begin{proof}[Proof of \cref{prop:length_target}]
    Let $(y,y')\in N^2$ and suppose there exists $\kappa > 1$ such that for all $\gamma \in \mathcal{AC}^p(I,N)$ such that $\gamma(a)=y$ and $\gamma(b)=y'$ we have 
    $$\lebesgue(I)^{p-1}\int_I \lvert \gamma'\rvert_N^p(t)\dif t > \kappa\,d_N(y,y')^p.$$
    Since the measure $\mu_M$ is not purely infinite, we have, by \cref{prop:closed_embedding_target}, that there exists an isometric mapping $\iota: N\to \Lph(M,N)$, so that $\iota(y)$ and $\iota(y')$ belong to $\Lph(M,N)$. In addition, there exists a measurable set $A$ with $\mu_M(A)\in (0,\infty)$ such that $\iota(z)|_A \equiv z$ and $\iota(z)|_{M\setminus A}= h|_{M\setminus A}$ for all $z\in N$. Now, since $\Lph(M,N)$ is a length space, there exists $c \in \mathcal{AC}^p(I,\Lph(M,N))$ such that $c(a) = \iota(y)$, $c(b)=\iota(y')$ and $\lebesgue(I)^{p-1}\int_I \lvert c'\rvert_p^p(t)\dif t \leq \kappa\, D_p(\iota(y),\iota(y'))^p$ (\cref{prop:energy_min_dist}). Then, by \cref{th:charact_ac_curves}, $\tilde{i}_p(c)\in L^p_{\underline{h}}(M,\mathcal{AC}^p(I,N))$ with $\tilde{i}_p(c)(a)= \iota(y)$ and $\tilde{i}_p(c)(b) = \iota(y')$, so that for $\mu_M$-a.e.~$x\in A$ the curve $\tilde{i}_p(c)(x)\in \mathcal{AC}^p(I,N)$ is such that $\tilde{i}_p(c)(x)(a)=y$ and $\tilde{i}_p(c)(x)(b)=y'$. Hence, by assumption, we have that for $\mu_M$-a.e.~$x\in A$
    $$\lebesgue(I)^{p-1}\int_I \lvert \tilde{i}_p(c)(x)'\rvert_N^p(t)\dif t > \kappa\,d_N(\iota(y)(x),\iota(y')(x))^p.$$
    Integrating over $A$ and using the Fubini--Tonelli theorem \cite[Proposition 5.2.1]{cohn2013measure} since $A$ has finite $\mu_M$-measure, we get that 
    \begin{align*}\lebesgue(I)^{p-1} \int_I\int_A \lvert \tilde{i}_p(c)(x)'\rvert_N^p(t)\dif\mu_M(x)\dif t &= \lebesgue(I)^{p-1} \int_A\int_I \lvert \tilde{i}_p(c)(x)'\rvert_N^p(t)\dif t\dif\mu_M(x)\\
    &>\kappa\int_A d_N(\iota(y)(x),\iota(y')(x))^p\dif\mu_M(x).
    \end{align*}
    Thus, recalling that $\lvert c'\rvert_p^p(t) = \int_M \lvert \tilde{i}_p(c)(x)'\rvert_N^p(t)\dif\mu_M(x)$ and that $\iota(y)|_{M\setminus A} = \iota(y')|_{M\setminus A}$, we have that 
    $$\lebesgue(I)^{p-1}\int_I \lvert c'\rvert_p^p(t)\dif t > \kappa\,D_p(\iota(y),\iota(y'))^p,$$
    hence a contradiction. The conclusion that $d_N$ coincides with the induced intrinsic metric $\widehat{d}_N$ (\cref{def:intrinsic_metric}), that is, that $N$ is a length space, then follows from the fact that the intrinsic metric can be equivalently defined as the infimum over the $p$-energy of curves (\cref{prop:energy_min_dist}). When $\Lph(M,N)$ is a geodesic space, the above proof works with $\kappa=1$. Alternatively, the geodesic space case can be proved using \cref{th:charact_geodesics} since, if $\Lph(M,N)$ is a geodesic space, $\iota(y)$ and $\iota(y')$ can be connected by a geodesic $c$ in $\Lph(M,N)$, so that, by \cref{th:charact_geodesics}, the curve $\tilde{i}_p(c)(x)$ is a geodesic connecting $y$ and $y'$ for $\mu_M$-a.e.~$x\in A$. Also, note that the proof does not require the separability assumption on $N$.
\end{proof}

As a direct consequence of \cref{prop:length_lebesgue,prop:length_target}, the length structure of nonlinear Lebesgue spaces is characterized by the length structure of the target space.

\begin{theorem}[Characterization of the length structure]
\label{th:length_structure_lebesgue} Let $h\in L^0(M,N)$ and $p\in (1,\infty)$. Suppose that $\mu_M$ is not purely infinite and that $N$ is complete and separable. Then, the following assertions are equivalent:
\begin{enumerate}[label=(\roman*)]
    \item $\Lph(M,N)$ is a length (resp.~geodesic) space.
    \item $N$ is a length (resp.~geodesic) space.
\end{enumerate}
\end{theorem}

\begin{remark}[Subsequent developments]
    Subsequent work of M.~Bauer, F.~Mémoli, T.~Needham and M.~Nishino \cite{bauer2026metric} recovers, for $p\in (1,\infty)$ and a constant base mapping, the geodesic case of \cref{th:length_structure_lebesgue} in the setting of a nonatomic Polish probability base space and a Polish target space, through a different argument based on a chain of submetries between nonlinear Lebesgue, $Z$-Gromov--Wasserstein and Wasserstein spaces. 
\end{remark}

In the next section, we prove that the geodesic case of \cref{th:length_structure_lebesgue} has some implications on the characterization of curvature bounds in nonlinear Lebesgue spaces using curvature bounds in the target space.

\subsubsection{\texorpdfstring{Characterization of curvature bounds for $p=2$}{Characterization of curvature bounds for p=2}}
\label{sec:charact_curvature}

A notable consequence of \cref{th:charact_geodesics,th:length_structure_lebesgue} is that the curvature of nonlinear Lebesgue spaces, in the sense of Alexandrov (see \cite[Chapter~4]{burago2022course} for a definition), depends solely on the Alexandrov curvature of the target space for $p=2$. Let us first recall the definition of the notion of curvature in the sense of Alexandrov. 

\begin{definition}[Alexandrov curvature]
\label{def:alexandrov_curv}
Let $X$ be a complete geodesic space (\cref{def:geodesic_space}). $X$ is said to be of global \emph{nonnegative curvature} (NNC), respectively of global \emph{nonpositive curvature} (NPC), in the sense of Alexandrov if and only if, for every point $z\in X$ and constant-speed geodesic $\gamma: [0,1]\to X$ (\cref{def:constant_geodesics}), the inequality (see \cite{lebedeva2010curvature} or \cite[Theorem~4.4]{sturm2012space} for global NNC)
$$d_X(z,\gamma(t))^2\geq (1 -t)\, d_X(z,\gamma(0))^2 + t\, d_X(z,\gamma(1))^2 - (1-t)t\,d_X(\gamma(0),\gamma(1))^2$$
 holds for all $t\in [0,1]$ and the converse inequality in the NPC case (see \cite[Corollary~1.5]{sturm1999metric} for global NPC).
A complete geodesic space of global NNC (resp.~NPC) in the sense of Alexandrov is thus called a \emph{global NNC (resp.~NPC) space}.
    
\end{definition}

Then, the Alexandrov curvature of nonlinear Lebesgue spaces for $p=2$ is determined by the Alexandrov curvature of the target space.

\begin{proposition}[Alexandrov curvature of nonlinear Lebesgue spaces]
\label{prop:curvature_lebesgue}
    Let $h\in L^0(M,N)$. Then, either $L^2_h(M,N)$ is trivial (\cref{prop:trivial_lebesgue}), hence is trivially both a global NNC space and a global NPC space, or the following assertions hold:
    \begin{enumerate}[label=(\roman*)]
        \item if $N$ is a separable and global NNC space, $L^2_h(M,N)$ is also a global NNC space.
        \item if $N$ is a global NPC space, $L^2_h(M,N)$ is also a global NPC space.
    \end{enumerate}
\end{proposition}

\begin{proof}
    Let $h\in L^0(M,N)$. When $L^2_h(M,N)$ is trivial, that is, $L^2_h(M,N)=\left\{h\right\}$, it is both a global NNC space and a global NPC space as it contains a single element. Therefore, assume that $L^2_h(M,N)$ is nontrivial.  
    \begin{enumerate}[label=(\roman*),wide]
        \item \emph{Step 1: $L^2_h(M,N)$ is complete and geodesic.} First note that $N$ being complete, separable and geodesic, we have, by \cref{prop:completeness_lebesgue}, that $L^2_h(M,N)$ is complete and, by \cref{prop:length_lebesgue}, that it is geodesic. The proof then comes down to verifying that every geodesic triangle satisfies the condition of NNC (\cref{def:alexandrov_curv}). 
        
        \par\medskip \noindent \emph{Step 2: NNC triangle comparison.} Let $c\in \mathcal{G}([0,1],L^2_h(M,N))$ and $f\in L^2_h(M,N)$. Then, by \cref{th:charact_geodesics}, $\tilde{i}_2(c)(x)$ belongs to $\mathcal{G}([0,1],N)$ for $\mu_M$-a.e.~$x\in M$ and, since $N$ is a global NNC space, we get for $\mu_M$-a.e.~$x\in M$ that the inequality 
    \begin{align*}d_N(f(x),\tilde{i}_2(c)(x)(t))^2 \geq& (1-t)d_N(f(x),\tilde{i}_2(c)(x)(0))^2 + t d_N(f(x),\tilde{i}_2(c)(x)(1))^2 \\
    &- (1-t)t\,d_N(\tilde{i}_2(c)(x)(0),\tilde{i}_2(c)(x)(1))^2
    \end{align*}
    holds for all $t\in [0,1]$. Hence, integrating over $M$ and recalling that $c(t)=(e_{t,2})_*(\tilde{i}_2(c))$ for all $t\in [0,1]$ (\cref{th:charact_ac_curves}), we get that the inequality 
        \begin{equation*}D_2(f,c(t))^2\geq (1 -t)\, D_2(f,c(0))^2 + t\, D_2(f,c(1))^2 - (1-t)t\,D_2(c(0),c(1))^2\end{equation*}
    holds for all $t\in [0,1]$. We thus proved that $L^2_h(M,N)$ is a global NNC space.
    \item The NPC case follows from the same arguments as (i), but with reverse inequalities. In that case, the proof does not require the use of \cref{prop:length_lebesgue}, so that the conclusion holds without assuming that $N$ is separable (\cref{rem:separability_length}). \qedhere
    \end{enumerate} 
\end{proof}

\begin{remark}[Related results in the literature]
    Assertion (ii) of \cref{prop:curvature_lebesgue} already appeared in works (under the assumption that $h$ is constant for some) by J.~Jost (see \cite[Section 2]{jost1994equilibrium} and \cite[Corollary~4.1.1]{jost1997nonpositive}), K.-T.~Sturm (see \cite[Proposition~3.3]{Sturm2001}) and M.~Bačák (see \cite[Proposition~1.2.18]{bacak2014convex}). In particular, for any two Lebesgue mappings, \cite[Proposition~3.3]{Sturm2001} constructs a geodesic in $L^2_h(M,N)$ from pointwise geodesics in $N$ connecting the two chosen Lebesgue mappings and shows that the NPC inequality is satisfied along this particular geodesic. The NPC inequality then yields that this particular geodesic is the unique geodesic connecting the two chosen Lebesgue mappings. Hence, the NPC inequality is satisfied along every geodesic, showing that $L^2_h(M,N)$ is a global NPC space.
\end{remark}

The reverse implication holds when restricting to a measure $\mu_M$ that is not purely infinite, but does not require $N$ to be separable. 

\begin{proposition}[Alexandrov curvature of the target space]
\label{prop:curvature_target}
    Let $h\in L^0(M,N)$. Suppose that $\mu_M$ is not purely infinite. Then, if $L^2_h(M,N)$ is a global NNC (resp.~NPC) space, so is $N$.
\end{proposition}

\begin{remark}[Why it fails when $\mu_M$ is purely infinite]
    When $\lvert N\rvert = 1$, $L^2_h(M,N)$ is trivial (\cref{prop:trivial_lebesgue}), hence both a global NNC and NPC space, but so is $N$ (since it contains a single element). Thus, the conclusion of \cref{prop:curvature_target} holds in that case. However, when $\mu_M$ is purely infinite and $\lvert N\rvert > 1$, we know that $L^2_h(M,N)$ remains trivial  (\cref{prop:trivial_lebesgue}), hence both a global NNC and NPC space, regardless of the nature of the target space $N$.
\end{remark}

\begin{proof}[Proof of \cref{prop:curvature_target}]
    We only prove the NNC case since the NPC case follows from the same arguments, but with reversed inequalities.

    \par\medskip \noindent \emph{Step 1: $N$ is complete and geodesic.} First, since $L^2_h(M,N)$ is complete and $\mu_M$ is not purely infinite, $N$ is complete, by \cite[Proposition 4.2]{serieys2025nonlinear}. Then, $L^2_h(M,N)$ being geodesic, $\mu_M$ not being purely infinite and $N$ being complete, we have, by \cref{prop:length_target}, that $N$ is geodesic. The proof then comes down to verifying that every geodesic triangle satisfies the condition of NNC (\cref{def:alexandrov_curv}). 
    
    \par\medskip \noindent \emph{Step 2: NNC triangle comparison.} Let $\gamma \in \mathcal{G}([0,1],N)$ and $z\in N$. By \cref{prop:closed_embedding_target}, there exists a scaled isometric mapping $\iota: N\to L^2_h(M,N)$, so that the curve $c:t\to \iota(\gamma(t))$ belongs to $\mathcal{G}([0,1],L^2_h(M,N))$ and $f\coloneqq\iota(z)\in L^2_h(M,N)$. Since $L^2_h(M,N)$ is NNC, we have for all $t\in [0,1]$ that 
    $$D_2(f,c(t))^2\geq (1 -t)\, D_2(f,c(0))^2 + t\, D_2(f,c(1))^2 - (1-t)t\,D_2(c(0),c(1))^2.$$
    Hence, $\iota$ being isometric (up to a scaling factor),
    \begin{equation*}d_N(z,\gamma(t))^2\geq (1 -t)\, d_N(z,\gamma(0))^2 + t\, d_N(z,\gamma(1))^2 - (1-t)t\,d_N(\gamma(0),\gamma(1))^2.\end{equation*}
    
    \noindent We thus proved that $N$ is a global NNC space.
\end{proof}

As a direct consequence of \cref{prop:curvature_lebesgue,prop:curvature_target}, the Alexandrov curvature of nonlinear Lebesgue spaces for $p=2$ is characterized by the Alexandrov curvature of the target space.

\begin{theorem}[Characterization of the Alexandrov curvature]
\label{th:curvature_lebesgue}
    Let $h\in L^0(M,N)$. Suppose that $\mu_M$ is not purely infinite and that $N$ is separable. Then, the following assertions are equivalent:
    \begin{enumerate}[label=(\roman*)]
        \item $L^2_h(M,N)$ is a global NNC (resp.~NPC) space.
        \item $N$ is a global NNC (resp.~NPC) space.
    \end{enumerate}
\end{theorem}

\begin{remark}[On the separability assumption on $N$]
    Note that, in the NPC case, the conclusion of \cref{th:curvature_lebesgue} holds without the separability assumption on $N$ (\cref{prop:curvature_lebesgue}). 
\end{remark}

\begin{remark}[Subsequent developments]
    Subsequent work of M.~Bauer, F.~Mémoli, T.~Needham and M.~Nishino \cite{bauer2026metric} recovers, in the setting of nonlinear Lebesgue spaces over nonatomic Polish probability spaces, the conclusion of \cref{th:curvature_lebesgue} by a different argument based on a chain of submetries. As observed by the authors, their conclusions on lower curvature bounds agree with those obtained in \cref{th:curvature_lebesgue}, while the present results are formulated in a more general measure-theoretic setting and notably allow non-constant base mappings. Their framework further gives a more detailed classification of upper Alexandrov curvature bounds and extends the resulting curvature statements to Wasserstein and $Z$-Gromov--Wasserstein spaces.
\end{remark}

In the next section, we prove that, when the target space has a differential structure satisfying appropriate conditions, a notion of velocity can be defined for absolutely continuous curves in nonlinear Lebesgue spaces despite the lack of differential structure.

\subsection{\texorpdfstring{Velocity of absolutely continuous curves for $p > 1$}{Velocity of absolutely continuous curves for p > 1}}

\label{sec:speed_ac}
\subsubsection{Measurable and Lebesgue sections}

First, recall the definition of a \emph{section} of a surjective mapping.
\begin{definition}[Sections]
    Let $E$ and $B$ be sets and $\pi_E: E\to B$ be a surjective mapping. Then, a mapping $s: B \to E$ is called a \emph{section of $\pi_E$} if $\pi_E\circ s = \operatorname{id}_B$. The set of all sections of $\pi_E$ is denoted by $\mathcal{S}(\pi_E)$.
\end{definition}

Asking for a metric on $E$ might be a bit restrictive, so we only require it to be a topological space in the following sections and try to exploit its fibered structure. One can then define the set of \emph{measurable sections}.

\begin{definition}[Measurable sections]
\label{def:measurable_sections}
    Suppose that $E$ is a topological space, that $(B,\Sigma_B)$ is a measurable space and that $\pi_E: E\to B$ is a surjective mapping. Then, a mapping $s: B\to E$ is called a \emph{measurable section of $\pi_E$} if it belongs to $\mathcal{S}(\pi_E)\cap \mathcal{L}^0(B,E)$. Then, define:
    \begin{enumerate}[label=(\roman*)]
        \item $\mathcal{S}^0_{\operatorname{s}}(\pi_E)$ the set of all measurable sections of $\pi_E$ with separable range.
        \item $S^0(\pi_E)$ the set of all equivalence classes of measurable sections of $\pi_E$ with separable range.
    \end{enumerate}
\end{definition}

\begin{remark}[On the definition of measurability]
    Note that the definition of measurable mappings given in \cref{def:measurable_map} holds when the target space is just a topological space, so that $\mathcal{L}^0(B,E)$ makes sense in \cref{def:measurable_sections} despite $E$ being only a topological space.
\end{remark}

Similarly, we can define \emph{Lebesgue sections} given that we only have a \emph{fibered metric} on the topological space $E$, that is, a Borel measurable mapping $m_E: E^2 \to [0,\infty)$ such that its restriction $m_x \coloneqq m_E|_{\pi_E^{-1}(\{x\})^2}$, for some $x\in B$, is a metric on the fiber $\pi_E^{-1}(\{x\})$ which induces the topology of the fiber.

\begin{definition}[Lebesgue sections]
    Let $h\in S^0(\pi_E)$ and $p\in [1,\infty]$. Suppose that $E$ is a topological space, that $(B,\Sigma_B,\mu_B)$ is a measure space, that $\pi_E: E\to B$ is a surjective mapping and that $m_{E}$ is a fibered metric on $E$. Then, a mapping $s: B\to E$ is called \emph{$p$-Lebesgue section} if it belongs to $\mathcal{S}^0_{\operatorname{s}}(\pi_E)$ and $D_{p,\pi_E}(s,h) < \infty$
    with 
    $$D_{p,\pi_E}(s,h)\coloneqq\begin{cases}
        \left(\int_B m_x(s(x),h(x))^p \dif\mu_B(x)\right)^{1/p},&\text{if $p\in [1,\infty)$},\\
        \mu_B\text{-}\esssup_{x\in B} m_x(s(x),h(x)),&\text{if $p=\infty$}.
    \end{cases}$$
    Then, define:
    \begin{enumerate}[label=(\roman*)]
        \item $\mathcal{S}^p_h(\pi_E)$ the set of all $p$-Lebesgue sections of $\pi_E$ with $h$ as base mapping.
        \item $S^p_h(\pi_E)$ the set of all equivalence classes of $p$-Lebesgue sections of $\pi_E$ with $h$ as base mapping.
    \end{enumerate}
    Hence, $D_{p,\pi_E}$ induces a finite metric on $S^p_h(\pi_E)$. When $m_E$ is a metric on $E$ which induces its topology, $S^p_h(\pi_E)$ is a subset of $\Lph(B,E)$ (\cref{def:lebesgue_spaces}) and $D_{p,\pi_E} = D_p$ (\cref{def:lp_metrics}).
\end{definition}
\begin{remark}[On the choice of $h$ for vector bundles]
    When $\pi_E: E \to B$ is a vector bundle, that is, the fibers $(\pi_E^{-1}(\{x\}))_{x\in B}$ are vector spaces, $h$ is usually taken as the zero section (if it is measurable) of $\pi_E$, that is, $h$ maps to each $x\in B$ the identity element of $\pi_E^{-1}(\{x\})$. In that case, we denote $\mathcal{S}^p(\pi_E)\coloneqq \mathcal{S}^p_h(\pi_E)$ and $S^p(\pi_E)\coloneqq S^p_h(\pi_E)$.
\end{remark}

In the next section, we show that the velocity of absolutely continuous curves in differentiable manifolds satisfying suitable conditions can be defined as Lebesgue sections. 

\subsubsection{Velocity of absolutely continuous curves in Finsler manifolds}

Let us first recall the definition of \emph{Finsler manifolds} introduced by R.~S.~Palais \cite[Definition 2.1]{palais1966lusternik} (see \cite[Definition 2.1]{jimenez2011some} for a more recent exposition and a comparison with other existing definitions of Finsler manifolds).

\begin{definition}[Finsler manifolds]
\label{def:finsler}
    Let $k\in \mathbb{N}^*\cup\{\infty\}$. A $\mathcal{C}^k$ manifold $X$ modeled on a Banach space $(\mathbb{B},\lVert \cdot \rVert_\mathbb{B})$ (\cref{def:banach_manifolds}) is called \emph{Finsler} when it is equipped with a mapping $b_X: TX \to [0,\infty)$ such that:
    \begin{enumerate}[label=(\roman*)]
        \item for every $x\in X$, the restriction $b_x\coloneqq b_X|_{T_xX}: T_xX\to [0,\infty)$ is a norm on the tangent space $T_xX$ which induces the manifold topology, that is, for every chart $(U,\varphi)$ such that $x\in U$, the norm $b_x(\dif\varphi^{-1}(\varphi(x))(\cdot))$ is equivalent to $\lVert \cdot \rVert_\mathbb{B}$.
        \item for every $x_0\in X$, $\varepsilon >0$ and chart $(U,\varphi)$ such that $x_0\in U$, there is an open neighborhood $V\subset U$ of $x_0$ such that if $x\in V$ and $v\in \mathbb{B}$
        $$\frac{1}{1+\varepsilon}b_{x_0}(\dif \varphi^{-1}(\varphi(x_0))(v))\leq b_x(\dif \varphi^{-1}(\varphi(x))(v))\leq (1+\varepsilon)\,b_{x_0}(\dif \varphi^{-1}(\varphi(x_0))(v)).$$
    \end{enumerate}
\end{definition}

\begin{remark}[On the continuity of $b_X$]
    Note that, when $(X,b_X)$ is a Finsler manifold as defined in \cref{def:finsler}, the mapping $b_X: TX\to [0,\infty)$ is continuous \cite[Theorem 2.7]{palais1966lusternik}.
\end{remark}

A Finsler structure induces a metric space structure \cite[Theorem 3.3]{palais1966lusternik}.

\begin{proposition}[Metric space structure of Finsler manifolds]
\label{prop:metric_space_finsler}
    Suppose that $(X,b_X)$ is a $\mathcal{C}^1$ Finsler manifold. Then, $X$ is a metric space when equipped with the metric $d_X: X^2 \to [0,\infty]$ defined as 
    $$d_X(x,x')\coloneqq \inf\left\{\int_I b_X(\dot{c}(t))\dif t: c\in \mathcal{C}^1(I,X), c(a)=x,c(b)=x'\right\}.$$
\end{proposition}

In addition, the induced metric is locally equivalent to the norm \cite[Lemma 1.3]{jaramillo2014characterization}.

\begin{proposition}[Local equivalence of the norm and the metric]
\label{prop:property_finsler}
    Let $(X,b_X)$ be a $\mathcal{C}^1$ Finsler manifold. Then, for every $x_0 \in X$, $\varepsilon >0$ and chart $(U,\varphi)$ such that $x_0\in U$, there exists an open neighborhood $V\subset U$ of $x_0$ such that for all $(x,x')\in V$ we have
    $$\frac{1}{1+\varepsilon}d_X(x,x')\leq b_{x_0}(\dif\varphi^{-1}(\varphi(x_0))(\varphi(x) - \varphi(x'))) \leq (1+\varepsilon)\, d_X(x,x').$$
\end{proposition}

We will see that \cref{prop:property_finsler} is useful to show that the norm of the velocity and the metric derivative coincide for absolutely continuous curves. To define a notion of velocity, let us first recall a definition of absolutely continuous curves in Banach spaces that is stronger than \cref{def:ac_curves}.

\begin{definition}[Strongly absolutely continuous curves in Banach spaces]
\label{def:strong_ac_curves}
 Let $p\in [1,\infty]$. When $\mathbb{B}$ is a Banach space, define the set $\mathcal{AC}^p_{\operatorname{s}}(I,\mathbb{B})$ of \emph{strongly absolutely continuous curves}, that is, the set of all continuous curves $c:I\to \mathbb{B}$ such that there is $\eta\in L^p(I,\mathbb{B})$ satisfying 
 $$c(t) - c(s) =\int_s^t \eta(u)\dif u\quad \text{for all $a\leq s\leq t\leq b$ },$$
 where the latter integral is meant in the sense of Bochner (see \cite[Section 1.2.a]{hytonen2016analysis} for a complete exposition).
\end{definition}

\cref{def:strong_ac_curves} extends to Banach manifolds through the use of local charts \cite[Definition 1.2]{schmeding2016manifolds}.

\begin{definition}[Strongly absolutely continuous curves in Banach manifolds]
\label{def:strong_ac_finsler}
    Let $p\in [1,\infty]$. When $X$ is a $\mathcal{C}^1$ Banach manifold (\cref{def:banach_manifolds}), define the set $\mathcal{AC}^p_{\operatorname{s}}(I,X)$ of all continuous curves $c: I\to X$ such that there is a partition $a=t_0< t_1<\ldots< t_n=b$ and charts $(U_i,\varphi_i)$ of $X$ for $1\leq i \leq n$ such that 
    $c([t_{i-1},t_i])\subset U_i$ and $\varphi_i\circ c|_{[t_{i-1},t_i]}\in \mathcal{AC}^p_{\operatorname{s}}([t_{i-1},t_i],\varphi_i(U_i))$.
\end{definition}
\begin{remark}[On the invariance to the choice of partition and charts]
    As pointed out in \cite[Lemma 3.21]{glockner2015measurable}, once a curve satisfies the conditions of \cref{def:strong_ac_finsler} for a partition of $I$ and a collection of charts, the conditions are also satisfied for every other partition $a=t_0< t_1<\ldots< t_n=b$ of $I$ and collection of charts $(U_i,\varphi_i)_{1\leq i\leq n}$ of $X$ such that $c([t_{i-1},t_i])\subset U_i$ for all $1\leq i \leq n$.
\end{remark}

One can then define a Banach bundle of $L^p$ velocities over the set of strongly absolutely continuous curves in Finsler manifolds.

\begin{proposition}[Vector bundle over $\mathcal{AC}^p_{\operatorname{s}}(I,X)$ in Finsler manifolds]
\label{prop:vector_bundle_ac}
    Let $p\in [1,\infty]$. Suppose that $(X,b_X)$ is a $\mathcal{C}^1$ Finsler manifold. Then, the set $L^p(I,TX) \coloneqq \cup_{c\in \mathcal{AC}^p_{\operatorname{s}}(I,X)} S^p(\pi_{c^*TX})$ forms a Banach bundle 
    $$\pi_{L^p(I,TX)}\coloneqq (\pi_{TX})_*: L^p(I,TX)\to \mathcal{AC}^p_{\operatorname{s}}(I,X),$$
     where $c^*TX\coloneqq \{(t,v)\in I\times TX: \pi_{TX}(v) =c(t)\} $ is the pullback by $c$ of the tangent bundle $TX$ and the projection $\pi_{c^*TX}: c^*TX\to I$ is given by the canonical projection on the first coordinate $\pi_{c^*TX}\coloneqq \operatorname{proj}_1$.
\end{proposition}
\begin{remark}[On the structure of $L^p(I,TX)$] 
\label{rem:structure_bundle_ac}
We emphasize that this result says nothing about a potential differentiable structure of $L^p(I,TX)$. This question is left open for now and will be further investigated in future work. Calling $\pi_{L^p(I,TX)}$ a \enquote{Banach bundle} is thus a slight abuse of notation since it is only a surjective mapping having Banach fibers.
\end{remark}
\begin{proof}[Proof of \cref{prop:vector_bundle_ac}]
    Let $p\in [1,\infty]$. For each $c\in \mathcal{AC}^p_{\operatorname{s}}(I,X)$, the zero section of $\pi_{c^*TX}$ is continuous, hence Borel measurable, as it is given by the composition of the zero section of $\pi_{TX}$, which is $\mathcal{C}^1$, and of $c$, which is continuous. Hence, $S^p(\pi_{c^*TX})$ is defined without ambiguity using the norm $b_{p,\pi_{c^*TX}}: S^p(\pi_{c^*TX})\to [0,\infty)$ defined as 
    $$b_{p,\pi_{c^*TX}}(v)\coloneqq\begin{cases}
        \left(\int_I b_{c(t)}(v(t))^p \dif t\right)^{1/p},&\text{if $p\in [1,\infty)$},\\
        \lebesgue\text{-}\esssup_{t\in I} b_{c(t)}(v(t)),&\text{if $p=\infty$}.
    \end{cases}$$
    In addition, $S^p(\pi_{c^*TX})$ is a Banach space. Indeed, $\pi_{c^*TX}$ being a Banach bundle, $S^p(\pi_{c^*TX})$ inherits its vector space structure from the pointwise operations of addition and scalar multiplication and is complete by performing the same proof as \cref{prop:completeness_lebesgue}. Then, we can define 
    $$L^p(I,TX) \coloneqq \cup_{c\in \mathcal{AC}^p_{\operatorname{s}}(I,X)} S^p(\pi_{c^*TX}).$$
    Furthermore, the mapping 
    $$\pi_{L^p(I,TX)}\coloneqq (\pi_{TX})_*: L^p(I,TX)\to \mathcal{AC}^p_{\operatorname{s}}(I,X)$$
    is surjective and its fibers are given by $\pi_{L^p(I,TX)}^{-1}(\{c\}) = S^p(\pi_{c^*TX})$ for all $c\in \mathcal{AC}^p_{\operatorname{s}}(I,X)$. Hence, $\pi_{L^p(I,TX)}$ is a Banach bundle.
\end{proof}

The mapping that associates its velocity to each strongly absolutely continuous curve in a Finsler manifold is then defined as a section of the Banach bundle given by \cref{prop:vector_bundle_ac}.

\begin{proposition}[Velocity of strongly absolutely continuous curves in Finsler manifolds]
\label{prop:speed_strong_ac_curves}
    Let $p\in [1,\infty]$. Suppose that $(X,b_X)$ is a $\mathcal{C}^1$ Finsler manifold. Then, the mapping
    $$\partial_X:\left\{\begin{array}{ccl}
        \mathcal{AC}^p_{\operatorname{s}}(I,X)&\longrightarrow&  L^p(I,TX)\\
        c&\longmapsto& \dot{c}\coloneqq\partial_X c
    \end{array}\right.$$
    is a section of $\pi_{L^p(I,TX)}$.
\end{proposition}

\begin{proof}
    Let $p\in [1,\infty]$ and $c\in \mathcal{AC}^p_{\operatorname{s}}(I,X)$. Thus, there exist a partition $a=t_0< t_1<\ldots< t_n=b$ and charts $(U_i,\varphi_i)$ of $X$ for $1\leq i \leq n$ such that 
    $c([t_{i-1},t_i])\subset U_i$ and $\varphi_i\circ c|_{[t_{i-1},t_i]}\in \mathcal{AC}^p_{\operatorname{s}}([t_{i-1},t_i],\varphi_i(U_i))$. Hence, denoting $\mathbb{B}$ the Banach space on which $X$ is modeled, there exists $\eta_i\in L^p([t_{i-1},t_i],\mathbb{B})$ such that for all $(s,t)\in [t_{i-1},t_i]^2$ satisfying $s<t$ we have 
    $$(\varphi_i\circ c)(t) - (\varphi_i\circ c)(s) = \int_s^t \eta_i(u)\dif u.$$
 Then, we can define for $\lebesgue$-a.e.~$t\in I$
 $$\dot{c}(t)\coloneqq \sum_{i=1}^n \mathds{1}_{[t_{i-1},t_i]}(t) \dif\varphi_i^{-1}(\varphi_i(c(t)))(\eta_i(t))\in T_{c(t)}X,$$
 which does not depend on the choice of charts. Therefore, $\dot{c}$ is measurable with separable range as the sum of measurable mappings with separable ranges and is a section of $\pi_{c^*TX}$. Thus, up to the identification with its class of equivalence, $\dot{c}$ belongs to $S^p(\pi_{c^*TX})$.
\end{proof}

\begin{remark}[Related results in the literature]
    \cref{prop:speed_strong_ac_curves} was proved by W.~P.~A.~Klingenberg \cite[Proposition 2.3.16]{klingenberg1995riemannian} in the case of finite-dimensional Riemannian manifolds and by A.~Schmeding  \cite[Proposition 2.8]{schmeding2016manifolds} in the case of (strong) infinite-dimensional Riemannian manifolds. Both results are corollaries of \cref{prop:speed_strong_ac_curves} since any Hilbert manifold with a (strong) continuous Riemannian metric is a Finsler manifold \cite[Theorem 2.12]{palais1966lusternik}. However, we emphasize that, unlike \cite[Proposition 2.3.16]{klingenberg1995riemannian} and \cite[Proposition 2.8]{schmeding2016manifolds}, \cref{prop:speed_strong_ac_curves} tells nothing about the regularity of the section which maps any absolutely continuous curve to its $L^p$ velocity. This question is left open for now and will be further investigated in future work as it is related to \cref{rem:structure_bundle_ac}.
\end{remark}

Now, one could ask under which conditions on $X$ can the above construction transpose to absolutely continuous curves in the sense of \cref{def:ac_curves}. To answer this question, let us first recall that, when $X$ is a Banach space satisfying the Radon--Nikod\'ym property (see \cite[Section 1.6.2]{lisini2006absolutely} for a definition), the set of absolutely continuous curves defined in \cref{def:ac_curves} coincides with the set of strongly absolutely continuous curves defined in \cref{def:strong_ac_curves} \cite[Theorem~1.16]{lisini2006absolutely}. 

\begin{proposition}[Coincidence of $\mathcal{AC}^p(I,X)$ and $\mathcal{AC}^p_{\operatorname{s}}(I,X)$ in RNP Banach spaces]
\label{prop:speed_ac_rnp}
Let $p\in [1,\infty]$. Suppose that $X$ is a Banach space satisfying the Radon--Nikod\'ym property. Then, $\mathcal{AC}^p(I,X)=\mathcal{AC}_{\operatorname{s}}^p(I,X)$. 
\end{proposition}

Then, \cref{prop:speed_ac_rnp} extends to Finsler manifolds modeled on a Banach space satisfying the Radon--Nikod\'ym property through the use of local charts, so that the notion of strongly absolutely continuous curves given by \cref{def:strong_ac_finsler} and the notion of absolutely continuous curves given by \cref{def:ac_curves} also coincide.

\begin{proposition}[Coincidence of $\mathcal{AC}^p(I,X)$ and $\mathcal{AC}^p_{\operatorname{s}}(I,X)$ in RNP Finsler manifolds]
\label{prop:ac_are_strong_ac}
    Let $p\in [1,\infty]$. Suppose that $(X,b_X)$ is a $\mathcal{C}^1$ Finsler manifold modeled on a Banach space $\mathbb{B}$ satisfying the Radon--Nikod\'ym property. Then, $\mathcal{AC}^p(I,X) = \mathcal{AC}^p_{\operatorname{s}}(I,X)$.
\end{proposition}
\begin{proof}
    Let $p\in [1,\infty]$. The inclusion $\mathcal{AC}^p_{\operatorname{s}}(I,X)\subset \mathcal{AC}^p(I,X)$ follows from \cref{prop:property_finsler}. Let us prove the reverse inclusion. Let $c\in \mathcal{AC}^p(I,X)$. Since $I$ is compact and $c$ is continuous, there exist a partition $a=t_0<t_1<\ldots<t_n=b$ and charts $(U_i,\varphi_i)$ of $X$ for $1\leq i\leq n$ such that $c([t_{i-1},t_i])\subset U_i$. By \cref{prop:property_finsler}, the fact that $c\in \mathcal{AC}^p(I,X)$ implies that $\varphi_i\circ c|_{[t_{i-1},t_i]}\in \mathcal{AC}^p([t_{i-1},t_i],\varphi_i(U_i))$ and, since $\mathbb{B}$ has the Radon--Nikod\'ym property, $\varphi_i\circ c|_{[t_{\bar{i}_1},t_i]}\in \mathcal{AC}^p_{\operatorname{s}}([t_{i-1},t_i],\varphi_i(U_i))$, by \cref{prop:speed_ac_rnp}. Hence, $c\in \mathcal{AC}^p_{\operatorname{s}}(I,X)$.
\end{proof}

As a consequence of \cref{prop:ac_are_strong_ac,prop:property_finsler}, the mapping that associates its velocity to each absolutely continuous curve is defined as in \cref{prop:speed_strong_ac_curves} and the norm of the velocity coincides with the metric derivative, when $X$ has a Finsler structure and is modeled on a Banach space satisfying the Radon--Nikod\'ym property. 

\begin{theorem}[Velocity of absolutely continuous curves in RNP Finsler manifolds]
\label{th:speed_ac_curves}
    Let $p\in [1,\infty]$. Suppose that $(X,b_X)$ is a $\mathcal{C}^1$ Finsler manifold modeled on a Banach space satisfying the Radon--Nikod\'ym property. Then, the mapping
    $$\partial_X:\left\{\begin{array}{ccl}
        \mathcal{AC}^p(I,X)&\longrightarrow& L^p(I,TX)\\
        c&\longmapsto& \dot{c}\coloneqq\partial_X c
    \end{array}\right.$$
    is a section of $\pi_{L^p(I,TX)}$ and satisfies for all $c\in \mathcal{AC}^p(I,X)$ that  
    \begin{align}
    \label{eq:norm_speed_equals_metric_derivative}
    \lvert c'\rvert_X(t) = b_{c(t)}(\dot{c}(t))\quad\text{for $\lebesgue$-a.e.~$t\in I$},
    \end{align}
\end{theorem}

\begin{proof}
    The proof follows by identifying absolutely continuous curves and strongly absolutely continuous curves using \cref{prop:ac_are_strong_ac}, repeating the arguments in the proof of \cref{prop:speed_strong_ac_curves} to construct the velocity and using \cref{prop:property_finsler} to show that the norm of the velocity and the metric derivative coincide.
    
    \par\medskip \noindent \emph{Step 1: Construction of the velocity.} Let $p\in [1,\infty]$ and $c\in \mathcal{AC}^p(I,X)$. By \cref{prop:ac_are_strong_ac}, $c\in \mathcal{AC}^p_{\operatorname{s}}(I,X)$. Hence, there exist a partition $a=t_0< t_1<\ldots< t_n=b$ and charts $(U_i,\varphi_i)$ of $X$ for $1\leq i \leq n$ such that 
    $c([t_{i-1},t_i])\subset U_i$ and $\varphi_i\circ c|_{[t_{i-1},t_i]}\in \mathcal{AC}^p_{\operatorname{s}}([t_{i-1},t_i],\varphi_i(U_i))$. Thus, denoting $\mathbb{B}$ the Banach space on which $X$ is modeled, there exists $\eta_i\in L^p([t_{i-1},t_i],\mathbb{B})$ such that for all $(s,t)\in [t_{i-1},t_i]^2$ satisfying $s<t$ we have 
    $$(\varphi_i\circ c)(t) - (\varphi_i\circ c)(s) = \int_s^t \eta_i(u)\dif u.$$
 Then, we can define for $\lebesgue$-a.e.~$t\in I$
 $$\dot{c}(t)\coloneqq \sum_{i=1}^n \mathds{1}_{[t_{i-1},t_i]}(t) \dif\varphi_i^{-1}(\varphi_i(c(t)))(\eta_i(t))\in T_{c(t)}X,$$
 which does not depend on the choice of charts. Therefore, $\dot{c}$ is measurable with separable range as the sum of measurable mappings with separable ranges and is a section of $\pi_{c^*TX}$. 
 Thus, up to the identification with its class of equivalence, $\dot{c}$ belongs to $S^p(\pi_{c^*TX})$.
 
 \par\medskip \noindent \emph{Step 2: Proof of \cref{eq:norm_speed_equals_metric_derivative}.} Finally, thanks to \cref{prop:property_finsler}, we get that for $\lebesgue$-a.e.~$t\in I$ it holds for all $\varepsilon>0$ that 
 $$\frac{1}{1+\varepsilon}\lvert c'\rvert_X(t)\leq b_{c(t)}(\dot{c}(t))\leq (1+\varepsilon)\lvert c'\vert_X(t).$$
 Hence, we have 
 \begin{equation*}\lvert c'\rvert_X(t) = b_{c(t)}(\dot{c}(t))\quad\text{for $\lebesgue$-a.e.~$t\in I$}.\qedhere\end{equation*}
\end{proof}

\begin{remark}[Related results in the literature]
    \cref{th:speed_ac_curves} was proved by A.~Burtscher \cite[Theorem 2.2 and Proposition 4.10]{burtscher2012length} in the case of finite-dimensional manifolds equipped with continuous Riemannian metrics. This result is a corollary of \cref{th:speed_ac_curves} since any Hilbert manifold with a continuous Riemannian metric is a Finsler manifold \cite[Theorem 2.12]{palais1966lusternik}. Also, since every Banach space is trivially a Finsler manifold, we get back from \cref{th:speed_ac_curves} the well-known fact that the norm of the velocity of an absolutely continuous curve in a Banach space satisfying the Radon--Nikod\'ym property coincides with its metric derivative (see \cite[Theorem 1.16]{lisini2006absolutely} or \cite[Remark 1.1.3]{ambrosio2005gradient}).
\end{remark}

In the next section, we define the velocity of absolutely continuous curves in nonlinear Lebesgue spaces in a similar manner as \cref{th:speed_ac_curves}, when the target space is a Finsler manifold and despite the lack of differential structure on nonlinear Lebesgue spaces outside the linear case. 

\subsubsection{Velocity of absolutely continuous curves in nonlinear Lebesgue spaces}

Let us first define a notion of tangent bundles over nonlinear Lebesgue spaces.

\begin{proposition}[Tangent bundles on nonlinear Lebesgue spaces]
\label{prop:tangent_bundle_lebesgue}
    Let $h\in L^0(M,N)$ and $p\in [1,\infty]$. When $(N,b_N)$ is a $\mathcal{C}^1$ Finsler manifold, define the tangent bundle of $\Lph(M,N)$ as the disjoint union
    $$T\Lph(M,N)\coloneqq\cup_{f\in \Lph(M,N)} S^p(\pi_{f^*TN}),$$
    where $f^*TN\coloneqq\{(x,v)\in M\times TN: \pi_{TN}(v)=f(x)\}$ is the pullback by $f\in \Lph(M,N)$ of the tangent bundle $TN$ and the projection $\pi_{f^*TN}: f^*TN\to M$ is given by the canonical projection on the first coordinate $\pi_{f^*TN}\coloneqq \operatorname{proj}_1$. In addition, the mapping $$\pi_{T\Lph(M,N)}\coloneqq (\pi_{TN})_*: T\Lph(M,N)\to \Lph(M,N)$$
    forms a Banach bundle over $\Lph(M,N)$ and we can define the mapping $B_p:T\Lph(M,N)\to [0,\infty)$ such that for all $f\in \Lph(M,N)$ it induces a norm on the tangent space $T_f\Lph(M,N)\coloneqq S^p(\pi_{f^*TN})$ defined as   
    $$B_p|_{T_f\Lph(M,N)}(z)\coloneqq \begin{cases}
    \left(\int_M b_{f(x)}(z(x))^p\dif\mu_M(x)\right)^{1/p},&\text{if $p\in [1,\infty)$},\\
    \mu_M\text{-}\esssup_{x\in M} b_{f(x)}(z(x)),&\text{if $p=\infty$},
    \end{cases}$$
    which we also denote by $B_f\coloneqq B_p|_{T_f\Lph(M,N)}$ for short and refer to as \emph{$\mathcal{L}^p$ fibered norm}.
\end{proposition}

\begin{remark}[On the structure of $TL^p_h(M,N)$]
    Recall that $L^p_h(M,N)$ does not have a differential structure despite $N$ having one. Similarly, $TL^p_h(M,N)$ has a priori no differential structure, so that calling it a \enquote{tangent bundle} is a slight abuse of notation since it is only the disjoint union of Banach spaces and $\pi_{T\Lph(M,N)}$ is just a surjective mapping with Banach fibers.
\end{remark}

\begin{proof}[Proof of \cref{prop:tangent_bundle_lebesgue}]
    Let $h\in L^0(M,N)$ and $p\in [1,\infty]$. For each $f\in \Lph(M,N)$, the zero section of $\pi_{f^*TN}$ is measurable as it is given by the composition of the zero section of $\pi_{TN}$, which is $\mathcal{C}^1$, and of $f$, which is measurable. Hence, $S^p(\pi_{f^*TN})$ is defined without ambiguity using the norm $B_{p,\pi_{f^*TN}}: S^p(\pi_{f^*TN})\to [0,\infty)$ defined as 
    $$B_{p,\pi_{f^*TN}}(z)\coloneqq\begin{cases}
        \left(\int_M b_{f(x)}(z(x))^p \dif \mu_M(x)\right)^{1/p},&\text{if $p\in [1,\infty)$},\\
        \mu_M\text{-}\esssup_{x\in M} b_{f(x)}(z(x)),&\text{if $p=\infty$}.
    \end{cases}$$
    In addition, $S^p(\pi_{f^*TN})$ is a Banach space. Indeed, $\pi_{f^*TN}$ being a Banach bundle, $S^p(\pi_{f^*TN})$ inherits its vector space structure from the pointwise operations of addition and scalar multiplication and is complete by performing the same proof as \cref{prop:completeness_lebesgue}. Then, we can define 
    $$T\Lph(M,N) \coloneqq \cup_{f\in \Lph(M,N)} S^p(\pi_{f^*TN}).$$
    Furthermore, the mapping 
    $$\pi_{T\Lph(M,N) }\coloneqq (\pi_{TN})_*: T\Lph(M,N) \to \Lph(M,N) $$
    is surjective and its fibers are given by $\pi_{T\Lph(M,N) }^{-1}(\{f\}) = S^p(\pi_{f^*TN})$ for all $f\in \Lph(M,N)$. Hence, $\pi_{T\Lph(M,N)}$ is a Banach bundle. Furthermore, the mapping $B_p:T\Lph(M,N)\to [0,\infty)$ is defined as $B_p|_{T_f\Lph(M,N)}\coloneqq B_{p,\pi_{f^*TN}}$ for all $f\in \Lph(M,N)$.
\end{proof}

There is a priori no guarantee that the mapping $B_p$ defined in \cref{prop:tangent_bundle_lebesgue} is measurable in any sense, hence we might not be able to integrate functionals over $T\Lph(M,N)$ involving $B_p$. To define a topology on $T\Lph(M,N)$ and make $B_p$ continuous on it, we will need a metric on $TN$, which will also be required to apply \cref{th:charact_ac_curves} in order to define the velocity of absolutely continuous curves in nonlinear Lebesgue spaces. Such a metric can be defined as soon as $N$ is equipped with a spray or, equivalently, with a connection (see \cite[Chapter IV, \S3]{lang2012fundamentals} for the definitions of the notions of sprays and connections). 

\begin{proposition}[Splitting norm on $TN$]
\label{prop:splitting_norm}
    Suppose that $(N,b_N)$ is a $\mathcal{C}^2$ Finsler manifold with a spray $S_N$ (or, equivalently, a connection). Then, if $K_N$ denotes the connector associated to $S_N$ (see \cite[Chapter X, Lemma 4.4]{lang2012fundamentals} for a definition) and $T\pi_{TN}$ the tangent mapping of $\pi_{TN}$ (see \cite[Chapter III,\S2]{lang2012fundamentals} for a definition), the $\mathcal{C}^1$ Banach manifold $TN$ can be equipped with the \emph{splitting norm} $m_{TN}: TTN \to [0,\infty)$ defined for all $v\in TN$ and $\xi\in T_vTN$ as
    \begin{equation}
    \label{eq:splitting_norm}
    m_{TN}|_{T_vTN}(\xi)^2 \coloneqq b_{\pi_{TN}(v)}(T\pi_{TN}(\xi))^2 + b_{\pi_{TN}(v)}(K_N(\xi))^2,\end{equation}
    so that $TN$ becomes a metric space when equipped with the induced metric $d_{TN}: TN^2\to [0,\infty]$ defined as
    $$d_{TN}(v,v')\coloneqq\inf\left\{\int_I m_{TN|T_vTN}(\dot{u})\dif t:u\in \mathcal{C}^1(I,TN),u(a)=v,u(b)=v'\right\}.$$
\end{proposition}
\begin{proof}
    The result is a consequence of Dombrowski's splitting theorem \cite[Chapter X, Theorem 4.5]{lang2012fundamentals}, which states that the mapping
    $$(\pi_{TN},T\pi_{TN},K_N): TTN\to TN\oplus TN\oplus TN$$
    is an isomorphism of fiber bundles over $N$. Indeed, the conclusion of \cref{prop:splitting_norm} follows from the same arguments exposed at the end of \cite[Chapter V, \S4]{lang2012fundamentals} or in \cite[Section 2.1]{schmeding2016manifolds}, but defining the splitting norm as in \cref{eq:splitting_norm} using $b_N$ instead of a Riemannian metric. In addition, the induced metric is compatible with the topology of $TN$. 
\end{proof}

\begin{remark}[Borel measurability of the zero section of $\pi_{T\Lph(M,N)}$]
\label{rem:measurability_zero_section}
    When $TN$ is equipped with the splitting norm, its zero section $s_{0,N}$ is isometric from $N$ to $TN$, so that the zero section of $\pi_{T\Lph(M,N)}$ defined as $s_{0,p} \coloneqq (s_{0,N})_*$ is isometric from $\Lph(M,N)$ to $T\Lph(M,N)$.
\end{remark}

As a direct consequence of \cref{prop:splitting_norm}, the tangent bundles over nonlinear Lebesgue spaces defined in \cref{prop:tangent_bundle_lebesgue} become metric spaces as soon as $N$ is equipped with a spray or, equivalently, a connection.

\begin{corollary}[Metric on $TL^p(M,N)$]
\label{cor:metric_tangent_lebesgue}
    Let $h\in L^0(M,N)$ and $p\in [1,\infty]$. Suppose that $(N,b_N)$ is a $\mathcal{C}^2$ Finsler manifold with a spray (or, equivalently, a connection). Then, $T\Lph(M,N)$ is a metric space when equipped with the metric $D_{p,TN}: T\Lph(M,N)^2 \to [0,\infty]$ defined as 
    $$D_{p,TN}(z,z')\coloneqq\begin{cases}
    \left(\int_M d_{TN}(z(x),z'(x))^p\dif\mu_M(x)\right)^{1/p},&\text{if $p\in [1,\infty)$},\\
    \mu_M\text{-}\esssup_{x\in M} d_{TN}(z(x),z'(x)),&\text{if $p=\infty$}.
    \end{cases}$$
\end{corollary}

Similarly to \cref{prop:vector_bundle_ac} and using both \cref{prop:tangent_bundle_lebesgue,cor:metric_tangent_lebesgue}, one can define a Banach bundle of $L^p$ velocities over the set of absolutely continuous curves in nonlinear Lebesgue spaces.

\begin{proposition}[Vector bundle over $\mathcal{AC}^p(I,\Lph(M,N))$]
\label{prop:vector_bundle_ac_lebesgue}
    Let $h\in L^0(M,N)$ and $p\in [1,\infty]$. Suppose that $(N,b_N)$ is a $\mathcal{C}^2$ Finsler manifold with a spray (or, equivalently, a connection). Then, the set 
    $$L^p(I,T\Lph(M,N)) \coloneqq \cup_{c\in \mathcal{AC}^p(I,\Lph(M,N))} S^p(\pi_{c^*T\Lph(M,N)})$$ forms a Banach bundle 
    $$\pi_{L^p(I,T\Lph(M,N))}\coloneqq (\pi_{T\Lph(M,N)})_*: L^p(I,T\Lph(M,N))\to \mathcal{AC}^p(I,\Lph(M,N)),$$
     where $c^*T\Lph(M,N)\coloneqq \{(t,z)\in I\times T\Lph(M,N): \pi_{T\Lph(M,N)}(z) =c(t)\} $ is the pullback by $c$ of the tangent bundle $T\Lph(M,N)$ and the projection $\pi_{c^*T\Lph(M,N)}: c^*T\Lph(M,N)\to I$ is given by the canonical projection on the first coordinate $\pi_{c^*T\Lph(M,N)}\coloneqq \operatorname{proj}_1$. 
\end{proposition}

\begin{remark}[On the structure of $L^p(I,T\Lph(M,N))$] 
As in \cref{rem:structure_bundle_ac}, we emphasize that this result says nothing about a potential topological manifold structure of $L^p(I,T\Lph(M,N))$. This question is left open for now and will be further investigated in future work. Calling $\pi_{L^p(I,T\Lph(M,N))}$ a \enquote{Banach bundle} is thus a slight abuse of notation since it is only a surjective mapping having Banach fibers.
\end{remark}
\begin{proof}[Proof of \cref{prop:vector_bundle_ac_lebesgue}]
    Let $h\in L^0(M,N)$ and $p\in [1,\infty]$. For each $c\in \mathcal{AC}^p(I,L^p_h(M,N))$, the zero section of $\pi_{c^*T\Lph(M,N)}$ is continuous, hence Borel measurable, as the composition of the zero section of $\pi_{T\Lph(M,N)}$, which is continuous  (\cref{rem:measurability_zero_section}), and of $c$, which is also continuous. Hence, $S^p(\pi_{c^*T\Lph(M,N)})$ is defined without ambiguity using the norm $B_{p,\pi_{c^*T\Lph(M,N)}}: S^p(\pi_{c^*T\Lph(M,N)})\to [0,\infty)$ defined as 
    $$B_{p,\pi_{c^*T\Lph(M,N)}}(z)\coloneqq\begin{cases}
        \left(\int_I B_{c(t)}(z(t))^p \dif t\right)^{1/p},&\text{if $p\in [1,\infty)$},\\
        \lebesgue\text{-}\esssup_{t\in I} B_{c(t)}(z(t)),&\text{if $p=\infty$}.
    \end{cases}$$
    In addition, $S^p(\pi_{c^*T\Lph(M,N)})$ is a Banach space. Indeed, $\pi_{c^*T\Lph(M,N)}$ being a Banach bundle, $S^p(\pi_{c^*T\Lph(M,N)})$ inherits its vector space structure from the pointwise operations of addition and scalar multiplication and is complete by performing the same proof as \cref{prop:completeness_lebesgue}. Then, we can define 
    $$L^p(I,T\Lph(M,N)) \coloneqq \cup_{c\in \mathcal{AC}^p(I,\Lph(M,N))} S^p(\pi_{c^*T\Lph(M,N)}).$$
    Furthermore, the mapping 
    $$\pi_{L^p(I,T\Lph(M,N))}\coloneqq (\pi_{T\Lph(M,N)})_*: L^p(I,T\Lph(M,N))\to \mathcal{AC}^p(I,\Lph(M,N))$$
    is surjective and its fibers are given by $\pi_{L^p(I,T\Lph(M,N))}^{-1}(\{c\}) = S^p(\pi_{c^*T\Lph(M,N)})$ for all $c\in \mathcal{AC}^p(I,\Lph(M,N))$. Hence, $\pi_{L^p(I,T\Lph(M,N))}$ is a Banach bundle.
\end{proof}

Then, similarly to \cref{th:speed_ac_curves}, the mapping that associates its velocity to each absolutely continuous curve in a nonlinear Lebesgue space is defined as a section of the Banach bundle given by \cref{prop:vector_bundle_ac_lebesgue} and the norm of the velocity coincides with the metric derivative.

\begin{theorem}[Velocity of absolutely continuous curves in nonlinear Lebesgue spaces for $p> 1$]
\label{th:speed_ac_lebesgue}
    Let $h\in L^0(M,N)$ and $p\in (1,\infty)$. Suppose that $(N,b_N)$ is a connected and complete $\mathcal{C}^2$ Finsler manifold with a spray (or, equivalently, a connection) and that $N$ is modeled on a Banach space satisfying the Radon--Nikod\'ym property. Then, the mapping
    $$\partial_p:\left\{\begin{array}{ccl}
        \mathcal{AC}^p(I,\Lph(M,N))&\longrightarrow& L^p(I,T\Lph(M,N))\\
        c&\longmapsto& \dot{c}\coloneqq\partial_p c
    \end{array}\right.$$
    is a section of $\pi_{L^p(I,T\Lph(M,N))}$ and satisfies for all $c\in \mathcal{AC}^p(I,\Lph(M,N))$ that
    \begin{align}
    \label{eq:norm_speed_equals_metric_derivative_lebesgue}
    \lvert c'\rvert_p(t) = B_{c(t)}(\dot{c}(t))\quad\text{for $\lebesgue$-a.e.~$t\in I$},
    \end{align}
\end{theorem}

\begin{remark}[On relaxing the connectedness assumption on $N$]
    The assumption that $N$ is connected in \cref{th:speed_ac_lebesgue} is solely made to ensure that both $d_N$ and $d_{TN}$ are finite in order to apply \cref{th:charact_ac_curves}. However, as remarked in \cref{rem:relax_metric_charact_ac}, the conclusion of \cref{th:charact_ac_curves} holds when $d_N$ is not assumed finite as long as $h$ is assumed constant. Hence, \cref{th:speed_ac_lebesgue} holds without assuming that $N$ is connected, but requires $h$ to be constant in that case.
\end{remark}

\begin{proof}[Proof of \cref{th:speed_ac_lebesgue}]
    Let $h\in \cLs(M,N)$ and $p\in (1,\infty)$. 
    
    \par\medskip \noindent \emph{Step 1: Construction of the velocity.} Let $c\in\mathcal{AC}^p(I,\Lph(M,N))$. Then, by \cref{th:charact_ac_curves}, $\tilde{i}_p(c)$ belongs to $ L^p_{\underline{h}}(M,\mathcal{AC}^p(I,N))$ and is such that 
    $$\int_M\int_I \lvert \tilde{i}_p(c)(x)'\rvert_N^p(t)\dif t\dif\mu_M(x) < \infty.$$
    Also, the measurable set $M_c\coloneqq \{x\in M: d_\infty(\tilde{i}_p(c)(x),\underline{h}(x)) > 0\}$ has $\sigma$-finite $\mu_M$-measure and there is $\hat{c}\in\cLs(I\times M,N)$ such that $\hat{c}(t,\cdot) = c(t)$ in $\Lph(M,N)$ for all $t\in I$ and $\hat{c}(\cdot,x) = \tilde{i}_p(c)(x)$ in $\mathcal{C}(I,N)$ for $\mu_M$-a.e.~$x\in M$ (\cref{th:extension_embed_iterated_simple,th:charact_ac_curves}). Then, since $N$ is a metric space, it is paracompact \cite[Theorem 5.1.3]{engelking1977general}. Thus, we can choose a locally finite atlas $\{(U_\alpha,\varphi_\alpha)\}_{\alpha\in A}$ of $N$ indexed by some set $A$ and define the sets $W_\alpha \coloneqq U_\alpha\times U_\alpha\subset N\times N$ for all $\alpha\in A$ as well as the neighborhood of the diagonal $\mathcal{O}\coloneqq \cup_{\alpha\in A} W_\alpha$. Since $\mathcal{O}$, equipped with the metric induced from $N\times N$, is a metrizable space, it is paracompact. Hence, there exists a locally finite continuous partition of unity $\{\rho_\alpha\}_{\alpha\in A}$ (\cref{def:partition_unity}) subordinated to the cover $\{W_\alpha\}_{\alpha\in A}$ \cite[Theorem 5.1.9]{engelking1977general}. Since $\{\rho_\alpha\}_{\alpha\in A}$ is locally finite, it can be assumed, up to replacing it with an appropriate partition of unity, that $$\operatorname{supp}_{\mathcal{O}}(\rho_\alpha)\coloneqq \overline{\{(y,y')\in \mathcal{O}:\rho_\alpha(y,y') \neq 0\}}^{\,\mathcal{O}}\subset W_\alpha$$ for every $\alpha\in A$ \cite[Section 16, p.~ 165]{kriegl1997convenient}. Define for every $\tau \neq 0$ the mapping $G_\tau: N\times N\to TN$ as 
    $$G_\tau(y,y')\coloneqq \begin{cases}
        \sum_{\alpha\in A} \rho_\alpha(y,y') \mathrm{d}\varphi_\alpha^{-1}(\varphi_\alpha(y))(\varphi_\alpha(y') - \varphi_\alpha(y))/\tau,&\text{if $(y,y')\in \mathcal{O}$},\\
        s_{0,N}(y), &\text{if $(y,y')\notin \mathcal{O}$},
    \end{cases}$$
    with $s_{0,N}$ the zero section of $\pi_{TN}$, which is continuous since $TN$ is equipped with the splitting metric $d_{TN}$ (\cref{rem:measurability_zero_section}). Hence, $G_\tau|_{\mathcal{O}}$ and $G_\tau|_{(N\times N)\setminus \mathcal{O}}$ are continuous, so that $G_\tau$ is Borel measurable. Then, for every $\tau \neq 0$ define the mapping 
    $$g_\tau:\left\{\begin{array}{ccl}
    I\times \mathcal{AC}^p(I,N)&\longrightarrow& TN\\
    (t,u)&\longmapsto& G_\tau(u(t),\bar{u}(t+\tau))
    \end{array}\right.,$$
    where $\bar{u}$ is the extension of $u$ such that $\bar{u}(t)= u(a)$ if $t< a$ and $\bar{u}(t)= u(b)$ if $t > b$. The mapping $g_\tau$ is Borel measurable from $I\times (\mathcal{AC}^p(I,N),d_\infty)$ to $TN$ for all $\tau\neq 0$ as the composition of such mappings. Hence, the set 
    \begin{align*}
    A&\coloneqq \{(t,x)\in I\times M: (\partial_N \tilde{i}_p(c)(x))(t)\text{ does not exist}\}\\
    &= \{(t,x)\in I\times M: \underset{\tau,\tau'\to 0}{\lim\,\sup}\, d_{TN}(g_\tau(t,\tilde{i}_p(c)(x)),g_{\tau'}(t,\tilde{i}_p(c)(x))) > 0\}
    \end{align*}
    is measurable as, $g_\tau$ being Borel measurable and $\tilde{i}_p(c)$ being measurable from $M$ to $(\mathcal{AC}^p(I,N),d_\infty)$ (\cref{prop:comparison_dp_dinf}), the mapping $(t,x)\mapsto g_\tau(t,\tilde{i}_p(c)(x))$ is measurable from $I\times M$ to $TN$ for every $\tau \neq 0$. In addition, $\lebesgue(A^x) =0$ for $\mu_M$-a.e.~$x\in M$ since $\tilde{i}_p(c)(x)\in \mathcal{AC}^p(I,N)$. Now, since $\tilde{i}_p(c)|_{M\setminus M_c} = \underline{h}|_{M\setminus M_c}$, $A$ is a subset of the measurable rectangle $I\times M_c$ and, $M_c$ having $\sigma$-finite $\mu_M$-measure and $I$ having finite $\lebesgue$-measure, we get, by the Fubini--Tonelli theorem \cite[Proposition 5.2.1]{cohn2013measure}, that $\lebesgue\otimes \mu_M(A)=0$ and $\mu_M(A_t)=0$ for $\lebesgue$-a.e.~$t\in I$. 
    Thus, we can define $\eta: I\times M\to TN$ as 
    $$\eta(t,x)\coloneqq \begin{cases}(\partial_N \tilde{i}_p(c)(x))(t), &\text{if $(t,x)\in (I\times M)\setminus A$},\\
    s_{0,N}(\tilde{i}_p(c)(x)(t)),& \text{if $(t,x)\in A$},
    \end{cases}$$
    which is measurable and separably valued as the pointwise limit of such mappings (\cref{prop:measurability_pointwise_limit}). Indeed, $\eta(t,x) = \lim_{\tau \to 0} g_\tau(t,\tilde{i}_p(c)(x))$ for every $(t,x)\in (I\times M)\setminus A$. In addition, by \cref{prop:comparison_dp_dinf}, $\mathcal{S}\coloneqq\tilde{i}_p(c)(M)$ is separable in $(\mathcal{AC}^p(I,N),d_\infty)$. For every $\tau\neq 0$, the mapping $(t,u)\mapsto (u(t),\bar{u}(t+\tau))$ is continuous from $I\times (\mathcal{AC}^p(I,N),d_\infty)$ to $N\times N$, hence the image $K_\tau$ of $I\times \mathcal{S}$ is separable. Since $G_\tau$ is continuous on $\mathcal{O}$ and agrees with $s_{0,N}\circ\operatorname{proj}_1$ on $(N\times N)\setminus \mathcal{O}$, both $G_\tau(K_\tau\cap \mathcal{O})$ and $G_\tau(K_\tau\setminus \mathcal{O})$ are separable in $(TN,d_{TN})$. Thus, $(t,x)\mapsto g_\tau(t,\tilde{i}_p(c)(x))$ is separably valued. Then,  $\eta(t,x)= (\partial_N \hat{c}(\cdot,x))(t)\in T_{\hat{c}(t,x)}$ for $\lebesgue\otimes \mu_M$-a.e.~$(t,x)\in I\times M$ and, by \cref{th:speed_ac_curves}, we get that for $\lebesgue\otimes\mu_M$-a.e.~$(t,x)\in I\times M$ 
    $$\lvert \tilde{i}_p(c)(x)'\rvert_N(t)= b_{\hat{c}(t,x)}(\eta(t,x)).$$
    Finally, integrating over $I$ and $M$ and using the Fubini--Tonelli theorem \cite[Proposition 5.2.1]{cohn2013measure} yields that 
    \begin{align*}\int_{I\times M} b_{\hat{c}(t,x)}(\eta(t,x))^p\dif\lebesgue\otimes \mu_M(t,x) &= \int_{I\times M} \lvert \tilde{i}_p(c)(x)'\rvert_N^p(t)\dif\lebesgue\otimes \mu_M(t,x)\\
    &= \int_{I\times M_c} \lvert \tilde{i}_p(c)(x)'\rvert_N^p(t)\dif\lebesgue\otimes \mu_M(t,x)\\
    &= \int_{M_c}\int_I \lvert \tilde{i}_p(c)(x)'\rvert_N^p(t)\dif t\dif\mu_M(x)\\
    &= \int_M\int_I \lvert \tilde{i}_p(c)(x)'\rvert_N^p(t)\dif t\dif\mu_M(x) < \infty.
    \end{align*}
    Since $N$ is equipped with a spray (or a connection), this even yields that 
    $$\int_{I\times M} d_{TN}(\eta(t,x),(s_{0,N}\circ \hat{c})(t,x))^p\dif\lebesgue\otimes \mu_M(t,x) <\infty.$$
    Thus, defining $\hat{c}^*s_{0,N}\coloneqq s_{0,N}\circ \hat{c}$, we have that $\eta\in L^p_{\hat{c}^*s_{0,N}}(I\times M, TN)$ and we can define $\dot{c} : t\mapsto \eta(t,\cdot)$. Then, $\dot{c}(t)$ belongs to $L^0(M,TN)$ for a $\lebesgue$-a.e.~$t\in I$. To prove the measurability of $\dot{c}$ from $I$ to $(L^0(M,TN),D_{p,TN})$, note that, by density of $E_{\hat{c}^*s_{0,N},\sigma}(I\times M,TN)\cap L^p_{\hat{c}^*s_{0,N}}(I\times M, TN)$ in $L^p_{\hat{c}^*s_{0,N}}(I\times M, TN)$ (see assertion (i) of \cref{prop:density_iterated_almost_simple}), there exists a sequence $(\hat{g}_n)_{n\in\mathbb{N}}$ in $E_{\hat{c}^*s_{0,N},\sigma}(I\times M,TN)\cap L^p_{\hat{c}^*s_{0,N}}(I\times M, TN)$ such that $\hat{g}_n\to \eta$ in $L^p_{\hat{c}^*s_{0,N}}(I\times M, TN)$ as $n\to \infty$. Thus, define $g_n: t\mapsto \hat{g}_n(t,\cdot)$, which is measurable as an almost simple mapping from $I$ to $(L^0(M,TN),D_{p,TN})$ with base mapping given by the mapping $t\mapsto \hat{c}^*s_{0,N}(t,\cdot)$, which is Borel measurable from $I$ to $(L^0(M,TN),D_{p,TN})$ since $D_{p,TN}(\hat{c}^*s_{0,N}(t,\cdot),c^*s_{0,p}(t))=0$ for all $t\in I$ with $c^*s_{0,p} \coloneqq s_{0,p}\circ c$ the zero section of $\pi_{c^*T\Lph(M,N)}$, which is continuous, hence Borel measurable, as the composition of $s_{0,p}\coloneqq(s_{0,N})_*$, the zero section of $\pi_{T\Lph(M,N)}$, which is continuous (\cref{rem:measurability_zero_section}), and of $c$, which is also continuous. Then, noticing that the union of the sets on which the $\hat{g}_n$'s and $\eta$ differ from $\hat{c}^*s_{0,N}$ is contained in a measurable rectangle $I\times \hat{M}$ such that $\hat{M}$ has $\sigma$-finite $\mu_M$-measure (see step 2 in the proof of assertion (i) of \cref{th:extension_embed_iterated_simple}), we get, using the Fubini--Tonelli theorem \cite[Proposition 5.2.1]{cohn2013measure}, that 
    $$d_{p,p}(g_n,\dot{c})=\hat{D}_p(\hat{g}_n,\eta)\longrightarrow0$$
    as $n\to \infty$, so that, up to the extraction of a subsequence, we can assume for $\lebesgue$-a.e.~$t\in I$ that $D_p(g_n(t),\dot{c}(t))\to 0$ as $n\to \infty$. Hence, $\dot{c}$ is measurable as the pointwise limit of measurable mappings. In addition, since $\mu_M(A_t) = 0$ for $\lebesgue$-a.e.~$t\in I$, we have for $\lebesgue$-a.e.~$t\in I$ that $\dot{c}(t)(x)=\eta(t,x)\in T_{\hat{c}(t,x)}N=T_{c(t)(x)}N$ and $s_{0,N}(\hat{c}(t,x)) = s_{0,N}(c(t)(x))$ for $\mu_M$-a.e.~$x\in M$. Thus, $\dot{c}(t)\in T_{c(t)}\Lph(M,N)$ for $\lebesgue$-a.e.~$t\in I$ and, using the Fubini--Tonelli theorem \cite[Proposition 5.2.1]{cohn2013measure}, we have $\dot{c}\in L^p_{c^*s_{0,p}}(I,T\Lph(M,N))$. Hence, $\dot{c}$ belongs to $S^p(\pi_{c^*T\Lph(M,N)})$. 

    \par\medskip \noindent \emph{Step 2: Proof of \cref{eq:norm_speed_equals_metric_derivative_lebesgue}.} Since $\mu_M(A_t)=0$ for $\lebesgue$-a.e.~$t\in I$, we get that for $\lebesgue$-a.e.~$t\in I$ the equality
    $$\lvert \tilde{i}_p(c)(x)'\rvert_N(t) = b_{c(t)(x)}(\dot{c}(t)(x))$$
    holds for $\mu_M$-a.e.~$x\in M$.
    Hence, integrating over $M$, we have
    \begin{equation*}\lvert c'\rvert_p(t) = \left(\int_M b_{c(t)(x)}(\dot{c}(t)(x))^p\dif\mu_M(x)\right)^{1/p} = B_{c(t)}(\dot{c}(t))\quad\text{for $\lebesgue$-a.e.~$t\in I$ }.\qedhere\end{equation*}
\end{proof}

Finally, combining \cref{prop:length_lebesgue,th:speed_ac_lebesgue} shows that the metric $D_p$ on nonlinear Lebesgue spaces coincides with the metric induced by the infimum of the energy functional associated with the fibered norm $B_p$ (\cref{prop:tangent_bundle_lebesgue}) over absolutely continuous curves. This is analogous to the situation on standard Finsler manifolds (\cref{prop:metric_space_finsler}); the only difference being the class of curves over which the infimum is taken, since nonlinear Lebesgue spaces lack a differential structure.

\begin{corollary}[$\mathcal{L}^p$ metrics are induced by $\mathcal{L}^p$ fibered norms]
\label{cor:min_norm_lebesgue}
    Let $h\in L^0(M,N)$ and $p\in (1,\infty)$. Suppose that $(N,b_N)$ is a connected, complete and separable $\mathcal{C}^2$ Finsler manifold with a spray (or, equivalently, a connection) and that $N$ is modeled on a Banach space satisfying the Radon--Nikod\'ym property. Then, we have for all $(f,f')\in L^p_h(M,N)^2$ that 
    $$D_p(f,f')=\inf\left\{\lebesgue(I)^{1-1/p}\left(\int_I B_{c(t)}(\dot{c}(t))^p\dif t\right)^{1/p}: c\in \mathcal{AC}^p(I,\Lph(M,N)),c(a)=f,c(b)=f'\right\}.$$
\end{corollary}

As a final consequence of \cref{cor:min_norm_lebesgue}, we relate the metric geometry of nonlinear Lebesgue spaces to the geometry induced by the $L^p$ metric on the manifold of smooth mappings (see \cite[Section 42.4]{kriegl1997convenient} for the construction of its manifold structure). Although the $L^p$ metric is a weak Finsler metric (see \cite[Appendix A]{bauer2024lp} for more details on infinite-dimensional Finsler geometry), we show that its induced geodesic distance is point-separating, a property that, unlike in the case of strong Finsler metrics, is not automatic (see \cite{michor2005vanishing} for an example of vanishing geodesic distance associated to a weak Riemannian metric and \cite[Theorem 4]{kolev2008geometric} for a concise statement of the result). Under additional assumptions ensuring the density of smooth mappings and the existence of smooth almost-minimizing interpolations in the target space, its metric completion can moreover be identified with the corresponding nonlinear Lebesgue spaces. For the completion result and its noncompact variant below, we use the following regularity assumptions on $\mu_M$, introduced in \cite[Section 6.1]{serieys2025nonlinear}:
\begin{enumerate}[
    label=\textsc{(oro)},
    ref=\textsc{(oro)},
]
    \item \label{itm:oro} $\mu_M$ is outer regular on $\mathcal{C}_M\cap \mathcal{F}_{\mu_M}$, the set of closed subsets of $M$ with finite $\mu_M$-measure, with respect to $\mathcal{O}_M$, the set of open subsets of $M$.
\end{enumerate}

\begin{enumerate}[
    label=\textsc{(irc)},
    ref=\textsc{(irc)},
]
    \item \label{itm:irc} $\mu_M$ is inner regular on $\mathcal{F}_{\mu_M}$ with respect to $\mathcal{C}_M$.
\end{enumerate}

\begin{enumerate}[
    label=\textsc{(irk)},
    ref=\textsc{(irk)},
]
    \item \label{itm:irk} $\mu_M$ is inner regular on $\mathcal{F}_{\mu_M}$ with respect to $\mathcal{K}_M$, the set of compact subsets of $M$.
\end{enumerate}
The meaning of \emph{regularity} for the measure $\mu_M$ is recalled in \cref{def:regularity_measures}.

\begin{theorem}[On point-separation and the completion of the geodesic distance induced by the $L^p$ metric on the manifold of smooth mappings]
\label{th:lp_separates_smooth}
Let $h\in L^0(M,N)$ and $p\in (1,\infty)$. Suppose that $M$ is a finite-dimensional smooth manifold (see \cite[Chapter 1]{lee2003smooth} for a definition), that $\Sigma_M$ includes $\mathcal{B}(M)$, that $\mu_M$ is positive on nonempty open sets, and that $(N,b_N)$ is a connected, complete and separable smooth Finsler manifold with a spray (or, equivalently, a connection) and that $N$ is modeled on a Banach space satisfying the Radon--Nikod\'ym property. Then, the following assertions hold:
\begin{enumerate}[label=(\roman*)]
    \item the geodesic distance induced by the $L^p$ metric, that is, the mapping $\operatorname{dist}_{L^p}: \mathcal{C}^\infty(M,N)^2\to [0,\infty]$ defined as 
\begin{align*}\operatorname{dist}_{L^p}(f,f')\coloneqq\inf\left\{\left(\int_0^1 B_{c(t)}(\dot{c}(t))^p\dif t\right)^{1/p}: c\in \mathcal{C}^\infty([0,1],\mathcal{C}^\infty(M,N)), c(0)=f,c(1)=f'\right\},\end{align*}
is point-separating on $\mathcal{C}^\infty(M,N)\cap \cLph(M,N)$.
\item if additionally $h$ is constant, $M$ is compact, $\mu_M$ satisfies both \ref{itm:oro} and \ref{itm:irc}, and for every $\kappa > 1$ there exists $H_\kappa\in \mathcal{C}^\infty(N^2\times [0,1],N)$ such that $H_\kappa(y,y',0)= y$, $H_\kappa(y,y',1)= y'$ and $\operatorname{Lip}(H_\kappa(y,y',\cdot))\leq \kappa\, d_N(y,y')$ for every $(y,y')\in N^2$, then the completion of the set of smooth Lebesgue mappings for the $L^p$ metric is isometric to the corresponding nonlinear $p$-Lebesgue space, equipped with its standard metric structure, that is, 
$$\overline{\mathcal{C}^\infty(M,N)\cap \cLph(M,N)}^{\,\operatorname{dist}_{L^p}}\cong \Lph(M,N).$$
\end{enumerate}
\end{theorem}

\begin{remark}[On relaxing the compactness assumption on $M$ in assertion (ii) of \cref{th:lp_separates_smooth}]
\label{rem:reduction_compact_lp_completion}
   The compactness assumption on $M$ can be relaxed in assertion (ii) of \cref{th:lp_separates_smooth} by fixing $z_0\in N$, setting $h\equiv z_0$, assuming that $\mu_M$ satisfies \ref{itm:irk} and replacing $\mathcal{C}^\infty(M,N)$ with
    $$\mathcal{C}^\infty_{z_0}(M,N)\coloneqq \left\{f\in \mathcal{C}^\infty(M,N): \text{there exists $K\subset M$ compact such that $f|_{M\setminus K}\equiv z_0$}\right\}.$$
    Under the remaining assumptions of assertion (ii), one then has $$\overline{\mathcal{C}^\infty_{z_0}(M,N)\cap \cLph(M,N)}^{\,\operatorname{dist}_{L^p}}\cong \Lph(M,N).$$
    Indeed, any two mappings in $\mathcal{C}^\infty_{z_0}(M,N)$ coincide with $z_0$ outside a common compact set. Hence the smooth interpolation used in the proof of \cref{th:lp_separates_smooth} has uniformly compactly supported variation and therefore defines a smooth curve in $\mathcal{C}^\infty(M,N)$. The conclusion then follows from the same argument, together with assertion (ii) of \cite[Theorem 7.1]{serieys2025nonlinear}.
\end{remark}

\begin{remark}[On examples of manifolds satisfying the interpolation condition of assertion (ii) in \cref{th:lp_separates_smooth}]
\label{rem:interpolation_condition_examples}
    The interpolation assumption in assertion (ii) is satisfied, in particular, whenever $N$ admits a smooth mapping $H: N^2\times [0,1],N$ such that $H(y,y',0)= y$, $H(y,y',1)= y'$ and $H(y,y',\cdot)\in \mathcal{G}([0,1],N)$ for every $(y,y')\in N^2$. This includes Hilbert spaces, finite-dimensional Cartan--Hadamard manifolds, and more generally simply connected complete Riemannian manifolds without conjugate points. The latter class is strictly larger than that of Cartan--Hadamard manifolds: indeed, the universal covers of the examples constructed in \cite{gulliver75} are complete simply connected Riemannian manifolds without conjugate points whose sectional curvature takes both positive and negative values. In the infinite-dimensional Finsler setting, another class of examples is provided by simply connected complete Finsler manifolds (\cref{def:finsler}) of seminegative curvature in the sense of K.-H.~Neeb \cite{neeb2002cartan}.
\end{remark}

\begin{proof}[Proof of \cref{th:lp_separates_smooth}]
    Let $h\in \cLs(M,N)$ and $p\in (1,\infty)$ 
    
    \begin{enumerate}[label=(\roman*),wide]
        \item $(f,f')\in (\mathcal{C}^\infty(M,N)\cap \cLph(M,N))^2$ such that $f\neq f'$. Thanks to \cref{cor:min_norm_lebesgue}, we immediately see that $\operatorname{dist}_{L^p}(f,f')\geq D_p(f,f')$, since any smooth curve in $\mathcal{C}^\infty(I,\mathcal{C}^\infty(M,N))$ connecting elements of $\mathcal{C}^\infty(M,N)\cap \cLph(M,N)$ is, in particular, an element of $\mathcal{AC}^p([0,1],L^p_h(M,N))$. Since $f\neq f'$, there exists a point $x_0\in M$ such that $d_N(f(x_0),f'(x_0)) > 0$. By continuity of both $f$ and $f'$, there exist a constant $\varepsilon >0$ and a neighborhood $U_{x_0}$ of $x_0$ such that $d_N(f(x),f'(x))\geq \varepsilon$ for every $x\in U_{x_0}$. Since $\mu_M$ is positive on nonempty open sets, we get
    $$D_p(f,f')^p= \int_Md_N (f(x),f'(x))^p\dif\mu_M(x)\geq \varepsilon^p \mu_M(U_{x_0}) > 0$$
    Hence $\operatorname{dist}_{L^p}(f,f') > 0$, which proves the point-separation property.
    \item Let $(f,f')\in (\mathcal{C}^\infty(M,N)\cap \cLph(M,N))^2$ and $\kappa > 1$. By assumption, there exists $H_\kappa \in \mathcal{C}^\infty(N^2\times [0,1],N)$ such that $H_\kappa(y,y',0)= y$, $H_\kappa(y,y',1)= y'$ and $\operatorname{Lip}(H_\kappa(y,y',\cdot))\leq \kappa\,d_N(y,y')$ for every $(y,y')\in N^2$. Define $c_\kappa(t,x)\coloneqq H_\kappa(f(x),f'(x),t)$ for every $(t,x)\in [0,1]\times M$. Since $H_\kappa$, $f$ and $f'$ are smooth, $c_\kappa$ is smooth too and it can be identified with an element of $\mathcal{C}^\infty([0,1],\mathcal{C}^\infty(M,N))$ \cite[Theorem 42.14]{kriegl1997convenient}. In addition, it satisfies  
    $$b_N(\partial_t c_\kappa(t,x))\leq \kappa\,d_N(f(x),f'(x))\quad\text{for every $x\in M$},$$
    so that, integrating over $M$ and $[0,1]$, we get that  
    $$\left(\int_0^1 B_{c_\kappa(t)}(\dot{c}(t))^p\dif t\right)^{1/p}\leq \kappa\,D_p(f,f').$$
    Hence, by definition of $\operatorname{dist}_{L^p}$,
    $$\operatorname{dist}_{L^p}(f,f')\leq \kappa\,D_p(f,f').$$
    Since the above inequality holds for every $\kappa >1$ and the inequality $D_p\leq \operatorname{dist}_{L^p}$ follows from the proof of assertion (i), we get that $\operatorname{dist}_{L^p}= D_p$ on $\mathcal{C}^\infty(M,N)\cap \cLph(M,N)$. Thus, since $\Lph(M,N)$ is complete (\cref{prop:completeness_lebesgue}) and since $\mathcal{C}^\infty(M,N)\cap \cLph(M,N)$ determines a dense subspace of $\Lph(M,N)$ \cite[Theorem 7.1]{serieys2025nonlinear}, we finally get that \begin{equation*}\overline{\mathcal{C}^\infty(M,N)\cap \cLph(M,N)}^{\,\operatorname{dist}_{L^p}} \cong \Lph(M,N).\qedhere\end{equation*}
    \end{enumerate}
\end{proof}

\begin{remark}[Related results in the literature] Regarding both assertions:
 \begin{enumerate}[label=(\roman*)]
     \item The point-separation property of the geodesic distance induced by the (weak) $L^2$ metric has already been studied by B.~Clarke. In particular, \cite{clarke2010metric} establishes this property on the manifold of smooth Riemannian metrics over a closed manifold (a more concise argument using \cite[Theorem 2.1]{clarke2013geodesics} is also provided at the beginning of \cite[Section 2.2]{clarke2013geodesics}). The pointwise length argument underlying the proof of \cref{th:lp_separates_smooth} also recovers the point-separation result of \cite{clarke2010metric} by considering a closed Riemannian manifold $(M,g_M)$ and $N=S^2_+T^*M\coloneqq \cup_{x\in M}S^2_+T^*_xM$, with $S^2_+T^*_xM$ denoting the cone of symmetric positive definite $(0,2)$-tensors based at $x\in M$. Indeed, B.~Clarke equips the vertical tangent bundle of $S^2_+T^*M$ with the fiberwise Riemannian metric defined as
     $$\langle b,c\rangle_a = \operatorname{tr}_a(bc)\sqrt{\det(g_M(x)^{-1} a)}$$
     for every $a\in S^2_+T^*_xM$  and $(b,c)\in (T_aS^2_+T^*_xM)^2$. Extending $\langle\cdot,\cdot\rangle$ to a Riemannian metric $g_N$ on $N=S^2_+T^*M$, the induced $L^2$ metric on $\mathcal{C}^\infty(M,S^2_+T^*M)$ restricts to the Ebin metric on the manifold of smooth Riemannian metrics, namely the space of smooth sections $\Gamma(S^2_+T^*M)\subset \mathcal{C}^\infty(M,S^2_+T^*M)$. See the beginning of \cite[Section 2.2]{clarke2013geodesics} for details on the complete construction. Note that the fiberwise metric is not complete; however, completeness is not required for the pointwise length argument invoked here.
 \item The identification of metric completions of weak $L^2$-type geometries with nonlinear Lebesgue spaces has appeared in related settings. For the Ebin metric, B.~Clarke characterized the metric completion of the manifold of Riemannian metrics in \cite{clarke2009thesis,clarke2013completion}, while N.~Cavallucci \cite{cavallucci2023l2} subsequently provided a shorter proof exhibiting it explicitly as a nonlinear $2$-Lebesgue space with $\operatorname{CAT}(0)$ target space. An analogous result for the generalized Ebin metric on full-rank vector-valued one-forms was obtained by N.~Cavallucci and Z.~Su in \cite{cavallucci2023metric}.  
 \end{enumerate}
\end{remark}

To illustrate the role of the smooth interpolation condition in assertion (ii) in \cref{th:lp_separates_smooth}, we provide a counterexample showing that the conclusion may fail when this condition is not satisfied.

\begin{cexample}[When the smooth interpolation condition of assertion (ii) in \cref{th:lp_separates_smooth} is not satisfied]
    Let $h\equiv 1$, $p\in (1,\infty)$ and $M=N=S^1$ be the $1$-dimensional sphere equipped with its standard structure. If $f\equiv 1$ and $f'\coloneqq \operatorname{id}_{S^1}$, then $D_p(f,f')<\infty$, but $\operatorname{dist}_{L^p}(f,f') =\infty$. Indeed, suppose that there exists a smooth curve $c\in \mathcal{C}^\infty([0,1],\mathcal{C}^\infty(S^1,S^1))$ connecting $f$ to $f'$. Then, $c$ defines a homotopy between $f$ and $f'$. However, homotopic mappings from $S^1$ to itself have the same degree \cite[Proposition 17.36]{lee2003smooth} (see \cite[Theorem 17.35]{lee2003smooth} for a definition of the degree of smooth mappings). By the definition of degree, constant mappings have degree $0$, while the identity has degree $1$. Hence, $\operatorname{deg}(f) = 0\neq 1= \operatorname{deg}(f')$, which yields a contradiction.
    \end{cexample}


\section{Notation and terminology}
\label{sec:notations}

This section clarifies some recurrent notions and notation used in the present article:
\subsection*{Spaces}
\begin{itemize}
    \item $\mathbb{N}$ denotes the set of nonnegative integers (including $0$). 
    \item $\mathbb{N}^*$ denotes the set of positive integers (excluding $0$). 
    \item $\mathbb{R}$ denotes the set of real numbers, also referred to as the \emph{real line}. 
    \item $\mathbb{R}_+$ denotes the set of nonnegative real numbers (including $0$).
    \item $\mathbb{R}^*$ denotes the set of nonzero real numbers (excluding $0$). 
     \item $\mathbb{Q}$ denotes the set of rational numbers.
\end{itemize}
\subsection*{Topology}

\begin{itemize} 
   \item $\mathcal{O}_X$ denotes the set of open subsets of the topological space $X$.
   \item $\mathcal{C}_X$ denotes the set of closed subsets of the topological space $X$.
   \item $\mathcal{K}_X$ denotes the set of compact subsets of the topological space $X$.
   \item $f(X)\coloneqq \left\{ f(x): x\in X\right\}$ denotes the \emph{range} of a mapping $f: X\to Y$ between sets $X$ and $Y$.
   \item $\operatorname{carr}(f)\coloneqq\{x\in X: f(x)\neq 0\}$ denotes the \emph{carrier} of a mapping $f: X\to \mathbb{R}$ defined on a set $X$
   \item $\overline{A}$ denotes the \emph{closure} of the subset $A \subset X$ in the topological space $X$, that is, the intersection of all closed subsets of $X$ that include $A$.
    \item $\overline{A}^{d_X}$ denotes the closure of the subset $A \subset X$ in the metric space $(X,d_X)$.
    \item $\operatorname{supp}(f)\coloneqq\overline{\operatorname{carr}(f)}$ denotes the \emph{support} of a mapping $f: X\to \mathbb{R}$ defined on a topological space $X$.
    \item A set $X$ is called \emph{countable} if there exists an injective mapping $\varphi: X\to \mathbb{N}$.
    \item A subset $S$ of a (semi-)metric space $(X,d_X)$ is called \emph{dense} when for each element $x$ of $X$ and real number $\varepsilon > 0$ there exists an element $s$ of $S$ such that $d_X(x,s) < \varepsilon$.
    \item A metric space $(X,d_X)$ is called \emph{separable} when it has a countable dense subset.
    \item A sequence $(x_n)_{n\in\mathbb{N}}$ in a metric space $(X,d_X)$ is called \emph{Cauchy} when for all real $\varepsilon > 0$ there exists an integer $n_0$ such that for all $n\geq n_0$ and $m\geq n_0$, $d_X(x_n,x_m) < \varepsilon$.
    \item A metric space $(X,d_X)$ is called \emph{complete} when every Cauchy sequence converges to an element of $X$.
    \item $A^c$ denotes the \emph{complement} of a subset $A$ of a set $X$, that is, $A^c \coloneqq \left\{x \in X: x\notin A\right\}$.
    \item $A\setminus B\coloneqq A \cap B^c$ denotes the \emph{set difference} between two subsets $A$ and $B$ of a set $X$. 
    \item $A\Delta A'\coloneqq (A\setminus A')\cup (A'\setminus A)$ denotes the \emph{symmetric set difference} between two subsets $A$ and $A'$ of a set $X$.
\end{itemize}

\subsection*{Measure theory}
\begin{itemize}
    \item $\mathscr{L}^d$ denotes the $d$-dimensional Lebesgue measure.
    \item $\#$ denotes the counting measure.
    \item A measure $\mu_M$ on a measurable space $(M,\Sigma_M)$ is called \emph{purely infinite} if its range is reduced to $\{0,\infty\}$.
    \item $\sigma(\mathcal{C})$ denotes the smallest $\sigma$-algebra that contains a family $\mathcal{C}$ of subsets of a set $M$.
    \item A subset of a measurable space $(M,\Sigma_M)$ is called \emph{measurable} when it belongs to $\Sigma_M$.
    \item $\mathcal{F}_{\mu_M}$ denotes the set of measurable subsets of a measure space $(M,\Sigma_M,\mu_M)$ with finite $\mu_M$-measure.
    \item $\Sigma_M|_B\coloneqq \left\{B\cap A: A\in \Sigma_M\right\}$ denotes the restriction of a $\sigma$-algebra $\Sigma_M$ on a set $M$ to subsets of $B\in \Sigma_M$.
    \item $\mu_M|_B$ denotes the restriction of a measure $\mu_M$ on a measurable space $(M,\Sigma_M)$ to $\Sigma_M|_B$ for some $B\in\Sigma_M$.
    \item $\mathcal{B}(M)$ denotes the \emph{Borel $\sigma$-algebra} on a topological space $M$.
    \item A measure $\mu_M$ on a topological space $M$ is called \emph{Borel} when it is defined on $\mathcal{B}(M)$.
    \item A subset $Z$ of a measure space $(M,\Sigma_M,\mu_M)$ is called \emph{$\mu_M$-null} if there exists $A\in \Sigma_M$ such that $Z\subset A$ and $\mu_M(A) = 0$. 
    \item $\mathcal{Z}_{\mu_M}$ denotes the set of all $\mu_M$-null sets of a measure space $(M,\Sigma_M,\mu_M)$.
    \item A property $\mathcal{P}$ that depends on the choice of point $x$ of a measure space $(M,\Sigma_M,\mu_M)$ is said to hold \emph{$\mu_M$-almost everywhere} ($\mu_M$-a.e. for short) if $\left\{x\in M: \text{$\mathcal{P}(x)$ is false}\right\}$ is a $\mu_M$-null set.
    \item In a measure space $(M,\Sigma_M,\mu_M)$, $\overline{\Sigma}_M$ denotes the \emph{completion of the $\sigma$-algebra $\Sigma_M$} with respect to $\mu_M$ and is defined as $\overline{\Sigma}_M\coloneqq \sigma(\Sigma_M \cup \mathcal{Z}_{\mu_M}) = \{ A\cup Z:(A,Z)\in \Sigma_M\times \mathcal{Z}_{\mu_M}\}$.
    \item In a measure space $(M,\Sigma_M,\mu_M)$, $\overline{\mu}_M$ denotes the \emph{completion of the measure $\mu_M$} and is defined as the mapping $\bar{\mu}_M: \overline{\Sigma}_M\to [0,\infty]$ such that $\bar{\mu}_M(A\cup Z)=\mu_M(A)$ for all $(A,Z)\in \Sigma_M\times \mathcal{Z}_{\mu_M}$. It is itself a measure on $(M,\overline{\Sigma}_M)$.
\end{itemize}
\label{sec:notations-end}

\section*{Acknowledgements}
\addcontentsline{toc}{section}{Acknowledgements}

This work is part of the author's PhD thesis. The author thus thanks his advisor Joan Alexis Glaun\`es (MAP5, Université Paris Cité) for his support, useful advice, and for introducing him to the subject of metamorphoses of manifold-valued images, which led him to dig into the theory of nonlinear Lebesgue spaces. The author also thanks Th\'eo Dumont (LIGM, Université Gustave Eiffel) for suggestions that significantly helped improve the exposition of the present article as well as Alain Trouv\'e (Centre Borelli, ENS Paris-Saclay) for careful proofreading of early versions of this article and fruitful technical discussions.

\appendix

\renewcommand{\theproposition}{\Alph{section}.\arabic{proposition}} 
\renewcommand{\thelemma}{\Alph{section}.\arabic{lemma}} 
\renewcommand{\thecorollary}{\Alph{section}.\arabic{corollary}} 
\renewcommand{\thedefinition}{\Alph{section}.\arabic{definition}} 
\renewcommand{\thetheorem}{\Alph{section}.\arabic{theorem}} 
\renewcommand{\theremark}{\Alph{section}.\arabic{remark}} 

\section{Omitted results}

\subsection{Comparison of Borel structures on absolutely continuous curves}
\label{appendix:comparison_dp_dinf}

Throughout this section, let $p\in(1,\infty)$ and let $(X,d_X)$ be a generic metric space. We prove \cref{prop:comparison_dp_dinf}. The main ingredient is that $d_p$ and $d_\infty$ induce the same topology on every sublevel of the $p$-energy (\cref{def:energy}). 

\begin{proposition}[Continuity of the identity mapping from $d_p$ to $d_\infty$]
\label{prop:continuity_id_dp_dinf}
    Let $p\in (1,\infty)$, $R \geq 0$ and define the set $$K_R\coloneqq \left\{c\in \mathcal{AC}^p(I,X): \int_I \lvert c'\rvert_X^p(t)\dif t \leq R \right\}.$$ 
    Then, the identity mapping $\operatorname{id}_R: (K_R,d_p)\to (K_R,d_\infty)$ is continuous.
\end{proposition}

\begin{proof}
    Let $p\in (1,\infty)$ and $R\geq 0$. When $R=0$, all curves in $K_R$ are constant and $d_p(c,c') = \lebesgue(I)^{1/p} d_\infty(c,c')$ for every $(c,c')\in K_R^2$, so $\operatorname{id}_R$ is continuous. Let us assume that $R > 0$ in the remainder of the proof. 

    \par\medskip \noindent \emph{Step 1: Uniform Hölder continuity of curves in $K_R$.} Let us first observe that, if $c\in K_R$, we have for all $(s,t)\in I^2$ that, by Hölder's inequality 
    $$d_X(c(s),c(t)) \leq \int_{\min\{s,t\}}^{\max\{s,t\}} \lvert c'\rvert_X(t)\dif t \leq \lvert t - s\rvert^{1-1/p} \left(\int_I \lvert c'\rvert_X^p(t)\dif t\right)^{1/p}\leq \lvert t - s\rvert^{1-1/p} R^{1/p}.$$
    Thus, we have that for all $(c,c')\in K_R^2$ and $(s,t)\in I^2$ that 
    \begin{equation}
    \label{eq:bound_curves}
    \lvert d_X(c(t),c'(t)) - d_X(c(s),c'(s))\rvert \leq 2\,R^{1/p} \lvert t - s\vert^{1- 1/p}.\end{equation}
    
    \par\medskip \noindent \emph{Step 2: Continuity of the identity mapping.} Let $(c_n)_{n\in\mathbb{N}}$ be a sequence in $K_R$ and $c$ be a curve in $K_R$ such that $d_p(c_n,c)\to 0$ as $n\to \infty$. Suppose that $d_\infty(c_n,c)\not\to 0$ as $n\to \infty$. Then, up to the extraction of a subsequence, we can assume that there exists $\varepsilon > 0$ such that $d_\infty(c_n,c) \geq 2\varepsilon$ for all $n\in \mathbb{N}$. Now, since the mapping $t\mapsto d_X(c_n(t),c(t))$ is continuous for all $n\in\mathbb{N}$, there exists $t_n \in I$ such that $d_X(c_n(t_n),c(t_n)) \geq \varepsilon$. Then, define $\delta \coloneqq (\varepsilon/(4 R^{1/p}))^{1/(1-1/p)}$, so that, using \cref{eq:bound_curves}, we have for all $t\in I$ such that $\lvert t - t_n\rvert \leq \delta$ 
    $$d_X(c_n(t),c(t))\geq d_X(c_n(t_n),c(t_n)) - \lvert d_X(c_n(t),c(t)) - d_X(c_n(t_n),c(t_n))\rvert \geq \varepsilon - \varepsilon/2 \geq \varepsilon /2.$$
    Therefore, since $\lebesgue(I\cap [t_n - \delta, t_n +\delta]) \geq \min\{\delta, b-a\}$, we get that 
    $$d_p(c_n,c)^p = \int_I d_X(c_n(t),c(t))^p\dif t \geq (\varepsilon/2)^p \min\{\delta, b -  a\} > 0.$$
    Since this holds for every $n\in\mathbb{N}$, this yields a contradiction. Hence, $d_\infty(c_n,c)\to 0$ as $n\to \infty$.
\end{proof}

As an immediate consequence of \cref{prop:continuity_id_dp_dinf}, the $p$-energy is lower semicontinuous on $(\mathcal{AC}^p(I,X),d_p)$.

\begin{corollary}[Lower semicontinuity of the $p$-energy for $d_p$]
\label{cor:lsc_energy}
    Let $p\in (1,\infty)$. Then, the $p$-energy viewed as a mapping 
    $$E_{p,X}: \left\{\begin{array}{ccl}
    (\mathcal{AC}^p(I,X),d_p)&\longrightarrow& [0,\infty)\\
    c&\longmapsto& \int_I \lvert c'\rvert_X^p(t)\dif t
\end{array}\right.$$
is lower semicontinuous.
\end{corollary}

\begin{proof}
    Let $p\in (1,\infty)$ and let $(c_n)_{n\in\mathbb{N}}$ be a sequence in $(\mathcal{AC}^p(I,X),d_p)$ converging to $c\in \mathcal{AC}^p(I,X)$. Then, let $L\coloneqq \underset{n\to \infty}{\lim\inf} E_{p,X}(c_n)$. If $L= \infty$, there is nothing to prove. If $L <\infty$, we can assume, up to the extraction of a subsequence, that $E_{p,X}(c_n)\to L$ as $n\to \infty$ and there exists $R > 0$ such that $E_{p,X}(c_n) \leq R$ for all $n\in\mathbb{N}$. Then, defining $R'\coloneqq \max\{R, E_{p,X}(c)\}$ and  $K_{R'}\coloneqq\{ c\in \mathcal{AC}^p(I,X): E_{p,X}(c)\leq R'\}$, we have, by continuity of the identity mapping $\operatorname{id}_{R'}: (K_{R'},d_p)\to (K_{R'},d_\infty)$ (\cref{prop:continuity_id_dp_dinf}) that $d_\infty(c_n,c)\to 0$ as $n\to \infty$. Hence, by the lower semicontinuity of $E_{p,X}$ on $(\mathcal{AC}^p(I,X),d_\infty)$ \cite[p.~18]{ambrosio2024metric}, $E_{p,X}(c)\leq \underset{n\to \infty}{\lim\inf} E_{p,X}(c_n)$.
\end{proof}

We next use \cref{prop:continuity_id_dp_dinf,cor:lsc_energy} to prove \cref{prop:comparison_dp_dinf}.

\begin{proof}[Proof of \cref{prop:comparison_dp_dinf}]
    Let $p\in (1,\infty)$.

    \begin{enumerate}[label=(\roman*),wide]
        \item \emph{Step 1: Borel sets for $d_p$ are Borel for $d_\infty$.} Since $d_p\leq \lebesgue(I)^{1/p}d_\infty$, the identity mapping $(\mathcal{AC}^p(I,X),d_\infty)\to (\mathcal{AC}^p(I,X),d_p)$ is continuous. Hence, $\mathcal{B}(\mathcal{AC}^p(I,X),d_p)\subset \mathcal{B}(\mathcal{AC}^p(I,X),d_\infty)$. 
        
        \par\medskip \noindent \emph{Step 2: Borel sets for $d_\infty$ are Borel for $d_p$.} Let $\mathcal{O}$ be an open subset of $(\mathcal{AC}^p(I,X),d_\infty)$ and define $K_n\coloneqq \{c\in \mathcal{AC}^p(I,X): E_{p,X}(c) \leq n\}$ for all $n\in\mathbb{N}$. Then, $\mathcal{O}= \cup_{n\in\mathbb{N}} (\mathcal{O}\cap K_n)$ and, by the continuity of the identity mapping $\operatorname{id}_n: (K_n,d_p)\to (K_n,d_\infty)$ (\cref{prop:continuity_id_dp_dinf}), we know that $\mathcal{O}\cap K_n$ is open in $(K_n,d_p)$. However, since $E_{p,X}$ is lower semicontinuous on $(\mathcal{AC}^p(I,X),d_p)$ (\cref{cor:lsc_energy}), $K_n$ is a Borel subset of $(\mathcal{AC}^p(I,X),d_p)$, hence $\mathcal{O}\cap K_n$ is one too. Therefore, $\mathcal{O}$ being the countable union of Borel subsets of $(\mathcal{AC}^p(I,X),d_p)$, $\mathcal{O}$ is itself Borel for $d_p$. Hence, $\mathcal{B}(\mathcal{AC}^p(I,X),d_\infty)\subset \mathcal{B}(\mathcal{AC}^p(I,X),d_p)$.

    \item Let $\mathcal{S}\subset \mathcal{AC}^p(I,X)$. 
    
    \par\medskip \noindent \emph{Step 1: Separable sets for $d_\infty$ are separable for $d_p$.} The fact that when $(\mathcal{S},d_\infty)$ is separable, $(\mathcal{S},d_p)$ must also be separable follows from the fact that $d_p\leq \lebesgue(I)^{1/p}d_\infty$. 
    
    \par\medskip \noindent \emph{Step 2: Separable sets for $d_p$ are separable for $d_\infty$.} Assume that $(\mathcal{S},d_p)$ is separable. First, observe that $\mathcal{S} = \cup_{n\in \mathbb{N}} (\mathcal{S}\cap K_n)$. Since $(\mathcal{S},d_p)$ is separable, so is $(\mathcal{S}\cap K_n,d_p)$ \cite[(3.10.9)]{dieudonne1960treatise}. Then, by continuity of the identity mapping $\operatorname{id}_n: (K_n,d_p)\to (K_n,d_\infty)$ (\cref{prop:continuity_id_dp_dinf}), $(\mathcal{S}\cap K_n,d_\infty)$ is separable as the continuous image of a separable set. As the countable union of separable sets for $d_\infty$, $(\mathcal{S},d_\infty)$ is separable.\qedhere
    \end{enumerate}
\end{proof}

\subsection{\texorpdfstring{Relaxation of assertion (iii) in \cref{prop:density_iterated_almost_simple}}{Relaxation of the assertion (iii) in Proposition}}

Let us provide a proof of the alternative density result mentioned in \cref{rem:relax_density}.
\begin{proposition}[Relaxation of the assumptions in (iii) of \cref{prop:density_iterated_almost_simple}]
\label{prop:alternative_density}
Let $h\in E(M,N)$ and $p\in [1,\infty)$. Suppose that $(I,\Sigma_I,\mu_I)$ is a measure space, with no further assumptions. Then, the closure of the target space of $\operatorname{sec}_{M,p}$ is 
$$\overline{E_{\underline{h}}(M, E(I,N))\cap L^p_{\underline{h}}(M,L^0(I,N))}^{\,D_{p,p}}= L^p_{\underline{h}}(M,L^0(I,N)).$$
\end{proposition}

\begin{proof}
    Let $h\in \mathcal{E}(M,N)$ and $p\in [1,\infty)$. Then, we can assume that there exists a finite indexing set $J$ such that $h(M)=\{y_j\in N:j\in J\}$ and define $B_j\coloneqq h^{-1}(\{y_j\})$, so that $M = \cup_{j\in J} B_j$ and we have $L^p_{\underline{h}}(M,L^0(I,N))\cong \prod_{j\in J} L^p_{\underline{h}}(B_j,L^0(I,N))$ since $J$ is finite (otherwise the gluing mapping may produce non-integrable mappings). However, on each $B_j$, we get that $L^p_{\underline{h}}(B_j,L^0(I,N)) = L^p_{\underline{h}}(B_j,L^p_{h_j}(I,N))$ with $h_j\equiv y_j$ and the density of $E_{\underline{h}}(B_j,E(I,N))\cap L^p_{\underline{h}}(B_j,L^p_{h_j}(I,N))$ follows from the same arguments as the proof of \cref{prop:density_iterated_almost_simple}. Hence, we get that $\prod_{j\in J} E_{\underline{h}}(B_j,E(I,N))\cap L^p_{\underline{h}}(B_j,L^p_{h_j}(I,N))$ is dense in $\prod_{j\in J} L^p_{\underline{h}}(B_j,L^p_{h_j}(I,N))$ by \cite[Corollary 2.3.5]{engelking1977general}. Finally, by noticing that 
    $E_{\underline{h}}(M,E(I,N)) \cap L^p_{\underline{h}}(M,L^0(I,N))\cong \prod_{j\in J} E_{\underline{h}}(B_j,E(I,N)) \cap L^p_{\underline{h}}(B_j,L^p_{h_j}(I,N))$
    since $J$ is finite (otherwise the gluing mapping may produce mappings that are not almost simple), we get that $E_{\underline{h}}(M,E(I,N))\cap L^p_{\underline{h}}(M,L^0(I,N))$ is a dense subspace of  $(L^p_{\underline{h}}(M,L^0(I,N)),D_{p,p})$.
\end{proof}

\section{Reminders}

\subsection{Length spaces}
\label{sec:metric_geometry}

This section introduces the basic notions of metric geometry on length spaces (see \cite[Chapter 2]{burago2022course} for a complete exposition on that matter). 

In a metric space $X$ with metric $d_X: X^2\to [0,\infty)$, the \emph{length} of an absolutely continuous curve (\cref{def:ac_curves}) is defined as

\begin{definition}[Length of a curve]
\label{def:length_curve}
The \emph{length} of a curve $c\in \mathcal{AC}(I,X)$ induced by $d_X$ is defined as 
$$L_X(c) \coloneqq \int_I \lvert c'\rvert_X(t)\dif t.$$
\end{definition}
Taking the infimum over curves connecting two points happens to induce a metric, potentially infinite-valued, when there is no absolutely continuous curve connecting the two points.
\begin{definition}[Intrinsic metric]
\label{def:intrinsic_metric}
The \emph{intrinsic metric}, $\widehat{d}_X: X^2 \to [0,\infty]$, induced by $d_X$ is defined as
\begin{align*}
\widehat{d}_X(x,x') \coloneqq \inf\{L_X(c): c\in \mathcal{AC}(I,X), c(a) = x, c(b) = x'\}.
\end{align*}
\end{definition}
A natural question arising from this construction is whether the intrinsic metric coincides with the original metric. This is exactly the notion of \emph{length space}.
\begin{definition}[Length spaces]
\label{def:length_spaces}
    $X$ is called a \emph{length space} if and only if $d_X = \widehat{d}_X$.
\end{definition} 

\begin{remark}[On showing that a metric space is a length space] Note that, by the triangle inequality, $d_X \leq \widehat{d}_X$. Hence, showing that a metric space is a length space usually comes down to showing that, for all $(x,x')\in X^2$ and all $\varepsilon > 0$ (resp.~$\kappa > 1$), there exists a curve $c\in \mathcal{AC}(I,X)$ connecting $x$ and $x'$ such that $L_X(c) \leq d_X(x,x') +\varepsilon$ (resp.~$L_X(c) \leq \kappa\, d_X(x,x')$). 
\end{remark}

If the infimum is a minimum, $X$ is called a \emph{geodesic space}. 
\begin{definition}[Geodesic spaces]
\label{def:geodesic_space}
    $X$ is called a \emph{geodesic space} if and only if for all $(x,x')\in X^2$ there exists $c\in \mathcal{AC}(I,X)$ connecting $x$ and $x'$ such that $d_X(x,x') = L_X(c)$.
\end{definition}

Curves minimizing the length between their endpoints are called \emph{minimizing geodesics} of $X$ and can be reparametrized as \emph{constant-speed geodesic}, that is, as curves $c: I \to X$ satisfying for all $(s,t)\in I^2$ that
$$d_X(c(s),c(t)) = \left\lvert \frac{s - t}{b-a}\right\rvert d_X(c(a),c(b)).$$

Note that a constant-speed geodesic is necessarily a minimizing geodesic for its endpoints. Now, let us recall the definition of the \emph{$p$-energy} of a curve.
\begin{definition}[Energy of a curve]
\label{def:energy}
    Let $p\in[1,\infty)$. Then, the \emph{$p$-energy} of a curve $c\in \mathcal{AC}^p(I,X)$ is defined as 
$$E_{p,X}(c) \coloneqq \int_I \lvert c'\rvert_X^p(t)\dif t.$$
The $1$-energy is the length, that is, $L_X(c) = E_{1,X}(c)$. 
\end{definition}

Then, the infimum of the length over $\mathcal{AC}(I,X)$ can be shown to be equal to the infimum over $\mathcal{AC}^p(I,X)$ of the $p$-th root of the $p$-energy for all $p\geq 1$.
\begin{proposition}[Energy minimizing distance]
\label{prop:energy_min_dist}
    Let $p\in [1,\infty)$ and $(x,x')\in X^2$. Then, 
    $$\widehat{d}_X(x,x') = \inf\left\{\lebesgue(I)^{1-1/p}E_{p,X}(c)^{1/p}: c\in \mathcal{AC}^p(I,X),c(a)=x,c(b)=x'\right\}.$$
\end{proposition}
\begin{proof}
    The proof follows the same approach as the proof of (i) in \cite[Theorem~3.7]{trouve1995infinite} or \cite[Proposition~3.8]{pierron2024graded}, which consists in finding a constant-speed approximation for each absolutely continuous curve.
    
    \noindent Let $(x,x') \in X^2$ be such that there exists $c \in \mathcal{AC}(I,X)$ connecting $x$ and $x'$, and $\varepsilon > 0$. Then, define $s_\varepsilon: I\to I$ as $$s_\varepsilon(t) \coloneqq a +\frac{\int_a^t \lvert c'\rvert_X(u) \dif u + (t-a)\varepsilon\lebesgue(I)^{-1}}{\int_I \lvert c'\rvert_X(u)\dif u + \varepsilon}\lebesgue(I).$$
    Its inverse mapping $t_\varepsilon$ is $\lebesgue$-a.e.~differentiable, with derivative given, for $\lebesgue$-a.e.~$s\in I$, by 
    $$\dot{t}_\varepsilon(s) = \frac{\lebesgue(I)^{-1}(\int_I \lvert c'\rvert_X(u)\dif u + \varepsilon)}{\lvert c'\rvert_X(t_\varepsilon(s)) + \varepsilon\lebesgue(I)^{-1}} > 0.$$
    Now, setting $\tilde{c} \coloneqq c \circ t_\varepsilon$, we have, for $\lebesgue$-a.e. $s\in I$,
    \begin{align*}\lvert \tilde{c}'\rvert_X(s) &= \lim\limits_{t\to s}\dfrac{d_X(c\circ t_\varepsilon(s),c\circ t_\varepsilon(t))}{\lvert t_\varepsilon(s) - t_\varepsilon(t)\rvert}\dfrac{\lvert t_\varepsilon(s) - t_\varepsilon(t)\rvert}{\lvert s - t\rvert } \\
    &= \lvert c'\rvert_X (t_\varepsilon(s)) \,\dot{t}_\varepsilon(s)
    \leq \lebesgue(I)^{-1}\left(\int_I\lvert c'\rvert_X(u) \dif u + \varepsilon\right).
    \end{align*}
    Therefore, for all $p\in [1,\infty)$, $\tilde{c} \in \mathcal{AC}^\infty(I,X) \subset \mathcal{AC}^p(I,X)$ and, by Hölder's inequality,
    $$L_X(c) = L_X(\tilde{c}) \leq \lebesgue(I)^{1-1/p}E_{p,X}(\tilde{c})^{1/p} \leq L_X(c) + \varepsilon.$$
    Since $\varepsilon$ is arbitrary, this yields the conclusion of \cref{prop:energy_min_dist}.
\end{proof}
\begin{remark}[On showing that a metric space is a length space using the energy]
    Thus, showing that $X$ is a length space also comes down to showing that, for all $(x,x')\in X^2$ and all $\varepsilon >0$ (resp.~$\kappa > 1$), there exists a curve $c\in \mathcal{AC}^p(I,X)$ connecting $x$ and $x'$ such that $\lebesgue(I)^{p-1}E_{X,p}(c) \leq d_X(x,x')^p +\varepsilon$ (resp.~$\lebesgue(I)^{p-1}E_{X,p}(c) \leq \kappa\,d_X(x,x')^p$).
\end{remark}

An important consequence of this result is that, when $X$ is a geodesic space, a curve minimizes the $p$-energy, $p\in [1,\infty)$,  if and only if it is a constant-speed geodesic.

\subsection{Measurable multifunctions}

We recall below the statement of Aumann's measurable selection theorem \cite[III-\textsection 4-28-2)]{castaing1977measurable} (see \cite[Corollary~18.27]{aliprantis2006infinite} for a more recent exposition).

\begin{theorem}[Aumann's measurable selection theorem]
\label{th:aumman_selection_theorem}
Let $(O,\Sigma_O)$ be a measurable space, $(X,d_X)$ be a complete and separable metric space. Let $\Gamma$ be a multifunction from $O$ to non-empty subsets of $X$, whose graph $G$ belongs to $\Sigma_O\otimes \mathcal{B}(X)$. Then, there exists a $(\hat{\Sigma}_O, \mathcal{B}(X))$-measurable mapping $S$ such that $S(o)\in \Gamma(o)$ for all $o\in O$, where $\hat{\Sigma}_O$ denotes the intersection of all completions of $\Sigma_O$ with respect to each finite measure on $(O,\Sigma_O)$. Such a mapping $S$ is called $(\hat{\Sigma}_O, \mathcal{B}(X))$-measurable selection.
\end{theorem}

\subsection{Differentiable manifolds}

\renewcommand{\thedefinition}{\Alph{section}.\arabic{definition}} 
\renewcommand{\thetheorem}{\Alph{section}.\arabic{theorem}} 

This section briefly recalls the construction of differentiable manifolds modeled on Banach spaces (see \cite[Chapter~II]{lang2012fundamentals} for a complete exposition on that matter). 

Let us first recall a definition of \emph{atlases} by paraphrasing \cite[Chapter~II,\S 1]{lang2012fundamentals}.

\begin{definition}[Atlases]
\label{def:atlases}
Let $r\in \mathbb{N}\cup \{\infty\}$. An \emph{atlas of class $\mathcal{C}^r$} on a set $X$ is a family of pairs $\{(U_\alpha,\varphi_\alpha)\}_{\alpha\in A}$ indexed by a set $A$ satisfying the following conditions:
\begin{enumerate}[label=(\roman*)]
    \item each $U_\alpha$ is a subset of $X$ and the $U_\alpha$'s cover $X$.
    \item each $\varphi_\alpha$ is a bijection from $U_\alpha$ onto an open subset $\varphi_\alpha (U_\alpha)$ of some Banach space $\mathbb{B}_\alpha$ and for all $(\alpha,\alpha')\in A^2$ the set $\varphi_\alpha(U_\alpha \cap U_{\alpha'}) $ is open in $\mathbb{B}_\alpha$.
    \item the mapping $\varphi_\alpha \circ \varphi_{\alpha'}^{-1}: \varphi_{\alpha'}(U_\alpha\cap U_{\alpha'}) \to \varphi_\alpha(U_\alpha\cap U_{\alpha'})$ is a $\mathcal{C}^r$ isomorphism for each pair of indices $(\alpha,\alpha')\in A^2$.
\end{enumerate}
Two atlases $\{(V_\beta, \psi_\beta)\}_{\beta\in B}$ and $\{(U_\alpha,\varphi_\alpha)\}_{\alpha\in A}$ on $X$ indexed by some sets $A$ and $B$ are said to be \emph{compatible} if each pair $(V_\beta, \psi_\beta)$ is such that the mapping $\varphi_\alpha \circ \psi_\beta^{-1}: \psi_\beta(U_\alpha\cap V_\beta)\to \varphi_\alpha(U_\alpha\cap V_\beta)$ is a $\mathcal{C}^r$ isomorphism for all pairs $(U_\alpha,\varphi_\alpha)$ and conversely. This relation defines an equivalence relation between atlases on $X$.
\end{definition}

\begin{remark}[On topologies induced by atlases]
    Note that an atlas uniquely defines a topology in which the charts $\{(U_\alpha,\varphi_\alpha)\}_{\alpha\in A}$ are such that the $U_\alpha$'s are open sets and the $\varphi_\alpha$'s are topological isomorphisms. This induced topology is not required to be Hausdorff.
\end{remark}

 Using the equivalence relation of compatibility between atlases, \emph{Banach manifolds} are then defined as follows:

\begin{definition}[Banach manifolds]
\label{def:banach_manifolds}
    Let $r\in \mathbb{N}\cup \{\infty\}$. A \emph{$\mathcal{C}^r$ Banach manifold} is a set $X$ together with an equivalence class of atlases of class $\mathcal{C}^r$ (\cref{def:atlases}) on this set.
\end{definition}

\subsection{Partitions of unity}

On topological spaces, one is often interested in using \emph{partitions of unity} to define mappings locally and extend them to the whole space. We thus recall a definition inspired by the exposition in \cite[Chapter~5, p.~300]{engelking1977general}.

\begin{definition}[Partitions of unity]
\label{def:partition_unity}
    Suppose $N$ is a topological space. Then,
    \begin{enumerate}[label=(\roman*)]
        \item a family $\{\rho_s\}_{s\in S}$, indexed by a set $S$, of continuous mappings from $N$ to $[0,1]$ is called a \emph{partition of unity} if $\sum_{s\in S} \rho_s(y) = 1$ for all $y\in N$ and denote by $\operatorname{carr}(\rho_s)\coloneqq \{x\in N: \rho_s(x) > 0\}$ the \emph{carrier} of $\rho_s$.
        \item a partition of unity $\{\rho_s\}_{s\in S}$ is called \emph{locally finite} if for all $y\in N$ there exist a neighborhood $U_y\subset N$ and  a finite set $S_y \subset S$ such that for all $y'\in U_y$ we have $\rho_s(y')=0$ for all $s\in S\setminus S_y$ and $\sum_{s\in S_y} \rho_s(y')=1$. This is equivalent to saying that the cover $\{\operatorname{carr}(\rho_s)\}_{s\in S}$ of $N$ is \emph{locally finite}, that is, for all $y\in N$ there exists a neighborhood $U_y\subset N$ such that $\{s\in S:\operatorname{carr}(\rho_s)\cap U_y \neq \emptyset\}$ is finite.
        \item a partition of unity $\{\rho_s\}_{s\in S}$ is called \emph{subordinated} to an open cover $\{U_\alpha\}_{\alpha\in A}$ of $N$ indexed by some set $A$ if $\{\operatorname{carr}(\rho_s)\}_{s\in S}$ is a \emph{refinement} of $\{U_\alpha\}_{\alpha\in A}$, that is, for all $s \in S$ there exists $\alpha \in A$ such that $\operatorname{carr}(\rho_s)\subset U_\alpha$.
    \end{enumerate}
\end{definition}

\subsection{Regularity of measures}

We recall below the meaning of the \emph{regularity} of the measure $\mu_M$. The following definition is based on \cite[342A Definitions (a)]{fremlin2000measure} and recalled from \cite[Definition 6.3]{serieys2025nonlinear}.

\begin{definition}[Regularity of measures]
\label{def:regularity_measures}
Let $(\mathcal{F},\mathcal{A})\in ({2^M})^2$. A measure $\mu_M$ on $(M,\Sigma_M)$ is called:
 \begin{enumerate}[label=(\roman*)]
     \item \emph{inner regular} on $\mathcal{F}$ with respect to $\mathcal{A}$ when each $B\in\mathcal{F}\cap \Sigma_M$ satisfies 
     $\mu_M(B) = \sup \left\{ \mu_M(A): A\in \mathcal{A}\cap \Sigma_M, A\subset B\right\}$, that is, the $\mu_M$-measure of $B$ can be approximated from within by the $\mu_M$-measure of sets in $\mathcal{A}$.
     \item \emph{outer regular} on $\mathcal{F}$ with respect to $\mathcal{A}$ when each $B\in\mathcal{F}\cap \Sigma_M$ satisfies 
     $\mu_M(B) = \inf \left\{ \mu_M(A): A\in \mathcal{A}\cap \Sigma_M, B\subset A\right\}$, that is, the $\mu_M$-measure of $B$ can be approximated from outside by the $\mu_M$-measure of sets in $\mathcal{A}$.

\end{enumerate}
\end{definition}

\bibliography{bibfile}
\bibliographystyle{alpha}

\end{document}